\documentclass[a4paper,reqno]{amsart}

\usepackage{amsfonts}
\usepackage{mathrsfs}
\usepackage{amsmath}
\usepackage{amsmath,graphics}
\usepackage{amssymb}
\usepackage{indentfirst,latexsym,bm,amsmath,pstricks,amssymb,amsthm,graphicx,fancyhdr,float,color}
\usepackage[all]{xy}
\usepackage{amsthm}
\usepackage{amscd}
\usepackage{hyperref}
\usepackage[all]{hypcap}

\newtheorem{theorem}{Theorem}[section]
\newtheorem{thm}{Theorem}
\newtheorem{cor}[thm]{Corollary}
\newtheorem{lemma}[theorem]{Lemma}
\newtheorem{conjecture}[theorem]{Conjecture}

\newtheorem{corollary}[theorem]{Corollary}
\newtheorem{proposition}[theorem]{Proposition}
\newtheorem{claim}[theorem]{Claim}

\theoremstyle{definition}
\newtheorem{definition}[theorem]{Definition}
\newtheorem{example}[theorem]{Example}

\theoremstyle{remark}
\newtheorem{remark}[theorem]{Remark}
\newtheorem{remarks}[theorem]{Remarks}

\numberwithin{equation}{section}
\renewcommand{\theequation}{\arabic{section}-\arabic{equation}}

\def\wt{\widetilde}
\def\ol{\overline}
\def\lra{\longrightarrow}

\def\({$($}
\def\){$)$}
\def\chit{\chi_{\rm top}}
\def\bbp{\mathbb P}
\def\cala{\mathcal A}

\def\calf{\mathcal F}
\def\calh{\mathcal H}

\def\calm{\mathcal M}

\def\calo{\mathcal O}
\def\calr{\mathcal R}
\def\cals{\mathcal S}
\def\calt{\mathcal T}
\def\calth{\mathcal T \mathcal H}
\def\calx{\mathcal X}
\def\Pic{\text{{\rm Pic\,}}}
\def\rank{\text{{\rm rank\,}}}
\def\v{\vspace{0.15cm}}

\newcommand{\adele}{{{\mathbb{A}}_f}}

\newcommand{\Cores}{\mathrm{Cor}}
\newcommand{\Hbb}{\mathbb{H}}
\newcommand{\Qbb}{\mathbb{Q}}
\newcommand{\Rbb}{\mathbb{R}}
\newcommand{\Zbb}{\mathbb{Z}}
\newcommand{\Zbhat}{\hat{\Zbb}}
\newcommand{\Mat}{\mathrm{Mat}}
\newcommand{\Mcal}{\mathcal{M}}
\newcommand{\Acal}{\mathcal{A}}
\newcommand{\Hscr}{\mathscr{H}}
\newcommand{\Cbb}{\mathbb{C}}
\newcommand{\la}{\leftarrow}
\newcommand{\ra}{\rightarrow}
\newcommand{\mono}{\hookrightarrow}
\newcommand{\Tcal}{\mathcal{T}}
\newcommand{\GSp}{\mathrm{GSp}}
\newcommand{\Sp}{\mathrm{Sp}}
\newcommand{\bsh}{\backslash}
\newcommand{\isom}{\simeq}

\newcommand{\Gbb}{\mathbb{G}}
\newcommand{\mrm}{\mathrm{m}}
\newcommand{\Gbf}{\mathbf{G}}
\newcommand{\Hbf}{\mathbf{H}}
\newcommand{\Xcal}{\mathcal{X}}
\newcommand{\Sbb}{\mathbb{S}}
\newcommand{\Res}{\mathrm{Res}}
\newcommand{\ad}{\mathrm{ad}}
\newcommand{\Ad}{\mathrm{Ad}}
\newcommand{\GL}{\mathrm{GL}}
\newcommand{\gfrak}{\mathfrak{g}}
\newcommand{\Lie}{\mathrm{Lie}}
\newcommand{\SL}{\mathrm{SL}}

\newcommand{\ot}{\overset}
\newcommand{\der}{\mathrm{der}}
\newcommand{\inv}{{-1}}
\newcommand{\ubf}{\mathbf{u}}
\newcommand{\Ker}{\mathrm{Ker}}
\newcommand{\Vbb}{\mathbb{V}}

\newcommand{\End}{\mathrm{End}}
\newcommand{\SU}{\mathrm{SU}}

\newcommand{\xbar}{\bar{x}}
\newcommand{\Spec}{\mathrm{Spec}\,}

\begin{document}

\title[Oort conjecture on Shimura curves in Torelli locus]
{The Oort conjecture on Shimura curves in the Torelli locus of curves}

\author{Xin Lu}
\address{Institut f\"ur Mathematik, Universit\"at Mainz, Mainz, Germany, 55099}

\email{lvxinwillv@gmail.com}
\thanks{This work is supported by SFB/Transregio 45 Periods, Moduli Spaces and Arithmetic of Algebraic Varieties of the DFG (Deutsche Forschungsgemeinschaft),
and partially supported by National Key Basic Research Program of China (Grant No. 2013CB834202) and NSFC}

\author{Kang Zuo}
\address{Institut f\"ur Mathematik, Universit\"at Mainz, Mainz, Germany, 55099}
\email{zuok@uni-mainz.de}


\date{August 18, 2014}



\maketitle

\begin{abstract}
Oort has conjectured that there do not exist Shimura curves
contained generically in the Torelli locus of genus-$g$ curves when
$g$ is large enough. In this paper we prove the Oort conjecture for
Shimura curves of Mumford type and Shimura curves parameterizing
principally polarized $g$-dimensional abelian varieties
isogenous to $g$-fold self-products of elliptic curves for $g>11$. We also prove that
there do not exist Shimura curves contained generically in the
Torelli locus of hyperelliptic curves of genus $g>7$. As a
consequence, we obtain a finiteness result regarding smooth genus-$g$ curves with
completely decomposable Jacobians, which is related to a question of Ekedahl and Serre.
\end{abstract}

\tableofcontents

\section{Introduction}\label{sectionintroduction}
This paper is devoted to the study of the conjecture of Oort on special subvarieties of the Siegel modular variety
that are contained in the Torelli locus.
In this section we state the main results and explain the basic idea of the proofs.

\subsection{Oort's conjecture}

We start with the conjectures of Coleman and Oort. A more thorough survey of the subject
is found in the beautiful paper \cite{mo11}.

Fix $n\geq 3$ an integer, we have  $\calm_g=\calm_{g,[n]}$ the moduli space of smooth projective curves over
complex number $\mathbb C$ of genus $g\geq 2$ with a full level $n$-structure,
and  $\cala_g=\cala_{g,[n]}$  the moduli space of $g$-dimensional
principally polarized  abelian varieties over $\mathbb C$ with full level-$n$ structure.
In this paper we treat them as the moduli schemes over $\mathbb C$ of the corresponding moduli functors.
No specific choice of the level $n(\geq3)$ is made because it is only imposed to assure the representability,
which plays no essential role in our study.

Recall that the Torelli morphism $$j^\circ:~\calm_g\lra \cala_g$$ associates to a curve
its Jacobian with its canonical principal polarization and level structure.
The image of $j^\circ$, denoted as $\calt_g^\circ$, is a locally closed subvariety in $\cala_g$,
whose closure is denoted as $\calt_g$.
$\calt_g$ is called the Torelli locus (in $\cala_g$) and $\calt_g^\circ$ is referred as the open Torelli locus.
We also have the Torelli locus $\calth_g\subseteq \calt_g$ of hyperelliptic curves corresponding to Jacobians of hyperelliptic curves.

A closed subvariety $Z\subseteq\cala_g$ of positive dimension is said to be {\it contained generically}
in the Torelli locus
(resp. the Torelli locus of hyperelliptic curves), written as $Z\Subset \calt_g$ (resp. $Z\Subset \calth_g$),
if $Z\subseteq\calt_g$ and $Z\cap\calt_g^\circ\neq\emptyset$
(resp. $Z\subseteq\calth_g$ and $Z\cap\calt_g^\circ\neq\emptyset$).

As is explained in \cite{hida shimura,milne05,moonen98},
the moduli scheme $\Acal_g=\Acal_{g,[n]}$ is isomorphic to a connected Shimura variety,
namely a geometrically connected component of the Shimura variety defined by the Shimura datum $(\GSp_{2g},\Hscr_g^\pm)$
associated to the group of symplectic similitude $\GSp_{2g}$, using some compact open subgroup
$K(n)\subset \GSp_{2g}(\adele)$, cf. Example \ref{Siegel modular variety}.
In $\cala_{g}$ there are special subvarieties and totally geodesic subvarieties,
the definitions of which are given later in Section \ref{sectiondefspecial}.
Special subvarieties are totally geodesic subvarieties containing CM points, cf. \cite{moonen98}.
They are of particular interest because on the one hand they are locally symmetric in the sense of differential geometry,
and on the other hand they parameterize abelian varieties with prescribed Hodge classes, cf.\cite{deligne milne}.
Special subvarieties of dimension zero are exactly CM points,
i.e. points in $\cala_{g}$ that parameterize abelian varieties with complex multiplication.

It was conjectured by Coleman \cite{coleman87}, that when the genus $g$ is sufficiently large,
there should be at most finitely many CM points on $\cala_g$ contained in the open Torelli locus $\calt_g^\circ$.
Oort \cite{oort97} made the following conjecture by combining Coleman's idea with the conjecture of Andr\'e-Oort:

\begin{conjecture}[Oort]\label{conjOort}
For $g$ large, there does not exist a special subvariety of positive dimension contained
generically in the Torelli locus $\calt_g$.
\end{conjecture}

The Andr\'e-Oort conjecture predicts that  in a Shimura variety,
a closed geometrically irreducible subvariety is special if and only if it contains a Zariski dense subset of CM points.
It is thus immediate that the formulation of Oort is equivalent to the one by Coleman modulo Andr\'e-Oort.
The readers are referred to \cite{noot06}, \cite{scanlon bourbaki}, \cite{yafaev bordeaux}, etc.
for surveys on the recent progress towards the Andr\'e-Oort conjecture.

\subsection{Progress on the Oort conjecture}
Although Coleman made his conjecture for $g\geq 4$, counterexamples have been found for $4\leq g\leq 7$.
However, if one aims at the non-existence of special subvarieties of a certain type with $g$ sufficiently large,
then much more evidence is available, as the concrete ``type'' of the special subvarieties often imposes constraints
so that no universal family of curves could be produced from such special subvarieties subject to the Torelli morphism
(cf. \cite{chai-oort,mo11}).

In \cite{hain99}, based on the properties of the mapping class groups (cf. \cite{farbmasur}),
Hain proved, under some natural technical assumptions,
that a special subvariety $Z$ of $\cala_g$ should be either a ball quotient or
that its intersection with the Torelli locus of hyperelliptic curves $\calth_g$ should be a divisor in $Z$.
This suggests that the ball quotients should play a special role in the study of the conjecture.
Building on Hain's method, de Jong and Zhang \cite{djongzh07} have proved that
$Z$ cannot be a Hecke translate of a Hilbert modular subvariety if $g\geq 5$,
which has also been proved for $g=4$ by Bainbridge and M\"oller \cite{bainbridgemoeller14} using degeneration techniques
and independent of properties of the mapping class groups.

In \cite{djongnoot89},
de Jong and Noot have proposed an approach based on an $p$-adic obstruction constructed by  Dwork-Ogus \cite{do86}
and they proved that  base varieties of  some specific
universal families  of curves arising from cyclic covers of $\mathbb P^1$ are not contained generically in $\calt_g$.
Recently,  Moonen \cite{moonen10} has extended  de Jong and Noot's result and proved that there are exactly twenty
families of curves coming from cyclic covers of $\mathbb P^1$ such that the base varieties
are contained generically in $\calt_g$, which implies that Conjecture \ref{conjOort} holds if
the corresponding special subvarieties arise from a universal cyclic cover of $\bbp^1$.

In \cite{kukulies10},  Kukulies proved  Oort's conjecture  for rational Shimura curves parameterizing principally
polarized abelian varieties isogenous to $g$-fold self-product of elliptic curves for $g\gg 0$.
His approach combines ideas of M\"oller, Viehweg and the second named author on the characterization of Shimura curves
and the Sato-Tate conjecture for modular curves, which is of arithmetic nature.

There has been also other progress on the conjecture, cf.
\cite{cfg13,frghpe14,grumoeller,mohajerzuo,moller11,peters14,rohde}
and further discussions in \cite{mo11}.

\subsection{Main results}\label{sectionmainresults}
In this paper we focus on the conjecture of Oort for Shimura curves of the following types::
\begin{enumerate}
\item[I.] Shimura curves
parameterizing principally polarized $g$-dimensional abelian varieties
that are isogenous to a $g$-fold self-product of some elliptic curve;
\item[II.] Shimura curves of Mumford type;
\item[III.] Shimura curves in the Torelli locus of hyperelliptic curves.
\end{enumerate}

In this paper Shimura curves are special subvarieties of dimension one, cf. Remark \ref{terminology}.
The precise definition of the Shimura curves of type I and II are given in Section \ref{sectiontwotype}.
While  Shimura curves  of type I are  of PEL type,
Shimura curves of  Mumford type
constructed in \cite{vz04} by generalizing  Mumford's original example in \cite{mumford69},
are not of PEL-type  except precisely two classes (see Section \ref{sectiontwotype}).

Our first main result is the following:
\begin{thm}\label{mainthm1}
For $g>11$, there does not exist a Shimura curve of  types I or II  contained generically in $\mathcal T_g$.
\end{thm}

This answers a question of Moonen and Oort \cite[Question\,6.7]{mo11} in the one-dimensional case.
The result by Kukulies \cite{kukulies10} also studies Shimura curves of type I, but it is more restrictive:
it only deals with  rational Shimura curves,
and the bound for $g$ depends on the curve in question,
which is not uniform. Our approach is of differential-geometric nature.
It gives an explicit lower bound, without resorting to deep number-theoretic results like the Sato-Tate conjecture.

Note that Theorem \ref{mainthm1} also implies a partial answer to a question raised by Ekedahl and Serre \cite{eserre},
where they asked for the existence of a smooth curve of higher genus whose Jacobian is completely decomposable,
i.e., isogenous to a $g$-fold product of elliptic curves. In fact we obtain the following corollary of finiteness:

\begin{cor}\label{corE-S}
For each fixed integer $g>11$, there exist, up to isomorphism, at most
finitely many smooth projective curves of genus $g$
whose Jacobians are isogenous to $g$-fold self-product of a single elliptic curve with bounded isogenous degrees.
\end{cor}

In \cite{djongzh07} De Jong and Zhang have shown that Hecke translates of Hilbert modular varieties associated to
totally real fields of degree $g$ over $\Qbb$ are not contained generically in the Torelli locus $\Tcal_g$ for $g\geq 5$
(which also holds for $g=4$ proved by Bainbridge-M\"oller \cite{bainbridgemoeller14}).
Using the non-existence of Shimura curves of type I
in the Torelli locus, we obtain an alternative approach to their result, under a slightly varied formulation with $g>11$:

\begin{cor}[real multiplication]\label{real multiplication}
Let $F$ be a totally real \'etale $\Qbb$-algebra of dimension $g$ over $\Qbb$, and
let $Z\subseteq\Acal_g$ be a Hecke translate of the generalized Hilbert modular subvariety defined by $(\Gbf^F,X^F)$
(cf. Definition {\rm \ref{Hilbert modular variety}}). Then $Z$ is not contained generically in $\Tcal_g$ for $g>11$.
\end{cor}

Here by totally real \'etale $\Qbb$-algebra we mean a product of finitely many totally real number fields,
and the generalized Hilbert modular subvariety they define are,
up to finite covering, products of usual Hilbert modular varieties
(involving several totally real fields). Such Shimura data contain the subdatum $(\GL_2,\Hscr_1^\pm)$
induced by the natural embedding $\Qbb\mono F$, hence the special subvariety $Z$ they define contain
a Shimura curve of type I.
Note that the idea of \cite{djongzh07} goes back to \cite{hain99} which relies on properties of the mapping class
groups studied in \cite{farbmasur}. It only works for the usual Hilbert modular varieties associated to a totally
real field of dimension $g\geq 5$.
It does not cover the case of the $g$-fold product of modular curves embedded in $\cala_g$,
because in this case the lattice involved is essentially reducible with low rank factors,
i.e. commensurable with $SL_2(\mathbb Z)^g$.
Our approach focuses on modular curves diagonally embedded in these generalized Hilbert modular varieties,
which has no restriction on the dimension of the factor fields of the totally real \'etale algebra,
and the proof does not involve the mapping class groups.

The following is a weaker result regarding Oort's conjecture
for Shimura curves of types I and II, already proved in our previous
preprint \cite[Theorem\,1.2]{luzuo13}.
\begin{thm}\label{mainthm2}
For $g>4$, there does not exist any one-dimensional family of semi-stable curves of genus $g$
with strictly maximal Higgs field.
\end{thm}

Higgs bundles and Higgs fields  on curves will be among the main tools in our proofs.
They are briefly recalled in the next subsection, and more details will be given in
Sections \ref{sectionHiggsonC} and \ref{sectionrepresent}.
For a semi-stable family of curves with strictly maximal Higgs field,
its image under the Torelli morphism is already a Shimura curve in $\Acal_g$ due to \cite{vz04},
and hence Theorem \ref{mainthm1} implies Theorem \ref{mainthm2} for $g>11$.
On the other hand, the converse is not true, namely not all Shimura curves of type I or II
contained generically in $\Tcal_g$ arise this way; see Section \ref{sectionrepresent} for further discussions,
and especially Example \ref{exshimurag=3strict} for $g=3$.

Our last main result focuses on Shimura curves in the Torelli locus of hyperelliptic curves,
which holds more generally for totally geodesic curves in $\Acal_g$:
\begin{thm}\label{mainthm3}
For $g>7$, there does not exist totally geodesic curves contained generically
in the Torelli locus of hyperelliptic curves $\calth_g$.
\end{thm}

\subsection{The main idea of proofs}\label{sectionmainidea}
In this subsection we explain the main idea of the proofs.

As is mentioned earlier, we intend to exclude the existence of Shimura curves of certain type in the open Torelli locus
using the natural constraints from the geometry of special subvarieties.
For the three types of Shimura curves studied in this paper, the constraints mainly come from two types of inequalities:
\begin{enumerate}
\item[(i).]
Arakelov inequalities of Higgs bundles for curves in $\cala_g$,
which are actually equalities when applied to Shimura curves;

\item[(ii).] strict Arakelov inequalities of Higgs bundles for curves
contained generically in the Torelli locus.
\end{enumerate}

The Arakelov (in)equalities of Higgs bundles play a crucial role in the works \cite{mvz12,vz04}
characterizing the geometry of totally geodesic subvarieties in Shimura varieties.
The survey \cite{viehweg09} by Viehweg is highly recommended.

In our case,
the universal family of abelian varieties $\mathfrak h:\,\Xcal_g \ra \Acal_g$
gives rise to the Higgs bundle $(E,\theta)$ on $\Acal_g$,
where $E=E^{1,0}\oplus E^{0,1}$ is the graded quotient of the Hodge filtration on the
$\Qbb$-VHS whose underlying local system is $\Vbb=R^1\mathfrak h_*\Qbb_{\Xcal_g}$,
and $\theta$ is induced by the Gauss-Manin connection.
For a smooth closed curve $\phi:\,C\hookrightarrow \Acal_g$ with a suitable smooth compactification $\ol C$
by joining a finite set of cusps $\Delta_{\ol C}$,
by pulling-back along $\phi$ one has the universal family of abelian varieties $h:\,X\to C$ and a
local system $\mathbb V_C:=Rh_*\Qbb_{X}=\phi^* \mathbb V$,
which underlies a $\Qbb$-VHS whose associated Higgs bundle on $C$ extends to a logarithmic Higgs bundle
$(E_{\ol C},\,\theta_{\ol C})$ on $\ol C$.
$(E_{\ol C},\theta_{\ol C})$ decomposes further into a direct sum of Higgs bundles
(cf. \cite{fujita78} or \cite{kollar87}):
\begin{equation}\label{decompC}
\left(E_{\ol C}^{1,0}\oplus E_{\ol C}^{0,1},~\theta_{\ol C}\right)
=\left(A_{\ol C}^{1,0}\oplus A_{\ol C}^{0,1},~\theta_{\ol C}\big|_{A_{\ol C}^{1,0}}\right)\oplus
\left(F_{\ol C}^{1,0}\oplus F_{\ol C}^{0,1},~ 0\right),
\end{equation}
where $A_{\ol C}^{1,0}$ is an ample vector bundle, while $F_{\ol C}^{1,0}$
and $F_{\ol C}^{0,1}$ are flat vector bundles associated to unitary
local subsystems  $\mathbb F_{C}^{1,0}\oplus \mathbb F_{C}^{0,1}\subseteq \mathbb V_{C}.$

\begin{definition}[\cite{vz03}]\label{defmaximalC}
The Higgs bundle $\left(E_{\ol C}^{1,0}\oplus E_{\ol C}^{0,1},~\theta_{\ol C}\right)$
is said to be  with {\it maximal} Higgs field if
$$\theta_{\ol C}\big|_{A_{\ol C}^{1,0}}:\,A_{\ol C}^{1,0} \lra A_{\ol C}^{0,1}\otimes \Omega^1_{\ol C}(\log \Delta_{\ol C})$$
is an isomorphism, and to be with  {\it strictly maximal}  Higgs field if furthermore $F_{\ol C}^{1,0}\oplus F_{\ol C}^{0,1}=0$.
By  \cite{vz03} the Higgs field $\theta_{\ol C}$ is strictly maximal (resp. maximal)  if and only if
the following {\it Arakelov equality} holds
\begin{equation}\label{arakelovC}
\begin{aligned}
\deg E^{1,0}_{\ol C}&= {g\over 2}\cdot\deg\Omega^1_{\ol C}(\log\Delta_{\ol C}),\\
\Big(\text{resp.}\quad \deg E^{1,0}_{\ol C}&= {\rank A_{\ol C}^{1,0}\over 2}\cdot\deg\Omega^1_{\ol C}(\log\Delta_{\ol C})~\Big).
\end{aligned}
\end{equation}
\end{definition}

The following theorem gives numerical characterizations of Shimura curves,
and more generally of totally geodesic curves:
\begin{theorem}[\cite{moller11,vz04}]\label{theoremnumchar}
Let $C \subseteq \cala_g$ be a smooth closed curve,
and $\left(E_{\ol C}^{1,0}\oplus E_{\ol C}^{0,1},~\theta_{\ol C}\right)$ the associated logarithmic Higgs bundle. Then

{\rm(i).} $C$ is a Shimura curve of type I or II if and only if
the associated logarithmic Higgs bundle $\left(E_{\ol C}^{1,0}\oplus E_{\ol C}^{0,1},~\theta_{\ol C}\right)$
has the strictly maximal Higgs field.

{\rm (ii).} $C$ is a totally geodesic curve
if and only if the associated logarithmic Higgs bundle has maximal Higgs field.
\end{theorem}

Now Theorems \ref{mainthm1} and \ref{mainthm3} are   immediate consequences of the  theorem above
and the following strict Arakelov inequalities for a smooth closed curve $C \Subset \calt_g \subseteq \cala_g$.
\begin{theorem}\label{theoremstrictarak}
Let $C \subseteq \cala_g$ be a smooth closed curve,
and $\left(E_{\ol C}^{1,0}\oplus E_{\ol C}^{0,1},~\theta_{\ol C}\right)$ be the associated logarithmic Higgs bundle.

{\rm (i).} If $C \Subset \calt_g$ with $g>11$, then
\begin{equation}\label{eqnstrictarak1}
\deg E^{1,0}_{\ol C}< {g\over 2}\cdot\deg\Omega^1_{\ol C}(\log\Delta_{\ol C}).
\end{equation}

{\rm (ii).} If $C \Subset \calth_g$ with  $g>7$, then
\begin{equation}\label{eqnstrictarak2}
\deg E^{1,0}_{\ol C}< {\rank A_{\ol C}^{1,0}\over 2}\cdot\deg\Omega^1_{\ol C}(\log\Delta_{\ol C}).
\end{equation}
\end{theorem}

Details of the proof of the strict Arakelov inequalities are given in Section \ref{sectionpfofArak}.
Roughly speaking, we have the following diagram in which the central square is Cartesian:
$$\xymatrix@M=0.15cm{
B\ar[rr]^-{normalization}&&\big(j^\circ\big)^{-1}(C)\ar[rr]\ar@{^(->}[d] && C \ar@{^(->}[d]^{\Subset}& \\
&&\Mcal_g\ar[rr]^{j^\circ}  && \calt_g \ar@{^(->}[r]&  \Acal_g
}$$
By pulling back the universal family over $\calm_g$ to $B$,
one obtains a family $f:\,S \to B$ of curves, which can be naturally extended to a
family $\bar f:\, \ol S \to \ol B$ of semi-stable curves over the smooth compactification
$\ol B \supseteq B$.
Then Theorem \ref{theoremstrictarak} will be a combination of two types of inequalities
(Miyaoka-Yau type inequality and sharp slope inequality) for a family of semi-stable curves,
the proofs of which rely highly on the geometrical property of fibred surfaces.

\subsection{Further perspectives}
The method of the paper was stimulated by the numerical characterization of Shimura curves using the Arakelov equality
given by M\"oller, Viehweg and the second named author \cite{moller11,vz04}.
In \cite{mvz12}, they generalized it to high dimension, and obtained
the numerical characterization of special subvarieties of arbitrary dimension in $\cala_g$,
where the Arakelov equality plays a key role.
It suggests a new way to prove Oort's conjecture by
proving a strict Arakelov inequality for subvarieties of $\cala_g$ which are contained generically
in $\calt_g$ for $g\gg0$.

In \cite{hain99}, Hain has dealt with those special subvarieties in $\calt_g$ containing no
divisor which is also special.
It does not treat special subvarieties associated to groups like $SO(n,2)$ and $SU(n,1)$,
which are of real rank one.
However, in such special subvarieties, one can construct special curves
following \cite{kudladuke,kudlarapoport}.
This is the subject of our ongoing research,
where we intend to exclude such special subvarieties by studying
the strict Arakelov inequalities for special curves in them, cf. \cite{chenlutanzuo}.

\vspace{0.15cm}
The paper is organized as follows. In Section \ref{sectionprelim} we collect facts about
special subvarieties in Shimura varieties, logarithmic Higgs bundles,
and surfaces fibred over curves with semi-stable fibers.
In Section \ref{sectionrepresent} we construct a family $\bar f:\,\ol S \to \ol B$ of semi-stable curves
representing a smooth closed curve $C\Subset\Tcal_g$,
and we explain the relation between the associated logarithmic Higgs bundles over $\ol B$ and over $\ol C$.
In Section \ref{sectionconclusion},
We prove the strict Arakelov inequalities and the main results, based on two types of inequalities
(Miyaoka-Yau type inequality and sharp slope inequality) for a family of semi-stable curves,
whose proofs are given in Sections \ref{sectionupper} and \ref{sectionlower} respectively.
An intermediate result needed in Section \ref{sectionconclusion} is proved in Section \ref{sectionflathyper},
which focuses on the flat part of the Higgs bundle associated to a family of semi-stable hyperelliptic curves.
Finally in Section \ref{sectionexample}, we present examples of Shimura curves
contained generically in the Torelli locus.

\section{Preliminaries}\label{sectionprelim}
\subsection{Special subvarieties in Shimura varieties}\label{sectiondefspecial}
We first recall the general notions of Shimura data and Shimura varieties,
following \cite{deligne77} and \cite{milne05}.

\begin{definition}[Shimura data and Shimura varieties]\label{Shimura data and Shimura varieties}Write $\Sbb=\Res_{\Cbb/\Rbb}\Gbb_\mrm$ for the Deligne torus.

(1) A Shimura datum is a pair $(\Gbf,X)$ where

\begin{enumerate}
\item[SD1.] $\Gbf$ is a connected reductive $\Qbb$-group,
            such that $\Gbf^\ad$ has no compact factors defined over $\Qbb$;
\item[SD2.] $X$ is a $\Gbf(\Rbb)$-conjugacy class of homomorphisms of $\Rbb$-groups $x:\Sbb\ra \Gbf_\Rbb$,
            such that \begin{itemize}\item the composition $\Ad\circ x:\Sbb\ra \Gbf_\Rbb\ra \GL_\Rbb(\gfrak_\Rbb)$
            defines a pure Hodge structure of type $\{(-1,1),(0,0),(1,-1)\}$ on the Lie algebra $\gfrak=\Lie\,\Gbf$;

\item the conjugation by $x(\sqrt{-1})$ induces a Cartan involution on $\Gbf^\ad(\Rbb)$.
\end{itemize}
\end{enumerate}

It turns out that each connected component of $X$ is an Hermitian symmetric domain.
We write $\Gbf(\Rbb)_+$ for the stabilizer in $\Gbf(\Rbb)$ of any such connected component,
and we put $\Gbf(\Qbb)_+=\Gbf(\Rbb)_+\cap\Gbf(\Qbb)$.

(2) Let $(\Gbf,X)$ be a Shimura datum and let $K\subseteq\Gbf(\adele)$ be a compact open subgroup.
The Shimura variety associated to $(\Gbf,X)$ at level $K$ is an algebraic variety $M_K(\Gbf,X)$
whose $\Cbb$-points are described by the formula $$ M_K(\Gbf,X)(\Cbb)=\Gbf(\Qbb)\bsh [X\times\Gbf(\adele)/K],$$
where $\Gbf(\Qbb)$ acts on $X\times \Gbf(\adele)/K$ diagonally.
Fix $X^+$ a connected component of $X$, we actually have the following expression of geometrically connected components
$$M_K(\Cbb)=\coprod_a\Gamma_K(a)\bsh X^+,$$
where $a$ runs through a set of representatives of the quotient
$\Gbf(\Qbb)_+\bsh\Gbf(\adele)/K$, and $\Gamma_K(a)=\Gbf(\Qbb)_+\cap aKa^\inv$ acts on $X^+$ through $\Gbf^\ad(\Rbb)^+$.

By \cite{baily borel}, each $\Gamma_K(a)\bsh X^+$ is a quasi-projective algebraic variety over $\Cbb$.
Moreover $M_K(\Gbf,X)$ admits a canonical model over some number fields (cf. \cite{milne ann arbor}).
\end{definition}

\begin{example}[{Siegel modular variety, cf. \cite[\S\,6]{milne05}}]\label{Siegel modular variety}
Let $(V=\Qbb^{2g},\psi)$ be the standard symplectic space of dimension $2g$,
with symplectic basis $e_1,e_{-1},\cdots,e_{g},e_{-g}$ such that
$\psi(e_i,e_{-j})=\delta_{ij}$.
Denote by $\GSp_{2g}$ the connected reductive $\Qbb$-group of symplectic similitude of $(V,\psi)$,
and $\Hscr_g^\pm$ be the Siegel double half space of complex symmetric $g\times g$ matrices with definite imaginary part.
Then $(\GSp_{2g},\Hscr_g^\pm)$ is a Shimura datum, because $\Hscr_g^\pm$ is identified with the set of homomorphisms
$h:\Sbb\ra\GSp_{2g,\Rbb}$ such that the composition $\Sbb\ra\GSp_{2g,\Rbb}\ra\GL_{2g,\Rbb}$
defines a complex structure on $V_\Rbb$ with $(x,y)\mapsto\psi(x,h(\sqrt{-1})y)$ symmetric definite.

Let $n\geq 3$ be an integer and take $K$ to be the principal $n$-th congruence subgroup
$$K=K(n)=\Ker\big(\GSp_{2g}(\Zbhat)\ra\GSp_{2g}(\Zbb/n)\big).$$
Then $M_K(\GSp_{2g},\Hscr_g^\pm)$ is just $\Acal_{g,[n]}$, the moduli scheme over $\Qbb$
parameterizing principally polarized abelian varieties with full level-$n$ structure.
The set of geometrically connected components of $M_K(\GSp_{2g},\Hscr_g^\pm)$ is in bijection
with the set $\mu_n$ of $n$-th roots of 1, each of them isomorphic to $\Gamma(n)\bsh \Hscr_g^+$,
with $\Gamma(n)=\Ker(\Sp_{2g}(\Zbb)\ra\Sp_{2g}(\Zbb/n))$ the $n$-th principal congruence subgroup.

For simplicity we will use $\Acal_g$ to denote the fixed connected component indexed by $1\in\mu_n$
of the Shimura variety $M_K(\GSp_{2g},\Hscr_g^\pm)$ described above.
$\Acal_g$ has its canonical model over $\Qbb(\zeta_n)$ the $n$-th cyclotomic field. It is isomorphic to  the moduli scheme $\Acal_{g,1,n}$ constructed by Mumford in \cite{mumfordGIT}.
The whole Shimura variety $M_K(\GSp_{2g},\Hscr_g^\pm)$ is the $\Qbb$-scheme by composing
$\Acal_{g,1,n}\ra\Spec\Qbb(\zeta_n)$ with $\Spec\Qbb(\zeta_n)\ra\Spec \Qbb$, cf. \cite[Chapter 6, Section 6.4]{hida shimura}.
\end{example}

We will mainly focus on geometrically connected components:
\begin{definition}[connected Shimura data and varieties]\label{connected Shimura data and varieties}

(1) A connected Shimura datum is a triple $(\Gbf,X;X^+)$
where $(\Gbf,X)$ is a Shimura datum and $X^+$ is a connected component of $X$.

(2) A connected Shimura variety is a quotient space of the form $Z=\Gamma\bsh X^+$,
where $X^+$ comes from some connected Shimura datum $(\Gbf,X;X^+)$, and $\Gamma$ is a congruence subgroup of
$\Gbf^\der(\Rbb)_+:=\Gbf(\Rbb)_+\cap\Gbf^\der(\Rbb)$, namely the stabilizer of $X^+$ in $\Gbf^\der(\Rbb)$.
\end{definition}

\begin{remarks}\label{compare with the classical definition}
(i). The definition above of connected Shimura varieties differs slightly from the  one adopted in
\cite{milne05}, where $\Gamma$ is taken to be a congruence subgroup of $\Gbf^\ad(\Qbb)^+$.
Since the center of $\Gbf(\Rbb)_+$ acts on $X^+$ trivially,
it only differs from our version of connected Shimura varieties by a finite covering.

(ii). Just like general Shimura varieties in the sense of Definition \ref{Shimura data and Shimura varieties},
connected Shimura varieties admit canonical models over suitable number fields.
In this paper it suffices to treat them as complex algebraic varieties.
\end{remarks}

\begin{definition}[special subvarieties]\label{special (Shimura) subvarieties}
Let $(\Gbf,X;X^+)$ be a connected Shimura datum, which defines a connected  Shimura variety $Z=\Gamma\bsh X^+$.

(1) A Shimura subdatum  of $(\Gbf,X)$ is a Shimura datum $(\Gbf',X')$
such that $\Gbf'$ is a $\Qbb$-subgroup of $\Gbf$ and $X'$ is the $\Gbf'(\Rbb)$-orbit of some $h\in X$
such that $h(\Sbb)\subseteq\Gbf'_\Rbb$.
Note that $X'\mono X$ is equivariant with respect to $\Gbf'(\Rbb)\mono\Gbf(\Rbb)$.

A  connected Shimura subdatum  of $(\Gbf,X;X^+)$ is a connected Shimura datum  $(\Gbf',X';X'^+)$,
such that $(\Gbf',X')$ is a Shimura subdatum of $(\Gbf,X)$ and $X'^+$
is a connected component of $X'$ which is contained in $X^+$.

(2) Write $\ubf_\Gamma$ for canonical projection $X^+ \ra \Gamma\bsh X^+,\ x\mapsto \Gamma x$,
which we call the uniformization map of $Z$.
Then a special subvariety of $Z$ is of the form $Z'=\ubf_\Gamma(X'^+)$,
where $X'^+$ comes from some subdatum $(\Gbf',X';X'^+)$.

$Z'$ is actually the image of a morphism between connected Shimura varieties $\Gamma'\bsh X'^+\ra \Gamma\bsh X$
for some congruence subgroup $\Gamma'\subseteq\Gbf'^\der(\Qbb)_+$,
and it is a closed subvariety of $Z$ over $\Cbb$, which actually admits a model over some number field.
\end{definition}

We mention briefly the notion of Hecke translation in the setting of connected Shimura varieties.

\begin{definition}[Hecke translation]\label{Hecke translation}
Let $Z=\Gamma\bsh X^+$ be a connected Shimura variety defined by $(\Gbf,X;\,X^+)$.
For $a\in\Gbf(\Qbb)_+$, the Hecke correspondence associated to $a$ is the following diagram
$$\Gamma\bsh \ot{q}\la \Gamma_a\bsh X^+\ot{q_a}\ra \Gamma\bsh X^+$$
where
\begin{itemize}
\item $\Gamma_a=\Gamma\cap a^\inv\Gamma a$;
\item $q(\Gamma_a x)=\Gamma x$ and $q_a(\Gamma_a x)=\Gamma ax$.
\end{itemize}
Both $q$ and $q_a$ are finite morphisms of degree equal to $[\Gamma:\Gamma_a]$.
For $Z'$ a closed irreducible subvariety of $Z$, any irreducible component of $q_a(q^\inv Z')$
is called a Hecke translate of $Z'$ by $a$.

This also makes sense for general cycles in $Z$,
where we write $q_{a*}q^*(Z')$ as multiplicities could arise,
and the map $q_{a*}q^*$ is called the Hecke operator associated to $a$ (acting on the space of cycles).
\end{definition}

\begin{remark}\label{remarkheckedense}
It is easy to verify that the Hecke translate of a special subvariety remains special.
Moreover, for any special point $s\in Z$, the union of all the Hecke translates of $s$
using $a\in \Gbf(\Qbb)_+$ is dense in $Z$ for the analytic topology,
because $\Gbf(\Qbb)_+$ is dense in $\Gbf(\Rbb)_+$ by the real approximation of linear $\Qbb$-groups.
\end{remark}

\begin{remark}[terminology]\label{terminology}
In \cite{ullmo yafaev} etc. a Shimura subvariety of $M_K(\Gbf,X)$ is the image of a morphism
between Shimura varieties $f:M_{K'}(\Gbf',X')\ra M_K(\Gbf,X)$ given by a morphism of Shimura data
$f:(\Gbf',X')\ra (\Gbf,X)$, using some compact open subgroup $K'\subset\Gbf'(\adele)\cap K$.
In this setting we have the notion of Hecke correspondence given by adelic points $a\in\Gbf(\adele)$,
and special subvarieties are defined as geometrically irreducible components of the
(adelic) Hecke translate of a Shimura subvariety.

It turns out that the special subvarieties thus defined are subvarieties contained in
suitable connected components of $M_K(\Gbf,X)$.
Shifting between connected components following \cite[Lemma\,2.13]{chen kuga},
we see that the special subvarieties in the sense of \cite{ullmo yafaev} are the same as ours when restricted to a
connected component of $M_K(\Gbf,X)$. In particular, the notion of Hecke translation is not involved
in our definition of special subvarieties, although it will be needed elsewhere,
like the description of Shimura curves of type I.

Since we only work with connected Shimura varieties, the adjective ``connected''
will be often omitted if no ambiguity occurs, and our special subvarieties will be also called Shimura subvarieties,
like the Shimura curves in Section \ref{sectiontwotype}.
\end{remark}

To end the subsection, we include the notion of totally geodesic subvarieties following \cite{moonen98}:

\begin{definition}[totally geodesic subvarieties]\label{totally geodesic subvarieties}
Let $Z=\Gamma\bsh X^+$ be a Shimura variety defined by $(\Gbf,X;\,X^+)$, with $\ubf_\Gamma$ the uniformization map.

A totally geodesic subvariety of $Z$ is of the form $\ubf_\Gamma(Y_1^+\times\{y_2\})$,
where for some subdatum $(\Hbf,Y;\,Y^+)\subseteq(\Gbf,X;\,X^+)$
we have $(\Hbf^\ad,Y^\ad;\,Y^{\ad +})\isom(\Hbf_1,Y_1;\,Y_1^+)\times(\Hbf_2,Y_2;\,Y_2^+)$
and $y_2\in Y_2^+$. Here $(\Hbf^\ad,Y^\ad;\,Y^{\ad +})$ is deduced from $(\Hbf,Y;\,Y^+)$
by taking $Y^\ad $ to be the $\Hbf^\ad(\Rbb)$-orbit of the composition
$\Sbb\ot{y}\ra\Hbf_\Rbb\ra\Hbf^\ad_\Rbb$ using any $y\in Y$;
in particular, $Y^+=Y^{\ad +}$ as the center of $\Hbf(\Rbb)$ acts on $Y$ trivially.
\end{definition}

\begin{remark}
Totally geodesic subvarieties can also be defined in terms of differential geometry.
However, we do not need the fine geometry of these subvarieties,
except for a numerical characterization in the case of curves due to Viehweg and the second named author,
cf. Theorem \ref{theoremnumchar} and the original paper \cite{vz04}.
\end{remark}

\subsection{Two types of Shimura curves}\label{sectiontwotype}
We recall briefly the definition of two types of Shimura curves that will be studied later.

The first class of Shimura curves are modular curves ``diagonally'' embedded in $\Acal_g$,
and the embedding factors through a slightly generalized form of Hilbert modular varieties,
which we describe as follows
\begin{definition}[Hilbert modular variety]\label{Hilbert modular variety}
(1) A totally real \'etale $\Qbb$-algebra is a   finite dimensional \'etale $\Qbb$-algebra $F$
(necessarily commutative) such that the $\Rbb$-algebra $F\otimes_\Qbb\Rbb$ is isomorphic to
the direct product $\Rbb$-algebra $\Rbb^g$, with $g=\dim_\Qbb F$.
It is clear that $F$ is isomorphic to a finite product of totally real number fields $\prod\limits_iF_i$
with $\sum\limits_i[F_i:\Qbb]=g$, and $F$ is a $\Qbb$-form of the product algebra $\Qbb^g$.

(2) Let $F$ be a totally real \'etale $\Qbb$-algebra of dimension $g$.
Then $\Gbf:=\Res_{F/\Qbb}\GL_2$ is a $\Qbb$-form of the $d$-fold product $\GL_2^d$,
which splits after the base change $\Qbb\ra \Rbb$.
We thus have $\Gbf_\Rbb\isom\prod_{\sigma}\GL_{2,\sigma} $,
where $\GL_{2,\sigma}$ stands for $\GL_{2,\Rbb}$ indexed by one of the
$g$ distinct homomorphisms of $\Rbb$-algebras $\sigma:F\otimes_\Qbb\Rbb\ra\Rbb$.

Let $X$ be the $\Gbf(\Rbb)$-conjugacy class of the homomorphism
$$h:\Sbb\ra(\Res_{F/\Qbb}\GL_2)_\Rbb,\qquad \rho\exp(\sqrt{-1}\theta)\mapsto \bigg(\left[ \begin{array}{cc}
 \rho\cos\theta & \rho\sin\theta\\
-\rho\sin\theta & \rho\cos\theta
\end{array} \right]\bigg)_\sigma$$ using $\Gbf_\Rbb\isom\prod_\sigma\GL_{2,\sigma}$.
It is then immediate that $(\Gbf,X)$ is a Shimura datum in the sense of \cite{deligne77},
with $X\isom\prod_{\sigma}\Hscr^\pm_{1,\sigma}$   the $g$-fold product of $\Hscr_1^\pm$.

Note that the center of $\Gbf$ is $\Res_{F/\Qbb}\Gbb_\mrm$,
which is too big to be put into $\GSp_{2g}$ directly respecting the moduli interpretation.
We thus restrict to the $\Qbb$-subgroup
$$\Gbf^F:=\{g\in\Gbf: \det(g)\in\Gbb_{\mrm,\Qbb}\subseteq\Res_{F/\Qbb}\Gbb_\mrm \}.$$
In other words, we may start with an embedding $\Res_{F/\Qbb}\SL_2\mono\Sp_{2g}$.
This can be done by choosing an $F$-linear structure on the $\Qbb$-Lagrangian decomposition $V=V_+\oplus V_-$
described in Example \ref{Siegel modular variety}, say identify $e_1,\cdots,e_g$ with a $\Qbb$-basis of $F$
and extend this $F$-linear structure from $V_+=\bigoplus\limits_{i=1}^{g}\Qbb e_i$
to $V_-=\bigoplus\limits_{i=1}^{g}\Qbb e_{-i}$ respecting the symplectic form $\psi$.
The embedding $\Res_{F/\Qbb}\SL_2\mono\Sp_{2g}$ extends to $\Gbf^F\mono\GSp_{2g}$
by joining a central $\Qbb$-torus isomorphic to $\Gbb_\mrm$.
It is then clear that the homomorphism $h$ mentioned above has its image in $\Gbf^F_\Rbb$,
and we get a smaller Shimura datum $(\Gbf^F,X^F=\Gbf^F(\Rbb)\cdot h)$,
which is a subdatum of $(\GSp_{2g},\Hscr_g^\pm)$.

The moduli interpretation of Shimura subvarieties associated to $(\Gbf^F,X^F)$ is similar to
the case of usual Hilbert modular varieties, namely they classify abelian varieties with
endomorphism by $F$ up to isogeny (plus suitable level structures and polarization constraints).
\end{definition}
\begin{definition}[Shimura curves of type I]\label{deftypeI}
In Shimura subvarieties of $\Acal_g$ defined by the subdatum $(\Gbf^F,X^F)$ in
Definition \ref{Hilbert modular variety}(2),
we have Shimura curves embedded diagonally. In fact the diagonal embedding
$\GL_{2.\Rbb}\ra\prod_{\sigma}\GL_{2,\sigma}$ descends to $\GL_{2,\Qbb}\mono\Res_{F/\Qbb}\GL_2$,
which has image in $\Gbf^F$; the homomorphism $h:\Sbb\ra\Gbf_{\Rbb}^F$ used there factors through it,
which gives the chain of subdata $(\GL_2,\Hscr_1^\pm)\mono(\Gbf^F,X^F)\mono(\GSp_{2g},\Hscr_g^\pm)$.
Such Shimura curves are called Shimura curves of type I.
\end{definition}

Note that the Shimura curves given by different embeddings
$(\GL_2,\Hscr_1^\pm)\mono(\Gbf^F,X^F)\mono(\GSp_{2g},\Hscr_g^\pm)$
only differ from each other by Hecke translation using $\GSp_{2g}(\Qbb)$.
For example, if $E$ and $F$ are two totally real \'etale $\Qbb$-algebra of dimension $g$ giving rise to
$(\Gbf^E,X^E)$ and $(\Gbf^F,X^F)$ as above, then their embeddings into $(\GSp_{2g},\Hscr_g^\pm)$
are the same as the choice of $E$-structure (resp. $F$-structure) on the $\Qbb$-Lagrangian subspace
of the underlying symplectic space.
The restriction  to $(\GL_2,\Hscr_1^\pm)$ simply treats $E$ (resp. $F$) as a $\Qbb$-vector space of dimension $g$,
hence the embedding $(\GL_2,\Hscr_1^\pm)\mono(\GSp_{2g},\Hscr_g^\pm)$ always factors through some $(\Gbf^L,X^L)$
with $L=\Qbb^g$ (direct product $\Qbb$-algebra) given by the choice of a basis for
a $\Qbb$-Lagrangian subspace of the symplectic space. Since different $\Qbb$-Lagrangians are conjugate
under $\GSp_{2g}(\Qbb)$, we see that these embeddings of $(\GL_2,\Hscr_1^\pm)$
are permuted to each other by Hecke translation.

We mention some facts about Shimura curves of type I:

\begin{lemma}\label{algebraic interpretation of Shimura curves of type I}
{\rm (1)} Let $A$ be a principally polarized abelian variety over $\Cbb$ with
$\End^\circ(A):=\End(A)\otimes_\Zbb\Qbb$.
Then $\End^\circ(A)$ contains the matrix algebra $\Mat_g(\Qbb)$ if and only if
$A$ is isogenous to a $g$-fold self-product of some elliptic curve.

{\rm (2)} Let $A$ be an abelian variety as in {\rm (1)}.
Then the point $x_A$ on $\Acal_g$ parameterizing $A$ falls in some Hecke translate of the Shimura curve
defined by the diagonal embedding $(\GL_2,\Hscr_1^\pm)\mono(\Gbf^L,X^L)\mono(\GSp_{2g},\Hscr_g^\pm)$,
using the trivial real $\Qbb$-algebra $L=\Qbb^g$.
\end{lemma}
\begin{proof}
(1) A $g$-dimensional principally polarized abelian variety $A$ (over $\Cbb$) admits a decomposition up to
isogeny $A\sim\prod\limits_{i=1}^rA_i^{m_i}$ where the $A_i$'s are simple abelian varieties non-isomorphic to
each other and $m_i>0$ are integers such that $\sum\limits_{i=1}^rm_i=g$.
Hence the algebra of endomorphisms up to isogeny of $A$ is
$$\End^\circ(A)\isom\prod_{i=1}^r\Mat_{m_i}(D_i),$$
where $D_i=\End^\circ(A_i)$ is a division algebra of finite dimension over $\Qbb$.
The maximal semi-simple split $\Qbb$-algebra (i.e. a finite product of matrix algebras $\Mat_d(\Qbb)$'s)
of $\End^\circ(A)$ is equal to $\prod\limits_{i=1}^r\Mat_{m_i}(\Qbb)$.

If $\End^\circ(A)$ contains a split simple $\Qbb$-algebra of the form $\Mat_g(\Qbb)$,
then one must have $r=1$ and $m_1=g$, which means $A$ is isogenous to
a $g$-fold self-product of a single elliptic curve.
Conversely, if $A$ is isogenous to $E^g$ with $E$ some elliptic curve,
then $\End^\circ(A)=\Mat_g\big(\End^\circ(E)\big)$ contains $\Mat_g(\Qbb)$.

(2)  Let $x\in\Hscr_g^+$ be a point giving the Hodge structure $h:\Sbb\ra\GSp_{2g,\Rbb}\ra\GL_{V,\Rbb}$ on $V$,
which defines a point $\xbar=\Gamma x$ on $\Gamma\bsh \Hscr_g^+$ for some congruence subgroup
$\Gamma\subset\Sp_{2g}(\Qbb)$ (say $\Gamma=\Gamma(n)$ principal for some $n$).
Take $a\in\GSp_{2g}(\Qbb)_+$ and consider the Hecke correspondence
$$\Gamma\bsh \Hscr^+_g\ot{q}\la \Gamma_a\bsh \Hscr_g^+\ot{q_a}\ra\Gamma\bsh \Hscr_g^+$$
then by the Definition \ref{Hecke translation} we see that the point $\Gamma ax$ lies in $q_a(q^\inv \{\xbar\})$.

Applying this to the isogeny $f:E^g\ra A$,
we get $a=f_*:H_1(E^g,\Qbb)\isom H_1(A,\Qbb)$.
$H_1(E^g)$ and $H_1(A,\Qbb)$ correspond to two rational Hodge structure on the $\Qbb$-vector space $V=\Qbb^{2g}$,
namely two points $x_E$ and $x_A$ in $\Hscr_g^{+}$.
$a=f_*$ is an isomorphism of
polarized rational Hodge structure, which gives an element in $\GSp_{2g}(\Qbb)$,
still denoted as $a$. We may choose suitable symplectic bases of $H_1(E^g,\Qbb)$ and  $H_1(A,\Qbb)$
such that $a\in\GSp_{2g}(\Qbb)_+$. The element $a$ transports $x_E$ to $x_A$,
namely the conjugation of $x_E$ by $a$ equals $x_A$,
hence $\Gamma x_A$ is a Hecke translate of $\Gamma x_E$ by $a\in\GSp_{2g}(\Qbb)_+$.
\end{proof}

The second class of Shimura curves are the Shimura curves of Mumford type
constructed from corestrictions of quaternion algebras,
the idea of which goes back to \cite{mumford69}. We recall briefly the construction given in \cite{vz04}.

Let $F$ be a totally real field of degree $d$ over $\Qbb$, with $d$ distinct real embeddings
$\sigma_1,\cdots,\sigma_d$, and we use $\sigma=\sigma_1$ to identify $F$ as a subfield of $\Rbb$.
Let $A$ be a quaternion algebra over $F$, equipped with isomorphisms
$$\rho_1:A\otimes_{\sigma_1}\Rbb\isom\Mat_2(\Rbb),\qquad \rho_i:A\otimes_{\sigma_i}\Rbb\isom\Hbb,\ (i=2,\cdots,d),$$
with $\Hbb$ Hamilton's quaternion algebra over $\Rbb$.
The corestriction $D=\Cores_{F/\Qbb}A$ is a central simple algebra over $\Qbb$,
which is isomorphic to either
\begin{itemize}
\item $\Mat_{2^d}(\Qbb)$ and $d$ is odd; or
\item $\Mat_{2^d}(L)$, for some quadratic extension $L=\Qbb(\sqrt{b})$ over $\Qbb$;
      $L$ is imaginary if and only if $d$ is even.
\end{itemize}
In both cases we have an embedding $D=\Cores_{F/\Qbb}A\mono \Mat_{2^{d+\epsilon}}(\Qbb)$ for $\epsilon\in{0,1}$,
and we simply write it as $D\subseteq\Mat_{2^m}(\Qbb)$ with $m$ minimal.

Write $A^1$ for the kernel of the reduced norm ${\rm Nrd}:A^\times\ra F^\times$.
The $\Qbb$-group $\Gbf'$ associated to $A^1$ is connected and semi-simple,
with $\Gbf'_\Rbb\isom \SL_{2,\Rbb}\times\SU_2(\Rbb)^{d-1}$.
From \cite{vz04} we know that the homomorphism $A^\times\ra D\mono\Mat_{2^m}(\Qbb)$
defines a representation of $A^1$ which preserves a symplectic form on $V=\Qbb^{2^{m}}$.

We enlarge $\Gbf'$ to a connected reductive $\Qbb$-group $\Gbf$ which only differs from $\Gbf'$
by the split center $\Gbb_{\mrm,\Qbb}$. It is the $\Qbb$-group associated to
$$A^*:=\{a\in A:{\rm Nrd}(a)\in\Qbb^\times\subset F^\times\}.$$
Similar to the case of the affine modular curve $Y(d)$ discussed above,
we have the Shimura datum $(\Gbf,X)$,
where $X$ is the $\Gbf(\Rbb)$-conjugacy class of the following homomorphism
$h:\Sbb\ra \Gbf_\Rbb$ given by
$$z=\rho\exp(\sqrt{-1} \theta)\mapsto \bigg(\left[ \begin{array}{cc}
 \rho\cos\theta & \rho\sin\theta\\
-\rho\sin\theta & \rho\cos\theta
\end{array} \right],\,  I_2,\cdots,\, I_2 \bigg).$$
$X$ is isomorphic to $\Hscr_1^\pm$, and $(\Gbf,X)$ is a subdatum of
$(\GSp_{2^m},\Hscr^\pm_{2^{m-1}})$ by the representation of $\Gbf$ induced by $A\ra \Mat_{2^m}(\Qbb)$.

Write $C_A$ for the connected Shimura curve defined by the datum $(\Gbf,X)$
above using suitable level structure and the component $\Hscr_1^+$,
with $\eta$ the generic point of $C_A$. The defining symplectic representation of $\Gbf$
gives a universal family of abelian varieties $X_A\ra C_A$,
and its generic fiber $X_\eta$ is an abelian variety.
The endomorphism algebra of $X_\eta$ has been classified in \cite{vz04},
and one of the two following cases holds:\begin{enumerate}
\item[(1)] $m=d>1$, $\dim X_\eta=2^{d-1}$ and $\End(X_\eta)\otimes_\Zbb\Qbb=\Qbb$;

\item[(2)] $m=d+1$, $\dim X_\eta=2^d$, and \begin{enumerate}
\item[a.] for $d$ odd, $\End(X_\eta)\otimes_\Zbb\Qbb$ is a totally definite quaternion algebra over $\Qbb$;

\item[b.] for $d$ even, $\End(X_\eta)\otimes_\Zbb\Qbb$ is a totally definite quaternion algebra over $\Qbb$.
\end{enumerate}
\end{enumerate}
We remark that for $d=1$ or $2$,
there are only two Shimura curves of Mumford type for the given quaternion algebra $A$,
and both of them are of PEL type.
The curve  classifies abelian surfaces resp. abelian fourfolds $X$
with $\End(X)\otimes_\Zbb\Qbb$ a totally indefinite resp. totally definite quaternion algebra over $\Qbb$.

\subsection{Logarithmic Higgs bundles on  curves in $\Acal_g$}\label{sectionHiggsonC}
Let  $\cala_g=\mathcal A_{g,[n]}$ ($n\geq 3$) be the moduli space of principal polarized abelian varieties
with level-$n$ structure and
$\ol {\mathcal A}_{g}\supseteq \mathcal A_{g}$ a smooth toroidal compactification
with $\Delta:=\ol {\mathcal A}_{g}\setminus \mathcal A_{g}$.
Note that $\mathcal A_{g}$  carries a universal family of abelian varieties (cf. \cite{popp77})
$$\mathfrak h:~ \calx_g \lra  \mathcal A_{g}.$$
The relative de Rham bundle $ \Big(H^1_{dR}\big(\calx_g/\mathcal A_{g}\big),~\nabla\Big) $,
together with a polarization  and the Hodge filtration
$\mathfrak h_*\big(\Omega^1_{\calx_0/\mathcal A_{g,[n]}}\big)\subseteq\,
H^1_{dR}\big(\calx_g/\mathcal A_{g}\big)$,
forms a polarized variation of Hodge structure (PVHS).

Consider  the underlying universal locally constant sheaf
$\mathbb V=R^1\mathfrak h_*\mathbb Q_{\calx_0}.$
If $n$ is large enough, then $\mathbb V$  has unipotent local monodromy
around all components of the boundary $\Delta$ by  \cite[\S\,4]{mumford77}.
We will always assume that $\mathbb V$ has the property.

The above PVHS has a unique  extension
over $\ol {\mathcal A}_{g,[n]}$, and the extended Gauss-Manin connection $\nabla$ has logarithmic poles along  $\Delta$,
cf. \cite[\S\,11.1, \S\,11.2]{peterssteenbrink}.
By taking the grading of the extended Hodge filtration, one obtains a logarithmic system of Hodge bundles
$\left(E^{1,0}\oplus E^{0,1},\,\theta\right)$,
where the Higgs field
$$ \theta:~ E^{1,0}\lra E^{0,1}\otimes\Omega^1_{ \ol {\mathcal A}_{g}}(\log\Delta)$$
is an extension of  the following Kodaira-Spencer map on the Hodge bundles
$$\theta:~ \mathfrak h_*\big(\Omega^1_{\calx_g/\mathcal A_{g}}\big)\lra
R^1\mathfrak h_*\mathcal O_{\calx_g}\otimes\Omega^1_{\mathcal A_{g}}.$$

Consider a (smooth) projective curve contained in $\ol\cala_{g,[n]}$:
$$\phi:~ \ol C \hookrightarrow \ol {\mathcal A}_{g,[n]},\qquad \text{with $\Delta_{\ol C}:=\phi^{-1}(\Delta)$ a divisor.}
$$
Then by pull-back $C:=\ol C\setminus \Delta_{\ol C}$ carries a universal family of abelian varieties
$$h:\,X\to C$$
and a PVHS with underlying local system $\mathbb V_{C}:= R^1h_*\mathbb Z_{X} =\phi^*\mathbb V,$
 which has unipotent local monodromy around $\Delta_{\ol C}$ by assumption.
Because of the compatibility of the Deligne's canonical extension with pullback under a morphism, we obtain
$$ \left(E^{1,0}_{\ol C}\oplus E^{0,1}_{\ol C},~\theta_{\ol C}\right)
   =\phi^*\left(E^{1,0}\oplus  E^{0,1},~\theta\right).$$
In particular, for a Shimura curve $C$, one obtains an associated logarithmic Higgs bundle on
the smooth completion $\ol C\supseteq C$.

\subsection{Families of semi-stable curves}\label{sectionfamilyofcurve}
Our main technique will be built on the theory of (one-dimensional) families of semi-stable curves.
In the subsection, we would like to review some basic facts and fix the notations,
which will be used freely in this paper, cf. \cite{bhpv,ch88,harrismorrison}.

Recall that a semi-stable  (resp. stable) curve is  a complete connected reduced nodal  curve such that
each   rational component intersects with the other components at $\geq 2$  (resp. 3) points.
A semi-stable (resp. stable)  family of curves is a
flat projective morphism  $\bar f:\,\ol S\to \ol B$
from a projective surface $\ol S$ to a smooth projective curve $\ol B$
with connected fibres such that all the singular fibres of $\bar f$ are
semi-stable (resp. stable) curves. Moreover, $\bar f$ is said to be
\begin{itemize}
\item a hyperelliptic family if a general fibre of $\bar f$ is a hyperelliptic curve;
\item isotrivial     if  all its smooth fibres are isomorphic to each other;
\item relatively minimal if no singular
fiber of $\bar f$ has  any  $(-1)$-component.
\end{itemize}
Note that  if $\bar f$ is
semi-stable, then $\bar f$ is relatively minimal. From now on, we  assume
that $\bar f:\,\ol S \to \ol B$ is a semi-stable family  of curves of genus
$g\geq 2$ with singular fibres $\Upsilon\to\Delta$ and     $\ol S$ is
smooth.

Denote by $\omega_{\ol S/\ol B}=\omega_{\ol S}\otimes \bar f^*\omega_{\ol B}^{\vee}$ the
relative canonical sheaf of $\bar f$. Let $b=g(\ol B)$,
$p_g=h^0(\ol S,\,\omega_{\ol S})$, $q=h^0(\ol S,\,\Omega_{\ol S}^1)$, $\chi(\mathcal
O_{\ol S})=p_g-q+1$, and $\chit(\cdot)$ be the topological Euler
characteristic. Consider the following relative invariants:
\begin{equation}\label{defofrelativeinv}
\left\{
\begin{aligned}
&\omega_{\ol S/\ol B}^2=\omega_{\ol S}^2-8(g-1)(b-1),\\
&\delta_{\bar f}=\chit(\ol S)-4(g-1)(b-1)=\sum_{F\in \Upsilon}\delta(F),\\
&\deg \bar f_*\omega_{\ol S/\ol B}=\chi(\mathcal O_{\ol S})-(g-1)(b-1),
\end{aligned}\right.
\end{equation}
where $\delta(F)$ is the number of nodes of $F$.
All the invariants in \eqref{defofrelativeinv} are nonnegative and
satisfy the Noether's formula:
\begin{equation}\label{formulanoether}
12\deg \bar f_*\omega_{\ol S/\ol B}=\omega_{\ol S/\ol B}^2+\delta_{\bar f}.
\end{equation}
And $\deg \bar f_*\omega_{\ol S/\ol B}=0$  ( or equivalently, $\omega_{\ol S/\ol B}^2=0$)
if and only if $\bar f$ is smooth and isotrivial.
Since $\bar f$ is semi-stable, we also have
\begin{equation}\label{eqnomega=Omega}
\bar f_*\omega_{\ol S/\ol B}=\bar f_*\Omega^1_{\ol S/\ol B}(\log\Upsilon),
\end{equation}
where $\bar f_*\Omega^1_{\ol S/\ol B}(\log\Upsilon)$ is defined by the following exact sequence
$$0\lra \bar f^*\Omega^1_{\ol B}(\log\Delta)\lra \Omega^1_{\ol S}(\log\Upsilon)
\lra \Omega^1_{\ol S/\ol B}(\log\Upsilon) \lra 0.$$

By contracting all $(-2)$-curves contained in singular fibres, one gets a stable family
$\bar f^{\#}:\,\ol S^{\#} \to \ol B$ and a commutative diagram as below:
$$\xymatrix{
  \ol S \ar[rr] \ar[dr]_-{\bar f}
                &  &    \ol S^{\#}  \ar[dl]^-{\bar f^{\#}}   \\
                & \ol B                 }$$
$\ol S^{\#}$ is not necessarily smooth. For every singular point $q$ of
$\ol S^{\#}$, $(\ol S^{\#},\,q)$ is a rational double point of type
$A_{\lambda_q}$ (cf. \cite{bhpv}) with $\lambda_q$ the number of $(-2)$-curves in
$\ol S$ over $q$.

We are going to define  invariants $\{\delta_i(F)| 0\leq i \leq
[g/2]\}$ for a singular fibre $F$ of  $\bar f$.  First, we say a singular
point $q$ of $F$ to be  of type $i\in [1,  g/2]$ (resp. 0) if the
partial normalization of $F$ at $q$ consists of two connected
components of arithmetic genera $i$ and $g-i$ (resp. is connected).
Then we define $\delta_i(F)$ to be the number of singular points of
type $i$ in $F$. Or alternatively we define $\delta_i(F)$ in terms
of the stable model $F^{\#}\subseteq \ol S^{\#}$. Recall that a singular
point $q\in F^{\#}$ is said to have multiplicity $m$ if    $\ol S^{\#}$
around $q$ is locally of the form $xy=t^{m}$, where $t$ is a local
coordinate of $\ol B$. Then $\delta_i(F)$ is defined to be the number of
singular points of type $i$ counting multiplicity  in $F^{\#}$. We
remark that $(\ol S^{\#},\,q)$ is a rational double point of type
$A_{m_q-1}$, if $m_q>1$ is the multiplicity of $q$.

Denote always by $\Upsilon_{ct}\to \Delta_{ct}$ (resp.
$\Upsilon_{nc}\triangleq \Upsilon\setminus \Upsilon_{ct} \to
\Delta_{nc}\triangleq\Delta\setminus\Delta_{ct}$ ) the singular
fibres with compact (resp.  non-compact) Jacobian.  Define
$\delta_h(F)=\sum\limits_{i=2}^{[g/2]} \delta_i(F)$, and
\begin{equation}\label{formulaofdelta_i}
\left\{
\begin{aligned}
&\delta_i(\Upsilon)=\sum_{F\in \Upsilon} \delta_i(F),\quad
\delta_i(\Upsilon_{ct})=\sum_{F\in \Upsilon_{ct}} \delta_i(F),\quad
\delta_i(\Upsilon_{nc})=\sum_{F\in \Upsilon_{nc}} \delta_i(F).\\
&\delta_h(\Upsilon)=\sum_{i=2}^{[g/2]} \delta_i(\Upsilon),\quad
\delta_h(\Upsilon_{ct})=\sum_{i=2}^{[g/2]} \delta_i(\Upsilon_{ct}).
\end{aligned}
\right.
\end{equation}
Then
\begin{equation}\label{formulaofdelta_f}
\left\{
\begin{aligned}
\delta(F)&=\sum_{i=0}^{[g/2]}\delta_i(F)=\delta_0(F)+\delta_1(F)+\delta_h(F),\\
\delta_{\bar f}&=\sum_{i=0}^{[g/2]}\delta_i(\Upsilon)=\delta_0(\Upsilon)+\delta_1(\Upsilon)+\delta_h(\Upsilon).
\end{aligned}
\right.
\end{equation}

When $F\in \Upsilon_{ct}$, each irreducible component of $F$ is smooth.
So one can define
\begin{equation}\label{eqndefofl_i(F)}
l_i(F)=\#\big\{D\subseteq F~\big|~ g(F)=i\big\},\qquad l_h(F)=\sum_{i\geq 2}l_i(F).
\end{equation}
Note that the dual graph of $F$ is a tree for $F\in \Upsilon_{ct}$.
Hence
\begin{equation}\label{eqnFcompJac}
\delta_0(F)=0,\quad
\sum_{j}\delta_j(F)=\sum_{i}l_i(F)\, -1,\quad
\sum_{i}i\cdot l_i(F)=g,\qquad \forall~F\in \Upsilon_{ct}.
\end{equation}

We also want to remark that these invariants $\delta_i(\Upsilon)$'s have the following moduli meanings.
Let $\ol \calm_g$ be the moduli space of complex stable  curves of
genus $g$. By \cite{dm69}, the boundary $\ol \calm_g \setminus
\calm_g$ is of codimension one and has $[g/2]+1$ irreducible
components  $\Delta_0,\,\Delta_1,\,\cdots,\,\Delta_{[g/2]}$, which
define divisor classes in $\Pic(\ol \calm_g)\otimes\mathbb Q$. Note
that a general point of $\Delta_0$ represents an irreducible stable
curve with one node, while a general point of $\Delta_i$ ($i>0$)
corresponds to a stable curve consisting of two components
  of arithmetic genera  $i$ and  $g-i$ respectively and intersecting
at one point.  There is also a natural class $\lambda \in \Pic(\ol
\calm_g)\otimes\mathbb Q$ called the Hodge class with the following
property (cf. \cite{dm69}): for every non-isotrivial semi-stable
family $\bar f:\, \ol S \to \ol B$ with the associated moduli morphism
$\varphi:\, \ol B \to \ol M_g$, then
\begin{equation}\label{modulidelta}
\deg \varphi^*(\lambda)=\deg \bar f_*\omega_{\ol S/\ol B},\qquad
\delta_i(\Upsilon)=\deg \varphi^*(\Delta_i).
\end{equation}

We now assume that  $\bar f:\, \ol S \to \ol B$ is  a semi-stable family of hyperelliptic
curves of genus $g\geq2$ till the end of this subsection. We are going to define
invariants
$$\xi_j(\Upsilon)=\sum_{F\in\Upsilon}\xi_j(F),\qquad \forall~0\leq j\leq [(g-1)/2].$$
It suffices to define  $\xi_j(F)$ for   singular
fibers $F$ of $\bar f$.

{ \leftmargini=4mm
\begin{itemize}
\item  First we define the index of a singular point
$p$ of  a stable $(2g+2)$-pointed nodal curve $\Gamma$ of arithmetic
genus zero. Note that $\Gamma\backslash\{p\}$ consists of two
connected components $\Gamma'$ and $\Gamma''$, which respectively
contain  $\alpha_1$ and $\alpha_2$   marked points. Clearly  $\alpha_1+\alpha_2=2g+2$.
We call $\min(\alpha_1,  \alpha_2)$ to be the index of
$p\in\Gamma$.

\item Next we describe singular points of a semi-stable
hyperelliptic curve $\wt F$ when $\wt F$ can be viewed as an admissible
double $\psi:\, \wt F\to \Gamma$ over a stable curve $\Gamma$ as above
(cf. \cite{ch88} or \cite{harrismumford82}).  If $p\in\Gamma$ has
odd index $2k+1$, then $\psi$  is branched at $p$   and the
unique  point $q\in \wt F$ lying above $p$ is a singular point of type
$k$.  If $p\in\Gamma$ has even index $2k+2$, then $\psi$ is
unbranched at $p$ and   two points $q', q''\in \wt F$ lying above $p$
are of type $0$. Define invariants
\begin{align*}
 &\xi_0(\wt F):=2 \cdot \#\{\text{singular points in }  \Gamma  \text{ of index }
 2\},\\
& \xi_j(\wt F) :=  \#\{\text{singular points in } \Gamma \text{ of index
} 2j+2\}, \quad 1\leq j\leq [(g-1)/2].
\end{align*}
\item Finally we define $\xi_j(F)$ for any singular fiber $F$ of $\bar f$. Let
$\tilde f:\,\wt S \to \wt B$ be the semi-stable family corresponding
to the base change of $\bar f$ with respect to a finite morphism $\pi:\wt
B\to \ol B$ of degree $d$. When $d>>0$, the pre-image $\wt F$ of $F$ is
an admissible double cover of a stable $(2g+2)$-pointed nodal curve
$\wt\Gamma$ of arithmetic genus zero. Then define
$$ \xi_j(F)=\frac{\xi_j(\wt F)}{d},\qquad \forall~0\leq j\leq [(g-1)/2].$$\end{itemize}}

Clearly the definition of  $\xi_j(F)$ is independent of the choice
of $\pi$. In particular we have
$$\delta_0(F)=\xi_0(F)+2\sum_{j=1}^{[(g-1)/2]}\xi_j(F).$$

Let $\calh_g\subseteq \calm_g$ (resp. $\ol \calh_g\subseteq \ol
\calm_g$) be the moduli space of smooth (resp. stable) hyperelliptic
complex curves of genus $g$. By \cite{ch88}, $\Delta_i\cap \ol
\calh_g$ is an irreducible divisor of $\ol \calm_g$, also denoted by
$\Delta_i$; $\Delta_0\cap \ol \calh_g$ is not irreducible, actually
$$\Delta_0\cap \ol \calh_g=\Xi_0\cup \Xi_1\cup\cdots\cup \Xi_{[(g-1)/2]}\,,$$
where  $\Xi_0$ consists of  irreducible stable hyperelliptic curves
with a unique node, and for $1\leq j\leq [(g-1)/2]$, a general point
of $\Xi_j$  represents a stable curve consisting of two
hyperelliptic curves  intersecting at two points and respectively of
genera $j$ and $g-j-1$.  As divisors (cf. \cite{ch88}),
$$h^*\left(\Delta_0\right)=\Xi_0+2\sum_{j=1}^{[(g-1)/2]}\Xi_j\,,\qquad
\text{where~$h:\,\ol \calh_g\hookrightarrow \ol \calm_g$~is the embedding.}$$

Assume that   $\bar f:\,\ol S\to \ol B$ is a non-isotrivial semi-stable family of
hyperelliptic curves and $\varphi:\, \ol B \to \ol \calh_g$ is the
induced map, then
\begin{equation}\label{relationdeltaxi}
\left\{
\begin{aligned}
\xi_j(\Upsilon)   &=\deg \varphi^*(\Xi_j),&&\forall~0\leq j\leq [(g-1)/2];\\
\delta_0(\Upsilon)&=\deg \varphi^*(\Xi_0)+2\sum_{j=1}^{[(g-1)/2]}\deg \varphi^*(\Xi_j);&&\\
\delta_i(\Upsilon)&=\deg \varphi^*(\Delta_i),&\quad&\forall~1\leq i\leq [g/2].
\end{aligned}\right.
\end{equation}

\section{Family of semi-stable curves representing a curve in $\calt_g$}\label{sectionrepresent}
Given a smooth closed curve $C\Subset \calt_g$ with a suitable smooth compactification $\ol C$
by joining a finite set of cusps $\Delta_{\ol C}$,
we would like to construct a family $\bar f:\,\ol S \to \ol B$ of semi-stable curves representing $C$ in the section.
We also investigate the exact relation between the associated logarithmic Higgs bundles over $\ol B$ and $\ol C$.

\vspace{0.15cm}
Fix an integer $n$,
let $\calm^{ct}_g=\calm^{ct}_{g,[n]}\supseteq \calm_g=\calm_{g,[n]}$ be the partial compactification of
the moduli space of smooth projective genus-$g$ curves  with level-$n$ structure by adding
  stable curves with  compact Jacobians.
When $n\geq 3$, it carries a universal family of stable curves with compact Jacobians (cf. \cite{popp77})
\begin{equation}\label{eqnuniverseonM}
\mathfrak f:~ \cals_g^{ct} \lra {\mathcal M}^{ct}_{g}.
\end{equation}
The Torelli morphism $j^{\rm o}$ can be naturally extended to $\calm^{ct}_{g}$:
$$ j:~\calm^{ct}_{g} \lra \cala_{g}, \qquad\text{with~}
\calt_{g}=j\big(\calm^{ct}_{g}\big).$$
The morphism $j^{\rm o}$ is 2:1 and ramified exactly on the locus of hyperelliptic curves (cf. \cite{os79}).
However the relative dimension of $j$ is positive along the boundary $\calt_g\setminus \calt^{\rm o}_g$.

Let $B$ be the normalization of the strict inverse image $j^{-1}(C)$ of $C$,
and denote by $j_B:\,B \to C$ the induced morphism.
If $B$ is reducible, then replace $B$ by one   irreducible component.
By pulling back the universal family $\mathfrak f:\,\cals_g^{ct} \to \calm^{ct}_{g}$
to $B$ and resolving singularities,
one gets a family  $f: S\to B$ of semi-stable curves that extends uniquely to a family
$ \bar f:\ol S\to \ol B$ of semi-stable curves
over the smooth completion $\ol B\supseteq B$.
\begin{definition}\label{defrepresenting}
The family $\bar f:\,\ol S \to \ol B$ is called the family of semi-stable curves representing $C\subseteq \calt_g$
via the Torelli morphism.
\end{definition}
Let $h: X \to C$ be the universal family in Section \ref{sectionHiggsonC}.
By the construction of  the Torelli morphism, we obtain the following
\begin{proposition}
Let $jac(f):\,Jac(S/B)\lra B$  denote the relative Jacobian of the family $f: S\to B$
and $j_B:\,B \to C$ the induced morphism as above. Then
$$ \Big(jac(f):~Jac(S/B)\lra B \Big) = j_B^*\big(h: X\lra C\big).$$
In particular,
$\mathbb V_{B}:=R^1jac(f)_*\mathbb Q_{Jac(S/B)}=j_B^*\mathbb V_{C}.$
\end{proposition}

It is well-known that the logarithmic Higgs bundle associated to $\mathbb V_{B}$ has the form
\begin{equation}\label{eqnpfmainHiigsB}
\left(E_{\ol B}^{1,0}\oplus E_{\ol B}^{0,1},~\theta_{\ol B}\right)
=\left(\bar f_*\omega_{\ol S/\ol B}\oplus R^1\bar f_*\mathcal O_{\ol S},~\theta_{\ol B}\right),
\end{equation}
and it admits a  decomposition of Higgs bundles similarly to \eqref{decompC}:
\begin{equation}\label{decompB}
\left(E_{\ol B}^{1,0}\oplus E_{\ol B}^{0,1},\theta_{\ol B}\right)
=\left(A_{\ol B}^{1,0}\oplus A_{\ol B}^{0,1},~\theta_{\ol B}\big|_{A_{\ol B}^{1,0}}\right)\oplus \left(F_{\ol B}^{1,0}\oplus F_{\ol B}^{0,1},~ 0\right).
\end{equation}

Since both $\ol B$ and $\ol C$ are smooth projective curves, the morphism $j_B: B\to C$
extends to a morphism $\bar j_B: \ol B\to \ol C$ such that $\Delta_{nc}:=\ol B\setminus B =\bar j_B^{-1}(\Delta_{\ol C})$
and $j_B^* (\mathbb V_{C})=\mathbb V_{B}.$ Hence
\begin{equation*}
\bar j_B^* \left(E^{1,0}_{\ol C}\oplus E^{0,1}_{\ol C},~\theta_{\ol C}\right)=\left(E^{1,0}_{\ol B}\oplus E^{0,1}_{\ol B},~\theta_{\ol B}\right).
\end{equation*}
In particular,
\begin{equation}\label{xinpullback}
\bar j_B^*\left(E^{1,0}_{\ol C}\right)=E^{1,0}_{\ol B}=\bar f_*\omega_{\ol S/\ol B},
\qquad \bar j_B^*\left(A^{1,0}_{\ol C}\right)=A^{1,0}_{\ol B}.
\end{equation}

\begin{definition}\label{defmaximalB}
Let $\bar f:\ol S\to \ol B$ be any family of semi-stable curves of genus $g\geq 2$
and $\Upsilon_{nc} \to \Delta_{nc}$ the singular fibres with non-compact Jacobian.
Then $\bar f$ is said to be with {\it maximal} Higgs field if
$$\theta_{\ol B}\big|_{A_{\ol B}^{1,0}}:\,A_{\ol B}^{1,0} \lra A_{\ol B}^{0,1}\otimes \Omega^1_{\ol B}(\log \Delta_{nc})$$
is an isomorphism, and to be with  {\it strictly maximal}  Higgs field if furthermore $F_{\ol B}^{1,0}\oplus F_{\ol B}^{0,1}=0$.
By \cite{vz03}, $\bar f$ has strictly maximal (resp. maximal) Higgs field if and only if
\begin{equation}\label{arakelovB}
\begin{aligned}
\deg \bar f_*\omega_{\ol S/\ol B}&={g\over 2}\cdot\deg\Omega^1_{\ol B}(\log \Delta_{nc}),\\
\Big(\text{resp.}\quad\deg \bar f_*\omega_{\ol S/\ol B}&={\rank A^{1,0}_{\ol B}\over 2}\cdot\deg \Omega^1_{\ol B}(\log\Delta_{nc})~\Big).
\end{aligned}
\end{equation}
\end{definition}

\v
For a smooth closed curve $C\Subset{\mathcal T}_g$, we have  given two definitions respectively regarding
the (strict) maximality of the Higgs fields $\theta_{\ol C}$ and $\theta_{\ol B}$.
To understand the relation between them, we start with the following:
\begin{proposition}\label{proprelationHiggs}
Let $\bar f:\ol S\to \ol B$ be the family of semi-stable curves representing a smooth closed curve $C\Subset \calt_g \subseteq \cala_g$
as above. Then
$\rank A^{1,0}_{\ol B}=\rank A^{1,0}_{\ol C}$, and

{\rm (i).} if $C\Subset \calth_g$, then
\begin{equation}\label{eqnhiggshyper}
\deg \bar f_*\omega_{\ol S/\ol B} = \deg E^{1,0}_{\ol C},\quad
\deg \Omega_{\ol B}^1(\log\Delta_{nc})=\deg\Omega_{\ol C}^1(\log\Delta_{\ol C});
\end{equation}

{\rm (ii).} if $C\nsubseteq \calth_g$, then
\begin{equation}\label{eqnhiggsnonhyper}
\deg \bar f_*\omega_{\ol S/\ol B} = 2\deg E^{1,0}_{\ol C},\quad
\deg \Omega_{\ol B}^1(\log\Delta_{nc})=2\deg\Omega_{\ol C}^1(\log\Delta_{\ol C})+|\Lambda|,
\end{equation}
where $\Lambda$ is the ramification locus of the induced cover $j_B:\,B \to C$.
\end{proposition}
\begin{proof}
By \eqref{xinpullback}, it is clear that $\rank A^{1,0}_{\ol B}=\rank A^{1,0}_{\ol C}$.
Note that the Torelli morphism $j^\circ:\calm_g\ra \cala_g$ is a 2-to-1
morphism ramified exactly on the hyperelliptic locus $\mathcal H_{g}$.
Hence if $C\Subset \calth_g$, then $j_B:\,B \to C$ is an isomorphism, and so is
$\bar j_B:\, \ol B \to \ol C$. Thus \eqref{eqnhiggshyper} follows from \eqref{xinpullback}.

Suppose $C\nsubseteq \calth_g$; then $j^{-1}(C) \to C$ is a 2-to-1 morphism. If
$j^{-1}(C)$ is reducible, then $B$ is the normalization of one
of irreducible components of $j^{-1}(C)$. So $B \cong C$,
$\ol B\cong \ol C$, and \eqref{eqnhiggsnonhyper} follows from \eqref{xinpullback}.
If $j^{-1}(C)$ is irreducible, then   $\bar j_B:\,\ol B \to \ol C$ is a double cover.
So by \eqref{xinpullback}, $\deg \bar f_*\omega_{\ol S/\ol B} = 2 \deg E^{1,0}_{\ol C}$;
and
Since $\Delta_{nc}=\bar j_B^{-1}(\Delta_{\ol C}),$ by Hurwitz
formula for sheaves of logarithmic 1-forms, one has
$$\deg \Omega_{\ol B}^1(\log\Delta_{nc})=2\deg\Omega_{\ol C}^1(\log\Delta_{\ol C})+|\Lambda|.$$
This completes the proof.
\end{proof}

\begin{corollary}[Hyperelliptic locus]\label{prophyper}
If $C\Subset \calth_g$, then $\theta_{\ol C}$ is strictly maximal {\rm(}resp. maximal{\rm)}
if and only if $\theta_{\ol B}$ is strictly maximal {\rm(}resp. maximal{\rm)}.
\end{corollary}

\begin{corollary}[Non-hyperelliptic locus]\label{propnonhyper}
Suppose $C\Subset\mathcal T_{g, [n]}$, but $C\nsubseteq \calth_g.$
Then

{\rm(i).} If $\theta_{\ol B}$ is strictly maximal {\rm(}resp.
maximal{\rm)}, then $\theta_{\ol C}$ is strictly maximal {\rm(}resp.
maximal{\rm)};

{\rm(ii).} Conversely, if $\theta_{\ol C}$ strictly maximal {\rm(}resp.
maximal{\rm)}, then
$$\begin{aligned}
\deg \bar f_*\omega_{\ol S/\ol B}&
={g\over 2}\cdot\deg\Omega^1_{\ol B}\left(\log\Delta_{nc}\right)-{g\over 2}\cdot|\Lambda|,\\
\Big(\text{resp.}\quad\deg \bar f_*\omega_{\ol S/\ol B}&
={\rank A^{1,0}_{\ol B}\over 2}\cdot\deg\Omega^1_{\ol B}(\log\Delta_{nc})
-{\rank A^{1,0}_{\ol B}\over 2}\cdot|\Lambda|~\Big),
\end{aligned}$$
\end{corollary}

\section{The strict Arakelov inequalities and proofs of the main results}\label{sectionconclusion}
In the section, we study the strict Arakelov inequalities, i.e. Theorem \ref{theoremstrictarak}.
The main Theorems \ref{mainthm1}, \ref{mainthm2} and \ref{mainthm3}
are immediate consequences of these inequalities and the numerical characterization of Shimura curves
and totally geodesic curves (cf. Theorem \ref{theoremnumchar}).

We first recall the Miyaoka-Yau type inequality and sharp slope inequality
for a family of semi-stable curves in Section \ref{sectiontwotypeineq},
from which we deduce the strict Arakelov inequalities in Section \ref{sectionpfofArak}.
The proofs of these two types of inequalities are postponed to
Sections \ref{sectionupper} and \ref{sectionlower} respectively.

\subsection{Two types of  inequalities for a family of semi-stable curves}\label{sectiontwotypeineq}
We state the two types of inequalities as the following theorems.

\begin{theorem}[Miyaoka-Yau type inequality I, cf. Section\,\ref{sectionpfofupper}]\label{thmupper2}
Let $\bar f:\,\ol S \to \ol B$ be a non-isotrivial
family of semi-stable curves of genus $g\geq2$. Then
\begin{equation}\label{eqnupper2}
\omega_{\ol S/\ol B}^2 \leq (2g-2)\cdot \deg\left(\Omega^1_{\ol B}(\log\Delta_{nc})\right)
+2\delta_1(\Upsilon_{ct})+3\delta_h(\Upsilon_{ct}).
\end{equation}
Moreover, if $\Delta_{nc}\neq \emptyset$ or $\Delta=\emptyset$, then the above inequality is strict.
\end{theorem}

The proof of Theorem \ref{thmupper2} is based on a theorem of Miyaoka (cf. \cite{miyaoka84})
for the bound on the number of quotient singularities in a surface plus base change technique.
Recently Peters (cf. \cite{peters14}) has informed us that he has a simplified proof
by using Cheng-Yau's theorem for a log surface instead of Miyaoka's.

\begin{theorem}[Moriwaki's Sharp slope inequality, cf. \cite{moriwaki98} and
Section\,\ref{sectionlower2}]\label{thmlower2}
Let $\bar f:\,\ol S \to \ol B$ be the same as in Theorem {\rm\ref{thmupper2}}.
Then
\begin{equation}\label{eqnlower2}
\omega_{\ol S/\ol B}^2\geq \frac{4(g-1)}{g}\cdot\deg \bar f_*\omega_{\ol S/\ol B}
+\frac{3g-4}{g}\delta_1(\Upsilon)+\frac{7g-16}{g}\delta_h(\Upsilon).
\end{equation}
\end{theorem}

\begin{theorem}[Sharp slope inequality I, cf. Section\,\ref{sectionlower3}]\label{thmlower3}
Let $\bar f:\,\ol S \to \ol B$ be the same  as in Theorem {\rm\ref{thmupper2}} and $q_{\bar f}=q(\ol S)-g(\ol B)$ the relative irregularity.
If $\bar f$ is hyperelliptic, then
\begin{eqnarray}
\omega_{\ol S/\ol B}^2&\geq&\frac{4(g-1)}{g-q_{\bar f}}
\cdot\deg \bar f_*\omega_{\ol S/\ol B}+\label{eqnlower3}\\[0.15cm]
&&
\left\{
\begin{aligned}
&\begin{aligned}
&\frac{3g^2-(8q_{\bar f}+1)g+10q_{\bar f}-4}{(g+1)(g-q_{\bar f})} \delta_1(\Upsilon)\\
&\quad+\frac{7g^2-(16q_{\bar f}+9)g+34q_{\bar f}-16}{(g+1)(g-q_{\bar f})} \delta_h(\Upsilon),
\end{aligned}&&\text{if~} \Delta_{nc}\neq \emptyset;\\[0.2cm]
&\sum_{i=1}^{[g/2]} \left(\frac{4(2g+1-3q_{\bar f})i(g-i)}{(2g+1)(g-q_{\bar f})}-1\right) \delta_i(\Upsilon), &\qquad&
\text{if~}\Delta_{nc}= \emptyset.
\end{aligned}\right.\nonumber
\end{eqnarray}
Moreover, if $\Delta_{nc}=\emptyset$ and $q_{\bar f}\geq 2$, then
\begin{equation}\label{eqnlower3'}
\sum_{i=q_{\bar f}}^{[g/2]} \frac{(2i+1)(2g+1-2i)}{g+1} \cdot \delta_i(\Upsilon)\geq
\sum_{i=1}^{q_{\bar f}-1} 4i(2i+1) \cdot \delta_i(\Upsilon).
\end{equation}
\end{theorem}

While Theorem \ref{thmlower2} is a direct consequence of Moriwaki's theorem (cf. \cite{moriwaki98}),
Theorem \ref{thmlower3} is proved  based on formulas given by Cornalba and Harris (cf. \cite{ch88}).
The observation that the smooth double cover
induced by the hyperelliptic involution is fibred when  $q_{\bar f}>0$ plays a crucial role.

To get the strict Arakelov inequality for a smooth closed curve $C\Subset \calt_{g,[n]}$,
we need to deal with the family  $\bar f:\,\ol S \to \ol B$   of semi-stable curves representing $C$;
in this case, the existence of  the ramification locus  $\Lambda$ of the  Torelli  morphism $j_B: B\to C$
is the main difficulty and we need a modified version of the above two types of inequalities.

\begin{theorem}[Miyaoka-Yau type inequality II, cf. Section\,\ref{sectionpfofupper}]\label{thmupper1}
Let $\bar f:\, \ol S \to \ol B$ be the family of semi-stable genus-$g$ curves representing
a smooth closed curve $C\Subset\calt_{g}$  such that $C \nsubseteq \calth_g$.
For any $p\in\ol B$, let $F_p=f^{-1}(p)$. If $g\geq 7$, then
\begin{equation}\label{eqnupper1}
\begin{aligned}
\hspace{-0.3cm}\omega_{\ol S/\ol B}^2 ~\leq&~ (2g-2)\cdot \deg\Omega^1_{\ol B}(\log\Delta_{nc})+\\[0.1cm]
&\hspace{-0.3cm}\sum_{p\in \Delta_{ct}\,\cap\,\Lambda}\frac32\big(l_h(F_p)+l_1(F_p)-1\big)
+\sum_{p\in \Delta_{ct} \setminus \Lambda}\big(3l_h(F_p)+2l_1(F_p)-3\big).
\end{aligned}
\end{equation}
Moreover, if $\Delta_{nc}\neq \emptyset$ or $\Delta=\emptyset$, then the above inequality is strict.
\end{theorem}

\begin{theorem}[Sharp slope inequality II, cf. Section\,\ref{sectionlower1}]\label{thmlower1}
Let $\bar f:\,\ol S \to \ol B$ be the same as in Theorem {\rm\ref{thmupper1}}.
If $g\geq 3$ and $\bar f_*\omega_{\ol S/\ol B}$ is a semi-stable vector bundle, then
\begin{equation}\label{eqnlower1}
\begin{aligned}
\hspace{-0.3cm}\omega_{\ol S/\ol B}^2 ~\geq&~ \frac{5g-6}{g}\deg \bar f_*\omega_{\ol S/\ol B}
+2(g-2)\cdot |\Lambda|+\\[0.1cm]
&\hspace{-0.3cm}\sum_{p\in \Delta_{ct}\,\cap\,\Lambda} 2\big(l_h(F_p)+l_1(F_p)-1\big)
+\sum_{p\in \Delta_{ct} \setminus \Lambda}\big(3l_h(F_p)+2l_1(F_p)-3\big).
\end{aligned}
\end{equation}
\end{theorem}
The proof of Theorem \ref{thmupper1} is the same as that of Theorem \ref{thmupper2},
while Theorem \ref{thmlower1} is proved relying on the derivative of the Torelli morphism,
i.e. the second multiplication map
$\varrho:~ S^2\left(\bar f_*\omega_{\ol S/\ol B}\right) \ra \bar f_*\big(\omega_{\ol S/\ol B}^{\otimes2}\big)$.

\subsection{The strict Arakelov inequalities}\label{sectionpfofArak}
In this subsection, we prove Theorem\,\,\ref{theoremstrictarak}.

As explained at the end of Section \ref{sectionmainidea},
the main technique is the theory of fibred surfaces.
Given a smooth closed curve $C \Subset \calt_g\subseteq \cala_g$,
we have constructed a family $\bar f:\,\ol S\to \ol B$ of semi-stable curves
representing $C$ in Section \ref{sectionrepresent}.
We have also established in Proposition \ref{proprelationHiggs}
the relation between the associated Higgs bundles over $\ol B$ and $\ol C$,
where $\ol C$ is a suitable smooth compactification of $C$ by joining a finite set of cusps $\Delta_{\ol C}$.
Hence one can first prove a strict Arakelov inequality for the family $\bar f$,
and then derive the inequality on $C$ by using Proposition \ref{proprelationHiggs}.

However, it turns out that the proof of \eqref{eqnstrictarak1} will be more complicated
when the ramification locus $\Lambda$ of the double cover $j_B:\,B \to C$ is not empty.
To illustrate the idea, we consider first the easier case $\Lambda=\emptyset$,
in which case the invariants involved on $\ol B$ and $\ol C$ are all proportional by Proposition \ref{proprelationHiggs}.
Hence \eqref{eqnstrictarak1} follows easily from the following strict Arakelov inequality
for a family of semi-stable curves.

\begin{theorem}\label{theoremstrictarakfamily}
Let $\bar f:\,\ol S \to \ol B$ be a non-isotrivial family of semi-stable curves of genus $g>4$. Then
\begin{equation}\label{eqnstrictarak3}
\deg \bar f_*\omega_{\ol S/\ol B} < {g\over 2}\cdot\deg\Omega^1_{\ol B}(\log\Delta_{nc}).
\end{equation}
\end{theorem}
\begin{proof}
Note that $0\leq \delta_1(\Upsilon_{ct}) \leq \delta_1(\Upsilon)$ and
$0\leq \delta_h(\Upsilon_{ct}) \leq \delta_h(\Upsilon)$.
Hence by \eqref{eqnupper2} and \eqref{eqnlower2},
one gets
\begin{equation}\label{eqnpfarak3}
\deg \bar f_*\omega_{\ol S/\ol B} \leq {g\over 2}\cdot\deg\Omega^1_{\ol B}(\log\Delta_{nc})
-\frac{g-4}{g}\cdot\big(\delta_1(\Upsilon)+4\delta_h(\Upsilon)\big)
\end{equation}
Thus \eqref{eqnstrictarak3} follows if either $\delta_1(\Upsilon)>0$ or $\delta_h(\Upsilon)>0$.
Suppose $\delta_1(\Upsilon)=\delta_h(\Upsilon)=0$. Then either $\Delta_{nc}\neq\emptyset$ or $\Delta=\emptyset$.
Therefore \eqref{eqnpfarak3} is strict since \eqref{eqnupper2} is so, which implies that
\eqref{eqnstrictarak3} also holds in this case.
\end{proof}

Now we are going to prove the strict Arakelov inequalities in general case.
The proof of  \eqref{eqnstrictarak1} requires a strong version of Arakelov inequality
for the family $\bar f$ by taking the ramification locus $\Lambda$ into account; and the proof of \eqref{eqnstrictarak2} relies highly on the strong slope inequality \eqref{eqnlower3}
with positive relative irregularity.

\begin{proof}[Proof of Theorem {\rm \ref{theoremstrictarak}}]
(i). We prove \eqref{eqnstrictarak1} by contradiction.
Suppose \eqref{eqnstrictarak1} does not hold.
According to \cite{faltings83}, one has the following Arakelov inequality:
\begin{equation}\label{eqnpfarak11}
\deg E_{\ol C}^{1,0} \leq \frac{g}{2}\cdot \deg \Omega_{\ol C}^1(\log \Delta_{\ol C}).
\end{equation}
Hence we may assume that the equality holds in \eqref{eqnpfarak11}.
By Corollary \ref{propnonhyper}, in order to derive a contradiction,
it is necessary and sufficient to prove the following stronger version of Arakelov inequality for the family $\bar f$.
\begin{equation}\label{eqnpfarak12}
\deg \bar f_*\omega_{\ol S/\ol B} < \frac g2 \cdot\Big(\deg\Omega^1_{\ol B}\left(\log\Delta_{nc}\right)-|\Lambda|\Big).
\end{equation}

The equality of \eqref{eqnpfarak11} means that
the associated Higgs bundle $\left(E_{\ol C}^{1,0}\oplus E_{\ol C}^{0,1},\,\theta_{\ol C}\right)$
has strictly maximal Higgs field (cf. \eqref{arakelovC}).
Hence by \cite[Proposition\,1.2]{vz04}, $E_{\ol C}^{1,0}$ is poly-stable;
in particular, it is semi-stable.
According to \eqref{xinpullback} and \cite[Lemma\,6.4.12]{lazarsfeld},
$\bar f_*\omega_{\ol S/\ol B}$ is also semi-stable.
Hence one can apply Theorems \ref{thmupper1} and \ref{thmlower1} to the family $\bar f$, and obtains
\begin{equation}\label{pfarak11}
\deg \bar f_*\omega_{\ol S/\ol B} \leq
\frac{2(g-1)g}{5g-6}\cdot\Big(\deg\Omega^1_{\ol B}\left(\log\Delta_{nc}\right)-|\Lambda|\Big)
+\frac{2g}{5g-6}\cdot|\Lambda|.
\end{equation}
Since $\delta_{\bar f}\geq 0$, according to \eqref{formulanoether} and \eqref{eqnlower1}, one gets
\begin{equation}\label{pfarak12}
 |\Lambda| \leq \frac{7g+6}{2(g-2)g} \cdot\deg \bar f_*\omega_{\ol S/\ol B}.
\end{equation}
Combing \eqref{pfarak11} with \eqref{pfarak12}, one obtains
\begin{eqnarray*}
\deg \bar f_*\omega_{\ol S/\ol B} &\leq&
\frac{2(g-1)(g-2)g}{5g^2-23g+6}\cdot\Big(\deg\Omega^1_{\ol B}\left(\log\Delta_{nc}\right)-|\Lambda|\Big)\\
&=&\left(\frac g2-\frac{(g^2-11g+2)g}{2(5g^2-23g+6)}\right)
      \cdot\Big(\deg\Omega^1_{\ol B}\left(\log\Delta_{nc}\right)-|\Lambda|\Big).
\end{eqnarray*}
Since $g>11$, \eqref{eqnpfarak12} follows.
The proof is complete.

\vspace{0.2cm}
(ii).
To start the proof, one needs a statement   below to compare $\rank F_{\ol B}^{1,0}$ with $q_{\bar f}$,
whose proof is postponed to Section \ref{sectionflathyper}.
\begin{theorem}\label{thmFtrivial}
Let $\bar f:\,\ol S \to \ol B$ be a non-isotrivial family of semi-stable hyperelliptic curves of genus $g\geq 2$.
Then after passing to a finite \'etale base change, one has
\begin{equation}\label{eqnFtrivial}
\rank F_{\ol B}^{1,0}=q_{\bar f}.
\end{equation}
\end{theorem}
By the above theorem, we may assume $q_{\bar f}=\rank F_{\ol B}^{1,0}=g-\rank A_{\ol B}^{1,0}$.
Combining \eqref{eqnupper2} with \eqref{eqnlower3}, we obtain that
if $\Delta_{nc}\neq \emptyset$, then
\begin{equation}\label{eqnpfarak21}
\deg \bar f_*\omega_{\ol S/\ol B}<\frac{\rank A^{1,0}_{\ol B}}{2}\cdot \deg \Omega_{\ol B}^1(\log\Delta_{nc})
-\big(\alpha_1\delta_1(\Upsilon)+\alpha_h\delta_h(\Upsilon)\big),
\end{equation}
and if $\Delta_{nc}= \emptyset$, then
\begin{equation}\label{eqnpfarak22}
\deg \bar f_*\omega_{\ol S/\ol B}\leq \frac{\rank A^{1,0}_{\ol B}}{2}\cdot \deg \Omega_{\ol B}^1(\log\Delta_{nc})
-\sum_{i=1}^{[g/2]}\beta_i\delta_i(\Upsilon),
\end{equation}
where
\begin{eqnarray*}
\alpha_1&=& \frac{g^2-(6q_{\bar f}+3)g+12q_{\bar f}-4}{4(g+1)(g-1)},\\
\alpha_h&=& \frac{4g^2-(13q_{\bar f}+12)g+37q_{\bar f}-16}{4(g+1)(g-1)},\\
\beta_1&=& \frac{2g+1-3q_{\bar f}}{2g+1}-\frac{3(g-q_{\bar f})}{4(g-1)},\\
\beta_i&=& \frac{(2g+1-3q_{\bar f})i(g-i)}{(2g+1)(g-1)}-\frac{g-q_{\bar f}}{g-1},\qquad \forall~2\leq i\leq [g/2].
\end{eqnarray*}
Since $\bar f$ is hyperelliptic, one gets (cf. \cite[Proposition 4.7]{ch88})
\begin{equation}\label{bothnot=0}
\text{
$\delta_i(\Upsilon)$'s are non-negative, and one of them is positive if  $\Delta_{nc} =\emptyset$.}
\end{equation}
Hence it is reasonable to imagine that the strict Arakelov inequality
\eqref{eqnstrictarak2} holds when  $g$ is large enough.
The detailed proof is divided into the two following cases:\vspace{0.15cm}

{\noindent\bf Case I.~} $\Delta_{nc}\neq\emptyset$.~~~
We prove by contradiction in the case.

Assume that the strict Arakelov inequality \eqref{eqnstrictarak2} does not hold.
Then by \eqref{eqnhiggshyper} and the classical Arakelov inequality (cf. \cite{faltings83}),
we must have
$$\deg \bar f_*\omega_{\ol S/\ol B}=\deg E_{\ol C}^{1,0}=\frac{\rank A^{1,0}_{\ol B}}{2}\cdot \deg \Omega_{\ol B}^1(\log\Delta_{nc})
=\frac{g-q_{\bar f}}{2}\cdot \deg \Omega_{\ol B}^1(\log\Delta_{nc}).$$
Combing this with \cite[Corollary\,1.7]{ltyz13} and the discussion after it,
one gets that $g(F)=q_{\bar f}$ for every fibre $F$ over $\Delta_{nc} (\neq \emptyset)$,
where $g(F)$ is the geometrical genus of $F$.
Hence by Proposition \ref{q_f>0fibred}, proved later, we obtain that $q_{\bar f}\leq 1$,
from which it follows that the coefficients $\alpha_1$ and $\alpha_h$ in \eqref{eqnpfarak21} are positive since $g>7$.
By \eqref{bothnot=0}, this is a contradiction.

\vspace{0.15cm}
{\noindent\bf Case II.~} $\Delta_{nc}=\emptyset$.\quad
First, we claim that $q_{\bar f} \leq \frac{g-1}{2}$ in this case.
Indeed, by \cite[Theorem\,1]{xiao92-0} or Proposition \ref{q_f>0fibred},
for any hyperelliptic family, one has $q_{\bar f}\leq \frac{g+1}{2}$,
and if the equality holds then $\bar f$ is isotrivial.
Hence $q_{\bar f} \leq g/2$. However, if $q_{\bar f}=g/2$,
then $\omega_{\ol S/\ol B}^2=\frac{8(g-1)}{g}\cdot\deg \bar f_*\omega_{\ol S/\ol B}$
by \cite[Theorem\,2(a)]{xiao92-0}.
Combining this with \eqref{eqnlower3}, one obtains that $\delta_i(\Upsilon)=0$ for all $i\geq 1$,
which contradicts \eqref{bothnot=0}.

If $q_{\bar f}< \frac{(g-4)(2g+1)}{3(2g-5)}$, then
it is easy to show that $\beta_i > \beta_1 >0$ for any $2\leq i\leq [g/2]$.
Hence \eqref{eqnstrictarak2} follows from \eqref{eqnpfarak22} together with \eqref{bothnot=0}.

Thus we may assume $q_{\bar f} \geq \frac{(g-4)(2g+1)}{3(2g-5)}$, from which it follows that
$\beta_1\leq 0$ and $q_{\bar f}\geq 2$ since $g>7$.
So according to \eqref{eqnlower3'} and \eqref{eqnpfarak22}, we obtain
\begin{equation}\label{eqnpfarak23}
\deg \bar f_*\omega_{\ol S/\ol B}\leq \frac{\rank A^{1,0}_{\ol B}}{2}\cdot \deg \Omega_{\ol B}^1(\log\Delta_{nc})
-\left(\sum_{i=2}^{q_{\bar f}-1} \xi_i\delta_i(\Upsilon)
       +\sum_{i=q_{\bar f}}^{[g/2]} \eta_i\delta_i(\Upsilon)\right),
\end{equation}
where
$$\begin{aligned}
\xi_i&\,=\,-\frac{i(2i+1)}{3}\cdot\beta_1+\beta_i,&\quad&\forall~2\leq i \leq q_{\bar f}-1;\\
\eta_i&\,=\,\frac{(2i+1)(2g+1-2i)}{12(g+1)}\cdot\beta_1+\beta_i,&&\forall~q_{\bar f}\leq i \leq [g/2].
\end{aligned}
$$
Combining \eqref{eqnpfarak23} with \eqref{eqnlower3'} and \eqref{bothnot=0},
it suffices to prove that $\xi_i$'s and $\eta_i$'s are all positive.

It is easy to see that $\xi_i>0$, since $\beta_1\leq 0$, and $\beta_i\geq \beta_2 >0$ for $2\leq i \leq q_{\bar f}-1$.
Let $\Theta=(g-4)(2g+1)-3(2g-5)q_{\bar f}$.
Then for $q_{\bar f}\leq i\leq [g/2]$, one has
$$\begin{aligned}
\eta_i\,&\,=\,\frac{1}{(g-1)(2g+1)}\cdot\left((2g+1-3q_{\bar f})+\frac{\Theta}{12(g+1)}\right) \cdot i(g-i)\\
      &\qquad\quad   -\frac{g-q_{\bar f}}{2g+1}+\frac{\Theta\cdot(2g+1)}{48(g-1)(2g+1)(g+1)}\\
\geq\,\eta_{q_{\bar f}}
      &\,=\, \frac{g-q_{\bar f}}{48(g-1)(g+1)}\cdot\Bigg(4q_{\bar f}(13g-21q_{\bar f}+8)-50g-51\\
&\hspace{4cm}+\frac{4\big((q_{\bar f}-1)+(g-2q_{\bar f})(g-1)\big)}{g-q_{\bar f}}\Bigg)\\
&\,\geq\, \frac{g-q_{\bar f}}{48(g-1)(g+1)}\cdot\Big( 4q_{\bar f}(13g-21q_{\bar f}+8)-50g-51\Big).
\end{aligned}
$$
Note that for $2\leq q_{\bar f} \leq (g-1)/2$, we have
$$\begin{aligned}
&4q_{\bar f}(13g-21q_{\bar f}+8)-50g-51\\
\geq~& \min\left\{4\cdot 2\cdot(13g-21\cdot 2+8),~
4\cdot\frac{g-1}{2}\cdot\Big(13g-21\cdot\frac{g-1}{2}+8\Big)\right\}-\,50g-51\\
>~&0,\hspace{2cm}\text{since~}g>7.
\end{aligned}$$
Hence for $q_{\bar f}\leq i \leq [g/2]$, $\eta_i>0$.
This completes the proof.
\end{proof}

\subsection{Proofs of the main results}\label{sectionpfofmainre}
Theorems \ref{mainthm1} and \ref{mainthm3} are immediate consequence of Theorems
\ref{theoremnumchar} and \ref{theoremstrictarak};
and Theorem \ref{mainthm2} follows from the definition (cf. Definition \ref{defmaximalB})
and Theorem \ref{theoremstrictarakfamily}.
Corollary \ref{corE-S} follows from
Theorem \ref{mainthm1} together with Lemma \ref{algebraic interpretation of Shimura curves of type I}.

Finally, we prove Corollary \ref{real multiplication} as follows.
Let $Z$ be as in Corollary \ref{real multiplication}.
Note that any Hecke translate of a Shimura curve of type I is still a Shimura curve of type I.
Hence by Remark \ref{remarkheckedense} and Definition \ref{deftypeI}, $Z$ contains a Shimura curve $C$ of type I.
If moreover $Z\Subset \calt_g$, then one can even assume $C\Subset \calt_g$ (cf. Remark \ref{remarkheckedense}).
Therefore our corollary follows from Theorem \ref{mainthm1}.

\section{Miyaoka-Yau type inequalities for a family of semi-stable curves}\label{sectionupper}
The section is aimed to prove a Miyaoka-Yau type inequality for families of semi-stable curves;
as a consequence, we will complete the proofs of Theorems \ref{thmupper2} and \ref{thmupper1}.
\begin{theorem}\label{theoremupperlinshi}
Let $\bar f:\, \ol S \to \ol B$ be a family of semi-stable curves of genus $g\geq 2$.
For any $p\in\ol B$, let $F_p=f^{-1}(p)$ and
$$\begin{aligned}
\Lambda'\,=\,&\left\{p\in \ol B ~\Bigg|
\begin{aligned}
&F_p\text{~is a hyperelliptic curve with compact Jacobian, and any}\\
&\text{two irreducible irrational components of $F_p$ do not intersect}
\end{aligned}
~\right\};\\[0.1cm]
\Delta_{ct,b}\,=\,&\left\{p\in \Delta_{ct}\setminus \Lambda' ~\big|~
\text{each irreducible component of $F_p$ is of genus $0$ or $1$}\right\};\\[0.1cm]
\Delta_{ct,ub}\,=\,&\Delta_{ct}\setminus\left\{\Lambda'\cup\Delta_{ct,b}\right\}.
\end{aligned}$$
Then
\begin{equation}\label{eqnupperlinshi}
\begin{aligned}
\omega_{\ol S/\ol B}^2 ~\leq~& (2g-2)\cdot \deg\left(\Omega^1_{\ol B}(\log\Delta_{nc})\right)
+\sum_{p\in \Delta_{ct} \cap \Lambda' }\frac32\cdot\big(l_h(F_p)+l_1(F_p)-1\big)\\
&+\left(2g-2-\frac{g-1}{6}\right)\cdot |\Delta_{ct,b}|
+\sum_{p\in \Delta_{ct,ub}}\big(3l_h(F_p)+2l_1(F_p)-3\big).
\end{aligned}
\end{equation}
Moreover, if $\Delta_{nc}\neq \emptyset$ or $\Delta=\emptyset$, then the above inequality is strict.
\end{theorem}

We  first prove in Section \ref{sectionpfofupper} Theorems \ref{thmupper2} and \ref{thmupper1}
based on Theorem \ref{theoremupperlinshi}, which will be proved in Section \ref{sectionpfofupperlinshi}.

\subsection{Proof of Theorems \ref{thmupper2} and \ref{thmupper1}
based on Theorem \ref{theoremupperlinshi}}\label{sectionpfofupper}
\begin{proof}[Proof of Theorem {\rm\ref{thmupper2}}]
By Theorem \ref{theoremupperlinshi}, it suffices to prove
\begin{eqnarray}
\frac32\cdot\big(l_h(F_p)+l_1(F_p)-1\big) &<&
    2\delta_1(F_p)+3\delta_h(F_p),\qquad \forall~p\in \Delta_{ct} \cap \Lambda';\label{eqnpfupper21}\\[0.1cm]
2g-2-\frac{g-1}{6} &<&
    2\delta_1(F_p)+3\delta_h(F_p),\qquad \forall~p\in \Delta_{ct,b};\label{eqnpfupper22}\\[0.1cm]
3l_h(F_p)+2l_1(F_p)-3 &\leq&
    2\delta_1(F_p)+3\delta_h(F_p),\qquad \forall~p\in \Delta_{ct,ub}.\label{eqnpfupper23}
\end{eqnarray}

Actually, \eqref{eqnpfupper21} follows directly from \eqref{eqnFcompJac}.
Similarly for \eqref{eqnpfupper22} if one notes $l_1(F_p)=g$ and $l_i(F_p)=0$ for $i>1$ when $p\in \Delta_{ct,b}$.
In order to prove \eqref{eqnpfupper23},
we first claim that for any $p\in \Delta_{ct,ub}$, we have
\begin{equation}\label{eqnpfupper24}
l_h(F_p)-1\leq \delta_h(F_p).
\end{equation}
Indeed, let
$$\cals_p=\left\{q \in F_p~\Bigg|~
\begin{aligned}
&\text{$q$ is a node of $F_p$, and each of the two connected components}\\
&\text{of $F_p\setminus q$ contains an irreducible component of genus $\geq 2$}
\end{aligned}\right\}.$$
Since $F_p$ is connected, we get $l_h(F_p)-1 \leq |\cals_p|$.
It is clear that $|\cals_p| \leq \delta_h(F_p)$.
Hence \eqref{eqnpfupper24} follows.
Now again by \eqref{eqnFcompJac}, one gets for any $p\in \Delta_{ct,ub}$,
\begin{equation*}
\begin{aligned}
3l_h(F_p)+2l_1(F_p)-3&=\big(l_h(F_p)-1\big)+2\big(l_1(F_p)+l_h(F_p)-1\big)\\[0.1cm]
&\leq \delta_h(F_p)+2\big(\delta_1(F_p)+\delta_h(F_p)\big)\\[0.1cm]
&=2\delta_1(F_p)+3\delta_h(F_p).
\end{aligned}
\end{equation*}
This completes the proof.
\end{proof}

\begin{proof}[Proof of Theorem {\rm\ref{thmupper1}}]
Let $\bar f$ be as in Theorem \ref{thmupper1}.
Then by the definition, it is easy to show that
$$\begin{aligned}
&2g-2-\frac{g-1}{6}  \leq 2g-3=3l_h(F_p)+2l_1(F_p)-3, &~~& \text{if~}p\in \Delta_{ct,b} \text{~and~}g\geq 7;\\
&\frac32\cdot\big(l_h(F_p)+l_1(F_p)-1\big) \leq
3l_h(F_p)+2l_1(F_p)-3,&&\text{if~}p\in \Delta_{ct} \cap \Lambda'\text{~and~}g\geq 3.
\end{aligned}$$
Thus by Theorem \ref{theoremupperlinshi}, it suffices to prove $\Lambda \subseteq \Lambda'$.
For this, it is enough to prove
\begin{lemma}\label{lemmapfofupper1}
For  $p\in \Lambda$, the fibre $F_p$ is hyperelliptic; if moreover $F_p$ is singular, then any two
irreducible irrational components of $F_p$ do not intersect.
\end{lemma}
To prove the above lemma, we may assume that $\Lambda\neq\emptyset$.
Consider the following Cartesian diagram:
$$\xymatrix@M=0.15cm{
B\ar[rr]^-{normalization}&&j^{-1}(C)\ar[rr]\ar@{^(->}[d] && C \ar@{^(->}[d]^{\Subset} \\
&&\Mcal_g^{ct}\ar[rr]^{j}  && \calt_g
}$$
Note that the family $f:\,S \to B$ is obtained by pulling back the universe family over $\calm_g^{ct}$
and resolving the singularities.
Hence by Lemma \ref{existprop}, there exists an involution $\sigma$ (resp. $\tau$) of $S$ (resp. $B$),
such that $f\circ\sigma=\tau\circ f$ and the fixed locus of $\tau$ is ${\rm Fix}(\tau)=\Lambda$.
Furthermore, for  $p\in \Lambda$, the fibre $F_p$ is a
hyperelliptic curve and the restricted involution
$$\sigma|_{F_p}:~F_p \to F_p$$
is the hyperelliptic involution of $F_p$.
It remains to prove the last statement.

Assume that there are two irrational components $D_1$ and $D_2$ of $F_p$ with an intersection point $q$.
Locally, we may assume that $D_1$ (resp. $D_2$) is defined by $x=0$ (resp. $y=0$),
and $f$ is given by $t=f(x,y)=xy$,
where $t$ \big(resp. $(x,y)$\big) is a local coordinate of $B$ (resp. $S$) around $p$ (resp. $q$).

Since the involution $\sigma|_{F_p}$ is the hyperelliptic involution of $F_p$,
$F_p\big/\langle\sigma|_{F_p}\rangle$ is a (may be singular) rational curve.
In particular, $\sigma$ keeps both $D_1$ and $D_2$ invariant, and $\sigma|_{D_1}$ (resp. $\sigma|_{D_2}$)
is not the identity on $D_1$ (resp. $D_2$),
since both $D_1$ and $D_2$ are irrational curves.
The first implies that there exist non-zero functions $\xi(x,y)$ and $\eta(x,y)$ such that
$$\sigma^*(x)=\xi x,\quad \sigma^*(y)=\eta y, \qquad
\text{with $\xi_0:=\xi(0,0)\neq 0$ and $\eta_0:=\eta(0,0)\neq 0$.}$$
Since $\sigma$ is an involution, we get that
$\sigma^*\big(\sigma^*(x)\big)=x$ and $\sigma^*\big(\sigma^*(y)\big)=y$,
from which it follows that $\xi_0^2=\eta_0^2=1$.
Let
$$\tilde x= x+\frac{\xi}{\xi_0}\cdot x,\qquad \tilde y= y+\frac{\eta}{\eta_0}\cdot y,
\quad\text{~and~}\quad \tilde t = \frac{(\xi_0+\xi)(\eta_0+\eta)}{\xi_0\eta_0}\cdot t.$$
It is easy to see that $\tilde t$ \big(resp. $(\tilde x,\tilde y)$\big)
can be viewed as a local coordinate of $B$ (resp. $S$) around $p$ (resp. $q$).
Moreover, $D_1$ (resp. $D_2$) is locally defined by $\tilde x=0$ (resp. $\tilde y=0$), $\bar f$ is given by
$\tilde t=\tilde x\tilde y$, and
$$\begin{aligned}
&\sigma^*(\tilde x)=\sigma^*\left(x+\frac{1}{\xi_0}\sigma^*(x)\right)
=\sigma^*(x)+\frac{1}{\xi_0}x=\xi_0\cdot \tilde x,\\
&\sigma^*(\tilde y)=\sigma^*\left(y+\frac{1}{\eta_0}\sigma^*(y)\right)
=\sigma^*(y)+\frac{1}{\eta_0}y=\eta_0\cdot \tilde y
\end{aligned}$$
As $\sigma|_{D_1}$ (resp. $\sigma|_{D_2}$) is not the identity on $D_1$ (resp. $D_2$),
one gets that $\eta_0\neq 1$ (resp. $\xi_0\neq 1$),
since $y|_{D_1}$ \big(resp. $x|_{D_2}$\big) is a local coordinate of $D_1$ (resp. $D_2$).
Hence $\xi_0=\eta_0=-1$.

Note that $\tau\circ f=f\circ \sigma$.
So
$$f^*\tau^*(\tilde t)=\sigma^*f^*(\tilde t)=\sigma^*(\tilde x\tilde y)
=\tilde \xi_0\tilde \eta_0\cdot \tilde x\tilde y=\tilde x\tilde y=f^*(\tilde t).$$
Since $\bar f$ is surjective, one gets $\tau^*(\tilde t)=\tilde t$,
which implies that $\tau$ is the identity map of $B$ around $p$.
It is a contradiction. This completes the proof of Lemma \ref{lemmapfofupper1}
and hence also Theorem \ref{thmupper1}.
\end{proof}

\subsection{Proof of Theorem \ref{theoremupperlinshi}}\label{sectionpfofupperlinshi}
Our proof of Theorem\,\ref{theoremupperlinshi} is
based on a generalized Miyaoka-Yau's inequality (cf. Theorem \ref{theoreminpfupper1}),
by choosing a suitable base change and  suitable components contained in singular fibres
(but not the entire singular fibres).

Recall from \cite{miyaoka84}  the generalized Miyaoka-Yau's theorem.
Let $X_x$ be the germ of a quotient singularity of $(\mathbb C^2/G_x)_0$ (in the analytic sense),
where $G_x$ is a finite subgroup of ${\rm GL}(2,\mathbb C)$ with the origin $0$ being its unique fixed point.
Let $X_E$ be the minimal resolution of $X_x$ and $E$ the exceptional divisor
(= the inverse image of $x$). Let
\begin{equation}\label{eqninpfupper1}
v(x)\triangleq \chit(E)-\frac{1}{|G_x|}.
\end{equation}
\begin{theorem}[Miyaoka\,\,{\cite[Corollary\,1.3]{miyaoka84}}]\label{theoreminpfupper1}
Let $X^{\#}$ be a projective surface with only rational double singularities, and $\mathcal I$ the
singular locus of $X^{\#}$. Let $D$ be a reduced normal crossing curve
which lies on the smooth part of $X^{\#}$. Let $X$
be the minimal resolution of $X^{\#}$.
Assume that $\mathcal O_{X}\big(K_{X}\big)$ is numerically effective.
Then
\begin{equation}\label{miyaokayuan}
\sum_{x\in \mathcal I}v(x) \leq \chit(X)-\chit(D)- \frac13(\omega_{X}+ D)^2.
\end{equation}
\end{theorem}

Note that for a singularity $x$ of type $A_k$,
the invariant $v(x)$ defined in \eqref{eqninpfupper1} is equal to $(k+1)-\frac{1}{k+1}$.
Therefore we get
\begin{theorem}\label{theoreminpfupper2}
Let conditions be the same as that of Theorem $\ref{theoreminpfupper1}$.
Assume that each point $x\in \mathcal I$
is a quotient singularities of type $A_{k_x}$,
$X$ is minimal and of general type.
Then
\begin{equation*}
\sum_{x\in \mathcal I}\left(3(k_x+1)-\frac{3}{k_x+1}\right)
\leq 3\left(\chit(X)-\chit(D)\right)-(\omega_{X}+ D)^2.
\end{equation*}
\end{theorem}

\v

\begin{proof}[Proof of Theorem {\rm\ref{theoremupperlinshi}}]
Let
\begin{eqnarray}
\hspace{-1cm}E_p&\hspace{-0.2cm}=&\hspace{-0.2cm}
\left\{\sum_{j} E_{p,j}~\,\Big|~ E_{p,j}\subseteq F_p \text{~is a~}
(-2)\text{-curve }\right\},\qquad\qquad \forall~p\in\Delta \setminus \Delta_{ct,ub};\label{MIYAOK13}\\[.1cm]
\hspace{-1cm}D_p&\hspace{-0.2cm}=&\hspace{-0.2cm}\left\{\sum_{j} D_{p,j}~\Big|~
D_{p,j}\subseteq F_p \text{~with~genus~} g\big(D_{p,j}\big)=0\text{~or~}1
\right\}, ~~\forall~p\in\Delta_{ct,ub}.\label{MIYAOK14}
\end{eqnarray}
Let $\ol S \to \ol S^{\#}$ be the contraction of
$$\sum\limits_{p\,\in\,\Delta \setminus \Delta_{ct,ub}} E_p\subseteq \ol S,$$
and $f^{\#}:\,\ol S^{\#} \to \ol B$ the induced morphism.
It is clear that for any $p\in\Delta_{ct,ub}$, the image of $D_p$ on $\ol S^{\#}$ lies on
the smooth part of ${\ol S^{\#}}$,
which we still denote by $D_p$.
For any singular point $q$ of $\ol S^{\#}$,
$(\ol S^{\#},\,q)$ is a rational double point of type $A_{\lambda_q}$,
here $\lambda_q$ is the number of $(-2)$-curves in $\ol S$ over $q$.
For convenience, we also denote by $q$ the singular point of the fibres on the smooth
part of $\ol S^{\#}$, in which case, $\lambda_q=0$.
So a singular point $(\ol S^{\#},\,q)$ of type $A_0$ is understood as
a node of the fibres but a smooth point of $\ol S^{\#}$.
For $p\in \ol B$, let $F^{\#}_p$ be the image of $F_p$ on $\ol S^{\#}$.
Then it is clear that
\begin{equation}\label{MIYAOK11}
\delta(F_p)=\sum_{q\in F^{\#}_p}(\lambda_q+1), \qquad \forall~p\,\in\,\Delta \setminus \Delta_{ct,ub}.
\end{equation}

Let $\phi:\,\wt B\to \ol B$ be a cover of $\ol B$ such that
$\phi$ is branched uniformly over $\Delta_{nc}$ (resp. $\Delta_{ct,b}$)
with ramification index equaling to $e_{nc}$ (resp. $e_{ct,b}$).
Such a cover exists. Indeed, by the Kodaira-Parshin construction (cf. \cite{vojta88} and \cite{tan95}),
one can first construct a cover $\phi':\,B'\to \ol B$ branched uniformly over $\Delta_{nc}$
with ramification index equaling to $e_{nc}$;
and then take a cover $\wt B \to B'$ branched uniformly over $(\phi')^{-1}(\Delta_{ct,b})$
with ramification index equaling to $e_{ct,b}$.
Then the composition $\phi:\,\wt B \to \ol B$ satisfies our requirements.
Let $\deg \phi=d$.
Then according to Hurwitz formula, one gets
\begin{equation}\label{MIYAOK2}
2\left(g(\wt B)-1\right)=d\cdot
\left(2\big(g(\ol B)-1\big)+\frac{e_{nc}-1}{e_{nc}}\cdot \left|\Delta_{nc}\right|
+\frac{e_{ct,b}-1}{e_{ct,b}}\cdot \left|\Delta_{ct,b}\right|\right).
\end{equation}

Let $\wt S^{\#}=\wt B \times_{\ol B} \ol S^{\#}$ be the fibre-product,
and $\wt S \to \wt S^{\#}$ the minimal resolution of singularities.
We have the following commutative diagram:
\begin{center}\mbox{}
\xymatrix{
\wt S \ar@/_7mm/"3,1"_{\tilde f} \ar[d] \ar[rr]^{\ol\Phi} && \ol S \ar[d] \ar@/^7mm/"3,3"^f \\
\wt S^{\#} \ar[d] \ar[rr]^{\Phi^{\#}} && \ol S^{\#} \ar[d]\\
\wt B \ar[rr]^{\phi} && \ol B}
\end{center}

For $p\in\Delta_{ct}\cap\Lambda'$ (resp. $p\in\Delta_{nc}$,~ resp. $p\in \Delta_{ct,b}$),
the inverse image of a singular point $(\ol S^{\#},\,q)$ of type $A_{\lambda_q}$ with $q\in F^{\#}_p$
is $d$ (resp. $\frac{d}{e_{nc}}$,~ resp. $\frac{d}{e_{ct,b}}$) singular points of type
$A_{\lambda_q}$ (resp. $A_{(\lambda_q+1)\cdot e_{nc}-1}$,~ resp. $A_{(\lambda_q+1)\cdot e_{ct,b}-1}$) in $\wt S^{\#}$.
Let $$D=\sum\limits_{p\in\Delta_{ct,ub}}\big(\Phi^{\#}\big)^{-1}(D_p).$$
Since $\phi$ is unbranched over $\Delta_{ct,ub}$,
$D$ lies on the smooth part of $\wt S^{\#}$ and
\begin{equation}\label{MIYAOK4}
3\chit(D)+2\,\omega_{\wt S}\cdot D+D^2=
d\cdot\sum_{p\in\Delta_{c,ub}} \left(3\chit(D_p)+2\omega_{\ol S}\cdot D_p+D_p^2\right).
\end{equation}
Because $\bar f$ is semi-stable, $\tilde f:\,\wt S \to \wt B$ is also semi-stable, and
\begin{equation}\label{MIYAOK5}
\delta_{\tilde f}=d\cdot \delta_{\bar f},\qquad \omega_{\wt S/\wt B}^2=d\cdot \omega_{\ol S/\ol B}^2.
\end{equation}
It is not difficult to see that $\wt S$ is minimal and of general type if $g(\wt B) \geq 1$,
which is satisfied when $d$ is large enough.
Hence applying Theorem \ref{theoreminpfupper2} to the case
by setting $X^{\#}=\wt S^{\#}$, $X=\wt S$, and $D$ as above,
we get
\begin{eqnarray}
&&\hspace{-0.3cm}3d\cdot\sum_{q\in F^{\#}_p~~  \atop p\in\Delta_{ct}\,\cap\,\Lambda'}
\left((\lambda_q+1)-\frac{1}{\lambda_q+1}\right)+
\frac{3d}{e_{nc}}\cdot \sum_{q\in F^{\#}_p~  \atop p\in\Delta_{nc}}
\left((\lambda_q+1) e_{nc}-\frac{1}{(\lambda_q+1) e_{nc}}\right)\nonumber\\[0.1cm]
&&\hspace{-0.3cm}+\frac{3d}{e_{ct,b}}\cdot \sum_{q\in F^{\#}_p~  \atop p\in\Delta_{ct,b}}
\left((\lambda_q+1)\cdot e_{ct,b}-\frac{1}{(\lambda_q+1)\cdot e_{ct,b}}\right)\nonumber\\[0.1cm]
&\leq& 3\left(\chit(\wt S)-\chit(D)\right)-\left(\omega_{\wt S}+D\right)^2\label{MIYAOK3}\\[0.1cm]
&=& d\cdot \left(3\delta_{\bar f}-\omega_{\ol S/\ol B}^2\right)-d\cdot\sum_{p\in\Delta_{ct,ub}}
\left(3\chit(D_p)+2\omega_{\ol S}\cdot D_p+D_p^2\right)\nonumber\\
&&\hspace{0.3cm}+d\cdot (2g-2)
\left(2\big(g(\ol B)-1\big)+\frac{e_{nc}-1}{e_{nc}}\cdot \left|\Delta_{nc}\right|
+\frac{e_{c,b}-1}{e_{ct,b}}\cdot \left|\Delta_{ct,b}\right|\right).\nonumber
\end{eqnarray}
We use \eqref{defofrelativeinv}, \eqref{MIYAOK2},
\eqref{MIYAOK4} and \eqref{MIYAOK5} in the last step above.
By \eqref{formulaofdelta_f} and \eqref{MIYAOK11}, we have
\begin{equation}\label{MIYAOK8}
\delta_{\bar f}
=\sum_{q\in F^{\#}_p~~  \atop p\in\Delta_{ct}\,\cap\,\Lambda'}(\lambda_q+1)+
\sum_{q\in F^{\#}_p~  \atop p\in\Delta_{nc}} (\lambda_q+1)+
\sum_{q\in F^{\#}_p~  \atop p\in\Delta_{ct,b}} (\lambda_q+1)+
\sum_{p\in\Delta_{ct,ub}} \delta(F_p).
\end{equation}
Combining \eqref{MIYAOK3} with \eqref{MIYAOK8}, one gets
\begin{eqnarray}
\omega_{\ol S/\ol B}^2 &\leq& (2g-2)\cdot\left(2\big(g(\ol B)-1\big)+ \left|\Delta_{nc}\right|
+ \left|\Delta_{ct,b}\right|\right)
+\sum_{q\in F^{\#}_p~~  \atop p\in\Delta_{ct}\,\cap\,\Lambda'}\frac{3}{\lambda_q+1}\qquad\quad \label{MIYAOK15}\\[-0.1cm]
&&+\left(\sum_{q\in F^{\#}_p~  \atop p\in\Delta_{nc}} \frac{1}{(\lambda_q+1)}\right)\cdot\frac{3}{ e_{nc}^2}
- \frac{(2g-2)\cdot\left|\Delta_{nc}\right|}{e_{nc}} \nonumber\\[.1cm]
&&+\sum_{p\in\Delta_{ct,b}}\left(\sum_{q\in F^{\#}_p} \frac{1}{(\lambda_q+1)}\cdot\frac{3}{ e_{ct,b}^2}
- \frac{2g-2}{e_{ct,b}} \right)\nonumber\\[.1cm]
&&+\sum_{p\in\Delta_{ct,ub}}\big(3\delta(F_p)-(3\chit(D_p)+2\omega_{\ol S}\cdot D_p+D_p^2)\big)\nonumber
\end{eqnarray}

We have the following claim, whose proof are given at the end of the section.
\begin{claim}\label{eDgeqdelta}
\begin{enumerate}
\item[(i).] For each $p\in \Delta_{ct}\,\cap\,\Lambda'$,
\begin{equation}\label{MIYAOK11}
\sum_{q\in F^{\#}_p}\frac{1}{\lambda_q+1}\leq\frac12\big(l_h(F_p)+l_1(F_p)-1\big).
\end{equation}
\item[(ii).] For each $p\in \Delta_{ct,b}$,
\begin{equation}\label{MIYAOK9}
\sum_{q\in F^{\#}_p}\frac{1}{\lambda_q+1}< 2g-2.
\end{equation}
\item[(iii).]
For each $p\in \Delta_{ct,ub}$,
\begin{equation}\label{MIYAOK10}
3\delta(F_p)-\big(3\chit(D_p)+2\omega_{\ol S}\cdot D_p+D_p^2\big)\leq 3l_h(F_p)+2l_1(F_p)-3.
\end{equation}
\end{enumerate}
\end{claim}

By taking $e_{ct,b}=6$ and $e_{nc} \to \infty$ in \eqref{MIYAOK15}, we get
the required inequality \eqref{eqnupperlinshi}.
If $\Delta_{nc}\neq \emptyset$,
then letting $e_{nc}$ be large enough, one has
$$\left(\sum_{q\in F^{\#}_p~  \atop p\in\Delta_{nc}} \frac{1}{(\lambda_q+1)}\right)\cdot\frac{3}{ e_{nc}^2}
- \frac{(2g-2)\cdot\left|\Delta_{nc}\right|}{e_{nc}}<0.$$
Hence if $\Delta_{nc}\neq \emptyset$, then the inequality \eqref{eqnupperlinshi} is strict by letting $e_{ct,b}=6$
and $e_{nc}$ be large enough.
Finally, if $\Delta=\emptyset$, then $\bar f$ is a Kodaira family, and
$$\deg\left(\Omega^1_{\ol B}(\log\Delta_{nc})\right)=2g(\ol B)-2.$$
So by \cite[Corollary\,0.6]{liukefeng96}, \eqref{eqnupperlinshi} is also strict in the case.
The proof is complete.
\end{proof}

\begin{remarks}\label{miyaoka12}
(i). If \eqref{eqnupper2} is indeed an equality, i.e.,
\begin{equation}\label{miyaoka9}
\omega_{\ol S/\ol B}^2 = (2g-2)\cdot \deg\left(\Omega^1_{\ol B}(\log\Delta_{nc})\right)
+2\delta_1(\Upsilon_{ct})+3\delta_h(\Upsilon_{ct}),
\end{equation}
then $\Delta_{nc}=\emptyset$; $\Delta_{ct}\cap\Lambda'=\emptyset$ by \eqref{eqnpfupper21};
$\Delta_{ct,b}=\emptyset$ by \eqref{eqnpfupper22};
and $l_0(F_p)=0$ for $p\in \Delta_{ct,ub}$ by \eqref{eqnpfupper23} and its proof.
In particular, $D_p$ contains at most elliptic curves for $p\in \Delta_{ct}=\Delta_{ct,ub}$.
Hence \eqref{miyaoka9} is equivalent to
$$c_1^2\left(\Omega_{\ol S}^1\bigg(\log \Big(\sum\limits_{p\in\Delta_{ct}}D_p\Big)\bigg)\right)
=3c_2\left(\Omega_{\ol S}^1\bigg(\log \Big(\sum\limits_{p\in\Delta_{ct}}D_p\Big)\bigg)\right).$$
It follows that $\ol S\setminus\left(\bigcup\limits_{p\in\Delta_{ct}}D_{p}\right)$
is a ball quotient by \cite{kobayashi} or \cite{mok12}. See  Example \ref{exshimurag=4} for such an example.
\\[0.08cm]
(ii). If one applies Theorem \ref{theoreminpfupper2} directly on the surface $\ol S$ without using base change technique,
then one gets
\begin{equation}\label{eqnupperlinshi1}
\begin{aligned}
\omega_{\ol S/\ol B}^2 ~\leq~& (2g-2)\cdot \deg\left(\Omega^1_{\ol B}(\log\Delta_{nc})\right)
+\sum_{p\in \Delta_{ct} \cap \Lambda' }\frac32\cdot\big(l_h(F_p)+l_1(F_p)-1\big)\\
&+\left(2g-2\right)\cdot |\Delta_{ct,b}|
+\sum_{p\in \Delta_{ct,ub}}\big(3l_h(F_p)+2l_1(F_p)-3\big).
\end{aligned}
\end{equation}
This is enough to imply \eqref{eqnupper2}.
However, we do not know when \eqref{eqnupper2} becomes strict;
and we cannot derive \eqref{eqnupper1} from \eqref{eqnupperlinshi1} due to the possible existence of $\Delta_{ct,b}$.
Recently, Peters has gotten in \cite{peters14} a simplified proof of
the strictness of \eqref{eqnupper2} if $\Delta_{nc}\neq \emptyset$
by using Cheng-Yau's technique instead of the base change technique.
\end{remarks}\vspace{0.1cm}

In the rest part of the section, we prove Claim \ref{eDgeqdelta}.
First we prove an easy lemma.
\begin{lemma}\label{norational}
Assume that $F^{\#}$ is a stable hyperelliptic curve with compact
Jacobian. Then $F^{\#}$ has no rational component.
\end{lemma}
\begin{proof}By \cite[p.\,467]{ch88},  $F^{\#}$ has a semi-stable model $F$ which is an
admissible double cover  $\psi:\,F\to \Gamma$ over a stable
$(2g+2)$-pointed nodal curve $\Gamma$ of arithmetic genus zero.

We claim that the index  of every singular point is odd. Otherwise,
assume that there exists a singular point  $p\in \Gamma$ with even
index $\alpha$. Then $p$ is not a branched point of $\psi$ and its
inverse image consists of two singular points of type 0, which
contradicts with the fact that $F$ has compact Jacobian. As a direct
consequence of all indices being odd, we obtain that all singular
points of $\Gamma$ are branched, and hence the pre-image of any
irreducible component of $\Gamma$ in $F$ is still irreducible.

Let $D\subseteq F$ be an irreducible component and
$D'=\psi(D)\subseteq \Gamma$. Set
$$\Sigma'=\big\{x\in D'~\big|~\text{$x$ is a marked or singular point of $\Gamma$}\big\}.$$
As $\Gamma$ is stable, $|\Sigma'|\geq 3$. From the above discussion,
the restricted map
$$\psi|_{D}:~D \lra D'$$
is a double cover branched exactly over $\Sigma'$. Hence
$$2g(D)-2 = 2\big(2g(D')-2\big)+|\Sigma'| \geq -1,$$
which implies that $g(D)\geq 1$.  Therefore,  $F$ and thus $F^{\#}$
 contain no rational component.
\end{proof}

\begin{proof}[Proof of Claim {\rm \ref{eDgeqdelta}}]~
Let $l(F_p^{\#})$ (resp. $v(F_p^{\#})$) be the number of irreducible components
(resp. nodes) of $F_p^{\#}$.
Then it is clear that $$v(F_p^{\#})=l(F_p^{\#})-1, \qquad \forall~p\in \Delta_{ct}.$$

(i).
Since $p\in \Lambda'$,
$F_p$ is a hyperelliptic curve. Hence by Lemma \ref{norational}, $F_p^{\#}$ contains no rational components.
So $$l(F_p^{\#})=l_h(F_p)+l_1(F_p).$$
Now by the definition of $\Delta_{ct}\,\cap\,\Lambda'$,
each node of $F_p^{\#}$ is a singular point of $\ol S^{\#}$, i.e.,
$$\lambda_q\geq 1,\qquad\qquad \text{for any node~}q\in F_p^{\#}\text{~and~}p\in \Delta_{ct}\,\cap\,\Lambda'.$$
Thus
$$\sum_{q\in F^{\#}_p}\frac{1}{\lambda_q+1}\leq\frac12v(F_p^{\#})=\frac12\big(l_h(F_p)+l_1(F_p)-1\big),
\qquad \forall~p\in \Delta_{ct}\,\cap\,\Lambda'.$$

(ii).
Note that $F_p^{\#}$ is the stable model of $F_p$ for $\in\Delta_{ct,b}$.
Hence the inverse image $C\subseteq F_p$ of any component $C^{\#}\subseteq F_p^{\#}$ has positive intersection
with $\omega_{\ol S}$, i.e., $\omega_{\ol S} \cdot C\geq 1$.
Since $\omega_{\ol S} \cdot F_p=2g-2$,
one has $$l(F^{\#})\leq 2g-2.$$
Therefore
$$\sum_{q\in F^{\#}_p}\frac{1}{\lambda_q+1}\leq  v(F_p^{\#})=l(F_p^{\#})-1<2g-2.$$

(iii).
As $F_p$ has a compact Jacobian for $p\in \Delta_{ct,ub}\subseteq \Delta_{ct}$,
$$3\delta(F_p)=3\big(l_0(F_p)+l_1(F_p)+l_h(F_p)-1\big).$$
So it suffices to prove that
\begin{equation}\label{MIYAOK16}
3\chit(D_p)+2\omega_{\ol S}\cdot D_p+D_p^2 \geq 3l_0(F_p)+l_1(F_p)
\end{equation}

Let $l_0(D_p)$ and $l_1(D_p)$ be the number of irreducible components contained in $D_p$
of genus zero and one respectively.
Since $p\in \Delta_{ct,ub}$,
by the definition of $D_p$ (cf. \eqref{MIYAOK14}), we get
\begin{equation}\label{ctlaim01}
l_0(D_p)=l_0(F_p),\qquad l_1(D_p)=l_1(F_p).
\end{equation}

Let $C_p \subseteq D_p$ be a connected component, and
$l_0(C_p)$ and $l_1(C_p)$ be the number of irreducible components contained in $C_p$
of genus zero and one respectively.
By \eqref{ctlaim01}, in order to prove \eqref{MIYAOK16}, it suffices to prove for each connected component $C_p$,
\begin{equation*}
3\chit(C_p)+2\omega_{\ol S}\cdot C_p+C_p^2 \geq 3l_0(C_p)+l_1(C_p).
\end{equation*}

For this purpose, note that $C_p^2 <0$,
since $C_p \subseteq D_p \subsetneqq F_p$ by the definition of $\Delta_{ct,ub}$.
So
$$3\chit(C_p)+2\omega_{\ol S}\cdot C_p+C_p^2=3l_0(C_p)+l_1(C_p)-1-C_p^2 \geq 3l_0(C_p)+l_1(C_p).$$
The proof is complete.
\end{proof}

\section{Sharp slope inequalities fora  family of semi-stable curves}\label{sectionlower}
\subsection{Proof of Theorem \ref{thmlower2}}\label{sectionlower2}
The inequality \eqref{eqnlower2}
 follows from Moriwaki's sharp slope inequality (cf. \cite[Theorem D]{moriwaki98}), which together with  \eqref{formulanoether} and \eqref{formulaofdelta_f} implies that
$$\begin{aligned}
(8g+4)\deg \bar f_*\omega_{\ol S/\ol B}&\,\geq\,
g\delta_0(\Upsilon)+\sum_{i=1}^{[g/2]}4i(g-i)\delta_i(\Upsilon)\\
&=\,g\left(12\deg \bar f_*\omega_{\ol S/\ol B}-\omega_{\ol S/\ol B}^2\right)
+\sum_{i=1}^{[g/2]}\big(4i(g-i)-g\big)\delta_i(\Upsilon)\\
&\geq\,g\left(12\deg \bar f_*\omega_{\ol S/\ol B}-\omega_{\ol S/\ol B}^2\right)
+(3g-4)\delta_1(\Upsilon)+(7g-16)\delta_h(\Upsilon).
\end{aligned}$$
By rearrangement,  we obtain \eqref{eqnlower2}. \qed

\subsection{Proof of Theorem \ref{thmlower3}}\label{sectionlower3}
By assumption, $\bar f:\,\ol S \to \ol B$ is a non-isotrivial semi-stable family of
hyperelliptic curves of genus $g\geq2$ with relative irregularity $q_{\bar f}=q(\ol S)-g(\ol B)$.
Recall from  \cite[Proposition 4.7]{ch88} a useful formula  given by Cornalba-Harris:
\begin{equation}\label{formulaofdegomega}
\begin{aligned}
\deg \bar f_*\omega_{\ol S/\ol B}\,=&~\frac{g}{4(2g+1)}\xi_0(\Upsilon)\\
&~+\sum_{i=1}^{[g/2]}\frac{i(g-i)}{2g+1}\delta_i(\Upsilon)
+\sum_{j=1}^{[(g-1)/2]}\frac{(j+1)(g-j)}{2(2g+1)}\xi_j(\Upsilon).
\end{aligned}
\end{equation}

 The proof  of Theorem \ref{thmlower3} is given in Section \ref{subsectpfthmlower}.
If $q_{\bar f}=0$, it  follows directly from Noether's formula and  \eqref{formulaofdegomega}.
If $q_{\bar f}>0$,  we first prove a relation among the invariants $\delta_i(\Upsilon)$'s and $\xi_j(\Upsilon)$'s
in Proposition\,\ref{boundofxi_0prop},  based on   the observation that the double cover induced by
the hyperelliptic involution is fibred and the technique of Cornalba-Harris \cite{ch88}.
Then together with \eqref{formulaofdegomega}, we complete the proof.

\subsubsection{Hyperelliptic family with positive relative irregularity}
The main purpose of this subsection is to prove the following
technical proposition.
\begin{proposition}\label{boundofxi_0prop}
Let $\bar f:\,\ol S\to \ol B$ be the same as in Theorem {\rm \ref{thmlower3}}.
If $q_{\bar f}>0$, then
\begin{equation}\label{boundofxi_0eqn}
\begin{aligned}
&\sum_{i= q_{\bar f}}^{[g/2]}\frac{(2i+1)(2g+1-2i)}{g+1}\delta_i(\Upsilon)
+\sum_{j= q_{\bar f}}^{[(g-1)/2]}\frac{2(j+1)(g-j)}{g+1}\xi_j(\Upsilon)\\
\geq~&~ \xi_0(\Upsilon)+\sum_{i= 1}^{q_{\bar f}-1}4i(2i+1)\delta_i(\Upsilon)
+\sum_{j= 1}^{q_{\bar f}-1}2(j+1)(2j+1)\xi_j(\Upsilon).
\end{aligned}
\end{equation}
\end{proposition}

As mentioned before, the key observation is  that
a double cover  $\tilde\pi:\,\wt S \to \wt Y$   of smooth surfaces  is fibred, where $\tilde \pi$ is obtained by
resolving the singular points of a  double cover $\pi:\,\ol S\to \ol Y$ over a ruled surface $\ol Y$,
while $\pi$ is induced by the hyperelliptic involution.
Let $\wt R \subseteq \wt Y$ be the smooth branched divisor of $\tilde \pi$,
$\tilde f:\,\wt S \to \ol B$ and $\tilde h:\, \wt Y \to \ol B$ be composite morphisms fitting into the following diagram.
\setcounter{figure}{0}
\renewcommand{\thefigure}{\arabic{section}.\arabic{subsection}-\arabic{figure}}
\begin{figure}[H]
\begin{center}\mbox{}
\xymatrix{
 \wt S \ar@/_7mm/"3,1"_{\tilde f} \ar@{->}[rr]^-{\tilde\pi} \ar[d] && \wt Y  \ar@{->}[d] \ar@/^7mm/"3,3"^{\tilde h} \\
 \ol S \ar@{->}[rr]^-{\pi} \ar[d]^-{\bar f}   && \ol Y\ar[d]_-{\bar h}\\
 \ol B\ar@{=}[rr]&&\ol B
}\vspace{-0.3cm}
\end{center}
\caption{Hyperelliptic involution.\label{hyperellipticdiagram}}
\end{figure}
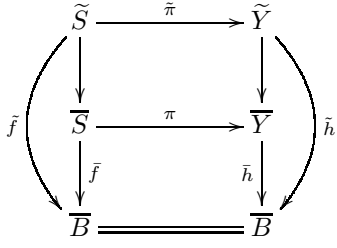

\begin{definition}[\cite{khashin83}]\label{definitonfibred}
A double cover $\pi:\, X\to X'$ of smooth projective surfaces with branched divisor
$R' \subseteq X'$ is called fibred if there exists a double cover $\pi':\, D \to D'$ of smooth projective
curves,  morphisms $p:\, X\to D $ and $p':\, X' \to D'$
with connected fibres, such that the diagram
\begin{center}\mbox{}
\xymatrix{
 X \ar@{->}[rr]^-{\pi}\ar@{->}[d]_-{p} && X' \ar@{->}[d]^-{p'}\\
 D \ar@{->}[rr]^-{\pi'}                  && D'
}
\end{center}
is commutative, $R'$ is contained in the fibres of $p'$, and
$$q(X)-q(X'):=\dim H^0(X,\,\Omega^1_X)-\dim H^0(X',\,\Omega^1_{X'})\,=\,g(D)-g(D').$$
\end{definition}

The next theorem is proven in \cite{khashin83}.
For readers' convenience, we reprove it here.
\begin{theorem}[{\cite[Theorem 1]{khashin83}}]\label{qX>qYfibred}
Let $\pi: X\to X'$ be a double cover between smooth surfaces with smooth branched divisor $R' \subseteq X'$.
Assume that   $p_g(X'):=\dim H^0(X',\,\Omega^2_{X'})= 0$ and $q(X) > q(X')$. Then
$\pi: X\to X'$   is fibred.
\end{theorem}
\begin{proof}
Note that the Galois group ${\rm Gal}(X/X')\cong \mathbb Z_2$ has a natural action on
$H^0(X,\Omega^1_X)$.
Let $$H^0(X,\Omega^1_X)=H^0(X,\Omega^1_X)_1\oplus H^0(X,\Omega^1_X)_{-1}$$
be the eigenspace decomposition.
Then
$$H^0(X,\Omega^1_X)_1=\pi^*H^0(X',\Omega^1_{X'}),\qquad
k\triangleq\dim H^0(X,\Omega^1_X)_{-1}=q(X)-q(X')>0.$$

We show there exists a morphism $p:\,X \to D$ to a curve $D$ with connected fibres
and a subspace $W_{D} \subseteq H^0\big(D,\Omega^1_{D}\big)$, such that
\begin{equation}\label{omega=pullalpha}
H^0(X,\Omega^1_X)_{-1}=p^* \big(W_{D}\big).
\end{equation}
If $k=1$, let $A_X$ (resp. $A_{X'}$) be the Albanese variety of $X$ (resp. $X'$).
Then we have a surjective homomorphism
$A_X \to A_{X'}$ with $\dim A_X-\dim A_{X'}=k=1.$
Hence there exists  an elliptic curve $D$ with an isogeny (cf. \cite{mumford74})
$A_X \to A_X' \times D$. By  construction, 
$$p_0:~X \lra A_X \lra A_X' \times D \lra D$$
is surjective. Then the desired morphism  $p $ with connected fibres follows from
the Stein factorization of $p_0 $ (cf. \cite[\S\,III-11]{hartshorne}), so does   \eqref{omega=pullalpha}. \\[0.1cm]
If $k\geq 2$, let $\omega, \omega' \in H^0(X,\Omega^1_X)_{-1}$, then
$$\omega\wedge\omega' \in \wedge^2 H^0(X,\Omega^1_X)_{-1} \subseteq  H^0(X,\Omega^2_X)$$
is invariant under the action of  ${\rm Gal}(X/X')$ and thus belongs to
$\pi^*\left(H^0(X',\Omega^2_{X'})\right),$ which is zero by our assumption.
By \cite[Proposition\,5.1]{bhpv}, there exists a morphism $p:\,X \to D$ with connected fibres
such that \eqref{omega=pullalpha} holds.\vspace{0.1cm}

Note that  \eqref{omega=pullalpha} implies that  $p:\,X \to D$ is unique.
In particular,  the  Galois action ${\rm Gal}(X/X')$ on $X$ induces a group action $\mathbb Z_2 $ on $D$.
Let $D'=D/\mathbb Z_2$, and $\pi':\,D \to D'$ be the natural morphism.
Then by construction, there exists a morphism
$p':\,X' \to D'$ such that $p'\circ \pi=\pi'\circ p$.

Let $1\neq \sigma \in  {\rm Gal}(X/X')$.
Then the fixed locus Fix$(\sigma)$ of $\sigma$ is clearly contained in the fibres of $p$.
So $R'=\pi\big(\text{Fix}(\sigma)\big)$ is contained in the fibres of $p'$.
By \eqref{omega=pullalpha}, one sees that
the eigenspace decomposition of $H^0(D,\Omega_{D}^1)$ with respect to the action of $\mathbb Z_2$ is
$$H^0\big(D,\Omega_{D}^1\big)= (\pi')^*H^0\big(D',\Omega_{D'}^1\big)\oplus H^0\big(D,\Omega_{D}^1\big)_{-1},$$
with
$$H^0\big(X,\Omega_X^1\big)_{-1}=p^*H^0\big(D,\Omega_{D}^1\big)_{-1}.$$
So $q(X)-q(X')=k= g(D)-g(D').$ The proof is complete.
\end{proof}

Come back to our case $\tilde\pi:\,\wt S \to \wt Y$.
Note that $q(\wt S)=q(\ol S)$ and $q(\wt Y)=g(\ol B)$.
If $q_{\bar f}=q(\ol S)-g(\ol B)>0$,
it follows that $q(\wt S)>q(\wt Y)$.
As $\wt Y$ is a ruled surface, the geometric genus $p_g(\wt Y)=0$.
Hence by Theorem \ref{qX>qYfibred} above, we get
\begin{proposition}\label{fibredprop}
The double cover $\tilde\pi:\,\wt S \to \wt Y$ is fibred, i.e.,
there exist a double cover $\pi': \ol B' \to \ol D'$ of smooth projective
curves, and morphisms $\tilde f':\, \wt S\to \ol B' $ and $\tilde h':\, \wt Y \to \ol D'$
with connected fibres, such that the diagram
\begin{center}\mbox{}
\xymatrix{
 \wt S \ar@{->}[rr]^-{\pi}\ar@{->}[d]_-{\tilde f'} && \wt Y \ar@{->}[d]^-{\tilde h'}\\
 \ol B' \ar@{->}[rr]^-{\pi'}                  && \ol D'
}
\end{center}
is commutative, $\wt R$ is contained in the fibres of $\tilde h'$ and
\begin{equation}\label{q_f=gB'-gZ}
q_{\bar f}=q(\wt S)-q(\wt Y) = g(\ol B')-g(\ol D').
\end{equation}
\end{proposition}

Since $\wt Y$ is a ruled surface, it is easy to see that $\ol D' \cong \bbp^1$.
So by \eqref{q_f=gB'-gZ}, $g(\ol B')=q_{\bar f}>0$,
which implies that $\tilde f'$ factors through $\ol S$ as
\begin{equation}\label{eqntildef'factorS}
\tilde f':~\wt S \lra \ol S \overset{\bar f'}\lra B'.
\end{equation}

\begin{proposition}\label{q_f>0fibred}
Let $F$ (resp. $F'$) be any fibre of $\bar f$ (resp. $\bar f'$), and $d=F\cdot F'$.
Then $d\geq 2$; and $2g(F)-2 \geq 2d\cdot (q_{\bar f}-1)$, where $g(F)$ is the geometric genus of $F$.
In particular,
\begin{equation}\label{boundofq_f}
q_{\bar f} \leq \frac{g-1}{d}+1.
\end{equation}
\end{proposition}
\begin{proof}
We first prove $d\geq 2$.
If $d=1$, then it follows that
$(\bar f,\bar f'):\,\ol S \to \ol B \times \ol B'$ is birational,
which is a contradiction to the non-isotriviality of $\bar f$.

Now consider the restriction map
\begin{equation*}
\bar f'\big|_{F}:~F \lra \ol B',
\end{equation*}
which is a finite morphism of degree $d$.
Since $q_{\bar f}=g(B')$, according to Hurwitz formula, we get \eqref{boundofq_f}.
\end{proof}

\begin{remark}\label{boundofq_fremark}
Xiao (\cite{xiao92-0}) has proved that if $q_{\bar f}=(g-1)/d+1,$ then $\bar f$ is isotrivial.
\end{remark}\vspace{0.001cm}

\begin{proof}[Proof of Proposition {\rm\ref{boundofxi_0prop}}]~\label{proofofpropboundxi}
In order to prove \eqref{boundofxi_0eqn},
we may limit ourselves to the family $\bar f:\,\ol S \to \ol B$
coming from an admissible double cover (cf. \cite{ch88} or \cite{harrismumford82}):
possibly contracting some $(-2)$-curves in fibres,
the family $\bar f$ is a double cover of a family $\bar h:\,\ol Y \to \ol B$ of stable $(2g+2)$-pointed noded curves of
arithmetic genus zero, branched along the $2g+2$ disjoint sections $\sigma_i$ of $\bar h$
and possibly at some of the nodes of fibres of $\bar h$.
Actually, we can get a family of admissible covers from a given $\bar f$
by base change unbranched over $\ol B\setminus\Delta$
and blow-ups of singular points in the fibres.
These operations have the effect of multiplying all the invariants $\delta_i(\Upsilon)$'s and
$\xi_j(\Upsilon)$'s by the same constant,
and the relative irregularity $q_{\bar f}$ is non-decreasing under these operation.

Same as in Figure~\ref{hyperellipticdiagram}, let  $\tilde \pi:\,\wt S \to \wt Y$  be  the resolution of $\pi$
with  branched divisor $\wt R$. By pullback,  the disjoint sections $\sigma_i$'s of $\bar h$
become disjoint sections of $\tilde h$, still denoted by  $\sigma_i$'s.
And $\wt R$ is  a union of $\wt R_{nv}:=\sum\limits_{i=1}^{2g+2} \sigma_i$  and some disjoint $(-2)$-curves contained
in fibres of $\tilde h$.

Let ${\mathfrak S}=\{p_i \in \ol Y\}$ be the set of   nodes of fibres of $\bar h$,
and $\alpha_i$ (resp. $m_i$) the index (resp. multiplicity) of $p_i\in {\mathfrak S}$ (cf. Section \ref{sectionprelim}).
Let $\wt {\mathfrak S}=\{q_l \in \wt Y\}$ be the set of  nodes of fibres of $\tilde h$.
Define the index  $\beta_l$  of $q_l\in \wt {\mathfrak S}$ to be the index  of its image in  ${\mathfrak S}$.
Note that a node $p_i \in {\mathfrak S}$ of index $\alpha_i$ with multiplicity $m_i$
would introduce $m_i$ nodes in $\wt {\mathfrak S}$ with the same indices $\alpha_i$.

Let $\tilde h':\, \wt Y \to \ol D'\cong \mathbb P^1$ be the morphism given in Proposition \ref{fibredprop}.
Let $\tilde\rho:\,\wt Y \to \hat Y$ be the largest contraction of `vertical' $(-1)$-curves
such that we still have an induced morphism $\hat h':\, \hat Y \to \ol D'$,
where `vertical' means such a curve is mapped to a point on $\ol B$.
\begin{center}\mbox{}
\xymatrix@R=0.5cm{
 &\ol D'\cong \mathbb P^1&\\
 \wt Y \ar@{->}[rr]^-{\tilde\rho}\ar@{->}[dr]_-{\tilde h} \ar@{->}[ur]^-{\tilde h'}
 && \hat Y \ar@{->}[dl]^-{\hat h} \ar@{->}[ul]_-{\hat h'}\\
 &\ol B&
}\end{center}

\begin{claim}\label{CRgeq2q_f+2}
Let $\hat R=\tilde\rho(\wt R)$ and $\hat R_{nv}=\tilde\rho(\wt R_{nv})$. Then
$$E\cdot\hat R \geq 2q_{\bar f}+2,\quad E\cdot\hat R_{nv} \geq 2q_{\bar f}+1, \qquad
\text{for any `vertical' $(-1)$-curve $E\subseteq \hat Y$}.$$
\end{claim}
\begin{proof}[Proof of the claim]
By the construction,
any `vertical' $(-1)$-curve $E\subseteq \hat Y$ is mapped surjectively onto $\ol D'$ by $\hat h'$.
Since $\wt R$ is contained in fibres of $\tilde h'$ by Proposition \ref{fibredprop},
$\hat R$ is contained in fibres of $\hat h'$.
Hence $E\nsubseteq \hat R$.

Note that $\hat R$ is the union of $\hat R_{nv}$ and some curves in fibres of $\hat h$.
Let $\hat\Gamma \subseteq \hat Y$ be the fibre of $\hat h$ containing $E$. Then
$$E\cdot(\hat R-\hat R_{nv})\leq E\cdot \big(\hat \Gamma-E\big) = -E^2=1.$$
Therefore it suffices to prove $E\cdot\hat R \geq 2q_{\bar f}+2$.

Let $\wt E'\subseteq \wt S$ and $\wt E\subseteq \wt Y$ be the strict inverse image of $E$
in $\wt S$ and $\wt Y$ respectively.
Then by construction, $\wt E'$ (resp. $\wt E$) is mapped surjectively onto $\ol B'$ (resp. $\ol D'$)
by $\tilde f'$ (resp. $\tilde h'$), and
$E\cdot\hat R\geq \wt E \cdot \wt R.$
Applying Hurwitz formula to the double cover $\wt E'\to \wt E\cong \bbp^1$,
whose branched locus is at most $\wt E \cap \wt R$, one gets
\begin{equation*}
2g(\wt E')-2\leq -4+\left|\wt E \cap \wt R\right|.
\end{equation*}
As $\wt E'$ is mapped surjectively onto $\ol B'$,
$g(\wt E')\geq g(\ol B')=q_{\bar f}$.
Hence
$$E\cdot\hat R\geq \wt E \cdot \wt R\geq \left|\wt E \cap \wt R\right|
\geq 2g(\wt E')+2\geq 2q_{\bar f}+2.$$
The proof is complete.
\end{proof}

Now we contract $\hat\rho:\,\hat Y \to \ol Y$ to a $\bbp^1$-bundle $\bar h:\,\ol Y\to \ol B$   such  that
the order of every singularity of $\ol R_{nv}=\hat\rho(\hat R_{nv})$ is at most $g+1$.
It is easy to see that such a contraction exists.
\begin{center}\mbox{}
\xymatrix{
 \wt Y \ar@{->}[rr]^-{\tilde\rho}\ar@{->}[drr]_-{\tilde h} && \hat Y\ar[rr]^-{\hat\rho} \ar@{->}[d]^-{\hat h}
 &&\ol Y\ar[dll]^-{\bar h}\\
 &&\ol B&&
}\end{center}
Let $\rho=\hat\rho\circ\tilde\rho$.
Then $\rho$ can be viewed as a sequence of blow-ups
$\rho_l:\,Y_l \to Y_{l-1}$ centered at $y_{l-1}\in Y_{l-1}$ with $Y_{t+s}=\wt Y$,
$Y_s=\hat Y$, and $Y_0=\ol Y$.
Let $R_{nv,l}\subseteq Y_l$ be the image of $\wt R_{nv}$ and
$y_{l-1}$ be a singularity of $R_{nv,l-1}$ of order $n_{l-1}$.
Then one sees that each blow-up $\rho_l$
creates a node $q\in \wt{\mathfrak S}$ with index
$\beta=n_{l-1}$. Hence
\begin{equation}\label{wtR=R}
\wt R_{nv}^2=\ol R_{nv}^2-\sum_{q_l\in \wt{\mathfrak S}} \beta_l^2\,.
\end{equation}

By Claim \ref{CRgeq2q_f+2}, for $1\leq l\leq s$,
every  blow-up $\rho_l:\,Y_l\to Y_{l-1}$ is centered at a point $y_{l-1}$ with $n_{l-1}\geq 2q_{\bar f}+1$.
In other words, for $1\leq l\leq s$,
each $\rho_l$ creates a node $q\in \wt{\mathfrak S}$ with index at least $2q_{\bar f}+1$.
 The set  $\wt {\mathfrak S}$ can be divided into two subsets  $\wt {\mathfrak S}_{\hat\rho}$ and $\wt {\mathfrak S}_{\tilde\rho}$, where the first is created by blow-ups contained in $\hat \rho$ and the second  by blow-ups contained in $\tilde \rho$.
Then
\begin{eqnarray}
\beta_l&\geq&2q_{\bar f}+1,\qquad\quad\forall~q_l\in \wt {\mathfrak S}_{\hat\rho};\label{q_lgeq2q_f+1}\\
\hat R_{nv}^2&=&\ol R_{nv}^2-\sum_{q_l\in \wt {\mathfrak S}_{\hat\rho}} \beta_l^2\,.\label{hatR=R}
\end{eqnarray}
Note that $\wt R_{nv}$ consists of $2g+2$ disjoint sections $\sigma_i$'s.
According to \cite[Lemma 4.8]{ch88},
\begin{equation}\label{wtR=bych88}
\wt R_{nv}^2=\sum_{i=1}^{2g+2}\sigma_i\cdot \sigma_i
=-\sum_{p_i\in {\mathfrak S}}\frac{m_i\alpha_i(2g+2-\alpha_i)}{2g+1}
=-\sum_{q_l\in \wt {\mathfrak S}}\frac{\beta_l(2g+2-\beta_l)}{2g+1}.
\end{equation}
Combining \eqref{wtR=R}, \eqref{hatR=R} and \eqref{wtR=bych88}, one gets
\begin{equation}\label{hatR=wtR}
\hat R_{nv}^2=\wt R_{nv}^2+\sum_{q_l\in \wt {\mathfrak S}_{\tilde\rho}} \beta_l^2=
\sum_{q_l\in \wt {\mathfrak S}_{\tilde\rho}} \frac{(2g+2)\beta_l(\beta_l-1)}{2g+1}
-\sum_{q_l\in \wt {\mathfrak S}_{\hat\rho}}\frac{\beta_l(2g+2-\beta_l)}{2g+1}.
\end{equation}
Now according to Proposition \ref{fibredprop},
$\wt R_{nv}\subseteq \wt R$ is contained in the fibres of $\tilde h'$.
So $\hat R_{nv}=\tilde\rho(\wt R_{nv})$ is contained in the fibres of $\hat h'$.
In particular, $\hat R_{nv}^2 \leq 0$. Hence by \eqref{hatR=wtR}, we obtain
\begin{equation}\label{xi11}
\sum_{q_l\in \wt {\mathfrak S}_{\tilde\rho}} \beta_l(\beta_l-1)
\leq \sum_{q_l\in \wt {\mathfrak S}_{\hat\rho}}\frac{\beta_l(2g+2-\beta_l)}{2g+2}.
\end{equation}

Let $\epsilon_k$ (resp. $\nu_k$) be the number of points  in $\wt {\mathfrak S}$
of index $2k+1$ (resp. $2k+2$),
which is also the number of points in $  {\mathfrak S}$ of index $2k+1$ (resp. $2k+2$),
accounted with multiplicity.
Hence (cf. \cite[(4.10)]{ch88}),
\begin{equation}\label{invrelationdouble}
\xi_0(\Upsilon)=2\nu_0;\,~\,
\delta_i(\Upsilon)=\epsilon_i/2,~\forall\,1\leq i\leq [g/2];\,~\,
\xi_j(\Upsilon)=\nu_j,~\forall\,1\leq j\leq [(g-1)/2].
\end{equation}
Therefore,
$$\begin{aligned}
&\sum_{i= q_{\bar f}}^{[g/2]}\frac{(2i+1)(2g+1-2i)}{g+1}\delta_i(\Upsilon)
+\sum_{j= q_{\bar f}}^{[(g-1)/2]}\frac{2(j+1)(g-j)}{g+1}\xi_j(\Upsilon)\\
=&\sum_{i= q_{\bar f}}^{[g/2]}\frac{(2i+1)\big((2g+2)-(2i+1)\big)}{2g+2}\epsilon_i
+\sum_{j= q_{\bar f}}^{[(g-1)/2]}\frac{(2j+2)\big((2g+2)-(2j+2)\big)}{2g+2}\nu_j\\
=&\sum\frac{\beta_l(2g+2-\beta_l)}{2g+2}, \qquad
\text{the sum is taken over all $q_l\in \wt {\mathfrak S}$ with index $\beta_l\geq 2q_{\bar f}+1$,}\\
\geq& \sum_{q_l\in \wt {\mathfrak S}_{\hat\rho}}\frac{\beta_l(2g+2-\beta_l)}{2g+2},\quad\,~
\text{since any point $q_l\in \wt {\mathfrak S}_{\hat\rho}$ is of index $\beta_l\geq 2q_{\bar f}+1$ by \eqref{q_lgeq2q_f+1},}
\end{aligned}$$
\vspace{-0.3cm}
$$\begin{aligned}
\geq& \sum_{q_l\in \wt {\mathfrak S}_{\tilde\rho}} \beta_l(\beta_l-1), \qquad\quad \text{by~}\eqref{xi11},\\
\geq&\sum \beta_l(\beta_l-1),\qquad\qquad~~
\begin{aligned}
&\text{the sum is taken over all $q_l\in \wt {\mathfrak S}$ with index $\beta_l< 2q_{\bar f}+1$,}\\
&\text{and such points are all contained in $\wt {\mathfrak S}_{\tilde\rho}$ by \eqref{q_lgeq2q_f+1},}
\end{aligned}\qquad\,
\\[0.1cm]
=&2\nu_0+\sum_{i= 1}^{q_{\bar f}-1}2i(2i+1)\epsilon_i
+\sum_{j= 1}^{q_{\bar f}-1}2(j+1)(2j+1)\nu_j\\
=& \xi_0(\Upsilon)+\sum_{i= 1}^{q_{\bar f}-1}4i(2i+1)\delta_i(\Upsilon)
+\sum_{j= 1}^{q_{\bar f}-1}2(j+1)(2j+1)\xi_j(\Upsilon).
\end{aligned}$$
This completes the proof.
\end{proof}

\subsubsection{Proof of Theorem {\rm\ref{thmlower3}}}\label{subsectpfthmlower}
First  we consider the case that $q_{\bar f}=0$.
By \eqref{formulaofdelta_f} and \eqref{relationdeltaxi},
\begin{equation}\label{formulaofdelta_f''}
\delta_{\bar f}=\xi_0(\Upsilon)+\sum_{i=1}^{[g/2]}\delta_i(\Upsilon)
+2\sum_{j=1}^{[(g-1)/2]}\xi_j(\Upsilon).
\end{equation}
From the above equation together with \eqref{formulanoether} and \eqref{formulaofdegomega}, it follows that
\begin{equation}\label{formulaofomega^2}
\begin{aligned}
\hspace{-0.3cm}\omega^2_{\ol S/\ol B}\,=&~\frac{g-1}{2g+1}\xi_0(\Upsilon)+\\
&\hspace{-0.2cm}\sum_{i=1}^{[g/2]}\left(\frac{12i(g-i)}{2g+1}-1\right)\delta_i(\Upsilon)
+\sum_{j=1}^{[(g-1)/2]}\left(\frac{6(j+1)(g-j)}{2g+1}-2\right)\xi_j(\Upsilon).
\end{aligned}
\end{equation}
Hence,   \eqref{eqnlower3}   in the case $q_{\bar f}=0$ is obtained as below:
\begin{eqnarray*}
&&     \omega_{\ol S/\ol B}^2-\frac{4(g-1)}{g}
       \cdot\deg \bar f_*\omega_{\ol S/\ol B}\\
& =  & \sum_{i=1}^{[g/2]}\frac{4i(g-i)-g}{g}\delta_i(\Upsilon)
       +\sum_{j=1}^{[(g-1)/2]}\frac{2(j+1)(g-j)-2g}{g}\xi_j(\Upsilon)\\
&\geq&
\left\{
\begin{aligned}
&\frac{3g-4}{g} \delta_1(\Upsilon)+\frac{7g-16}{g} \delta_h(\Upsilon), &\qquad&
\text{if~} \Delta_{nc}\neq \emptyset;\\
&\sum_{i=1}^{[g/2]}\frac{4i(g-i)-g}{g}\delta_i(\Upsilon), &&
\text{if~}\Delta_{nc}= \emptyset.
\end{aligned}\right.
\end{eqnarray*}

Next we consider the case $q_{\bar f}>0$.
It is based on  \eqref{formulaofdegomega}
and \eqref{boundofxi_0eqn}.
Assume that $\Delta_{nc}\neq \emptyset$.
By \eqref{formulaofdegomega} and \eqref{formulaofomega^2}, one gets
$$\begin{aligned}
\omega_{\ol S/\ol B}^2-\frac{4(g-1)}{g-q_{\bar f}}\deg \bar f_*\omega_{\ol S/\ol B}=&
-\frac{(g-1)q_{\bar f}}{(2g+1)(g-q_{\bar f})}\xi_0(\Upsilon)\\
&\hspace{-0.2cm}+\sum_{i=1}^{[g/2]}\left(\frac{4(2g-3q_{\bar f}+1)i(g-i)}{(2g+1)(g-q_{\bar f})}-1\right)\delta_i(\Upsilon)\\
&\hspace{-0.4cm}+\sum_{j=1}^{[(g-1)/2]}\left(\frac{2(2g-3q_{\bar f}+1)(j+1)(g-j)}{(2g+1)(g-q_{\bar f})}-2\right)\xi_j(\Upsilon).
\end{aligned}$$
Combining this with \eqref{boundofxi_0eqn}, one gets
\begin{eqnarray*}
&&\omega_{\ol S/\ol B}^2-\frac{4(g-1)}{g-q_{\bar f}}\deg \bar f_*\omega_{\ol S/\ol B}\\
&\geq& \sum_{i= 1}^{q_{\bar f}-1} a_i \delta_i(\Upsilon)+
\sum_{i= q_{\bar f}}^{[g/2]} b_i \delta_i(\Upsilon)+
\sum_{j= 1}^{q_{\bar f}-1} c_j \xi_j(\Upsilon)+
\sum_{j= q_{\bar f}}^{[(g-1)/2]} d_j \xi_j(\Upsilon),
\end{eqnarray*}
where
$$\left\{
\begin{aligned}
&a_i=\left(\frac{4(2g-3q_{\bar f}+1)i(g-i)}{(2g+1)(g-q_{\bar f})}-1\right)
     +\frac{(g-1)q_{\bar f}}{(2g+1)(g-q_{\bar f})} \cdot 4i(2i+1),\\
&b_i=\left(\frac{4(2g-3q_{\bar f}+1)i(g-i)}{(2g+1)(g-q_{\bar f})}-1\right)
     -\frac{(g-1)q_{\bar f}}{(2g+1)(g-q_{\bar f})} \cdot \frac{(2i+1)(2g+1-2i)}{g+1},\\
&c_j=\left(\frac{2(2g-3q_{\bar f}+1)(j+1)(g-j)}{(2g+1)(g-q_{\bar f})}-2\right)
     +\frac{(g-1)q_{\bar f}}{(2g+1)(g-q_{\bar f})} \cdot 2(j+1)(2j+1),\\
&d_j=\left(\frac{2(2g-3q_{\bar f}+1)(j+1)(g-j)}{(2g+1)(g-q_{\bar f})}-2\right)
     -\frac{(g-1)q_{\bar f}}{(2g+1)(g-q_{\bar f})} \cdot \frac{2(j+1)(g-j)}{g+1}.
\end{aligned}\right.
$$
If $q_{\bar f}=1$, then
$$
\begin{aligned}
b_1&=\frac{3g-6}{g+1};&\quad&\\
b_i&=\frac{4i(g-i)-g-2}{g+1}\geq \frac{7g-18}{g+1}, && \forall~2\leq i\leq [g/2];\\
d_j&=\frac{2\big((j+1)(g-j)-(g+1)\big)}{g+1}\geq 0,&& \forall~1\leq j\leq [(g-1)/2].
\end{aligned}
$$
If $q_{\bar f}\geq 2$, then
$$
\begin{aligned}
a_1&\geq \frac{3g^2-(8q_{\bar f}+1)g+10q_{\bar f}-4}{(g+1)(g-q_{\bar f})};&\quad&\\
a_i&\geq \frac{7g^2-(16q_{\bar f}+9)g+34q_{\bar f}-16}{(g+1)(g-q_{\bar f})},&\quad&\forall~2\leq i\leq q_{\bar f}-1;\\
b_i&\geq \frac{7g^2-(16q_{\bar f}+9)g+34q_{\bar f}-16}{(g+1)(g-q_{\bar f})}, && \forall~q_{\bar f}\leq i\leq [g/2];\\
c_j&\geq 0,&& \forall~1\leq j\leq q_{\bar f}-1;\\
d_j&\geq 0,&& \forall~q_{\bar f}\leq j\leq [(g-1)/2].
\end{aligned}
$$
Hence \eqref{eqnlower3} holds for $\Delta_{nc}\neq\emptyset$.
\vspace{0.2cm}

Now we consider the case that $\Delta_{nc}=\emptyset$.
Note that in this case,
\begin{equation}\label{xi=0}
\xi_j(\Upsilon)=0, \qquad\qquad \forall~ 0 \leq j \leq [(g-1)/2].
\end{equation}
Hence by \eqref{formulaofdegomega} and \eqref{formulaofomega^2}, we get
$$\omega_{\ol S/\ol B}^2-\frac{4(g-1)}{g-q_{\bar f}}\deg \bar f_*\omega_{\ol S/\ol B}=
\sum_{i=1}^{[g/2]} \left(\frac{4(2g+1-3q_{\bar f})i(g-i)}{(2g+1)(g-q_{\bar f})}-1\right) \delta_i(\Upsilon).$$
Hence \eqref{eqnlower3} holds.
If moreover $q_{\bar f}\geq 2$, then according to \eqref{boundofxi_0eqn} and \eqref{xi=0},
we get
$$\sum_{i=q_{\bar f}}^{[g/2]} \frac{(2i+1)(2g+1-2i)}{g+1} \cdot \delta_i(\Upsilon)\geq
\sum_{i=1}^{q_{\bar f}-1} 4i(2i+1) \cdot \delta_i(\Upsilon).$$
So \eqref{eqnlower3'} is proved.
\qed

\subsection{Proof of Theorem \ref{thmlower1}}\label{sectionlower1}
It is based on analyzing the following  natural multiplication
\begin{equation}\label{multiplication}
\varrho:\,S^2\left(\bar f_*\omega_{\ol S/\ol B}\right) \lra \bar f_*\big(\omega_{\ol S/\ol B}^{\otimes2}\big),
\end{equation}
where $S^2\left(\bar f_*\omega_{\ol S/\ol B}\right)$ is the symmetric power of $\bar f_*\omega_{\ol S/\ol B}$.

\begin{proof}[Proof of Theorem {\rm\ref{thmlower1}}]
  As $\bar f$ is non-hyperelliptic,
the morphism $\varrho$  in \eqref{multiplication} is generically surjective.
One gets an exact sequence as below:
\begin{equation*} 
0 \lra \calr \lra S^2\left(\bar f_*\omega_{\ol S/\ol B}\right) \overset{\varrho}\lra
\bar f_*\big(\omega_{\ol S/\ol B}^{\otimes2}\big) \lra \cals \lra 0,
\end{equation*}
where  $\calr$ and $\cals$ are the kernel and cokernel of  $\varrho$, and   $\cals$ is a torsion module.
So
\begin{equation} \label{deg00}
\deg \calr+\deg \bar f_*\big(\omega_{\ol S/\ol B}^{\otimes2}\big)=\deg S^2\left(\bar f_*\omega_{\ol S/\ol B}\right) +\deg \cals.
\end{equation}
It is not difficult to show that
\begin{eqnarray}
\deg S^2\left(\bar f_*\omega_{\ol S/\ol B}\right)&=&(g+1)\deg \bar f_*\omega_{\ol S/\ol B},\label{deg0}\\
\deg \bar f_*\big(\omega_{\ol S/\ol B}^{\otimes2}\big)&=& \omega_{\ol S/\ol B}^2+\deg \bar f_*\omega_{\ol S/\ol B}. \label{exactdeg}
\end{eqnarray}
Since $\bar f_*\omega_{\ol S/\ol B}$ is semi-stable,
so is $S^2\left(\bar f_*\omega_{\ol S/\ol B}\right)$ of slope $\mu_2=\frac{2\deg \bar f_*\omega_{\ol S/\ol B}}{g}$.
Note that
$$\rank \calr = \rank S^2\left(\bar f_*\omega_{\ol S/\ol B}\right) -\rank \bar f_*\big(\omega_{\ol S/\ol B}^{\otimes2}\big)=\frac{(g-2)(g-3)}{2}.$$
Hence for the subsheaf $\calr\subseteq S^2\left(\bar f_*\omega_{\ol S/\ol B}\right)$, we have
\begin{equation}\label{deg1}
\deg \calr \leq \rank \calr \cdot \mu_2 = \frac{(g-2)(g-3)}{g}\cdot \deg \bar f_*\omega_{\ol S/\ol B}.
\end{equation}
Since $\cals$ is a  torsion module,
\begin{equation}\label{deg2}
\deg \cals=\sum_{p} {\rm length\,} \cals_p.
\end{equation}
Therefore, by \eqref{deg00}, \eqref{deg0}, \eqref{exactdeg}, \eqref{deg1} and \eqref{deg2},
it suffices to prove that
\begin{equation}\label{lengthsq}
{\rm length\,}\cals_p \geq\left\{
\begin{aligned}
&3\big(l_h(F_p)-1\big)+2l_1(F_p), && \forall\, p \in \Delta_{ct} \setminus \Lambda,\\[0.1cm]
&2(g-2)+2\big(l_h(F_p)+l_1(F_p)-1\big), &\qquad& \forall\, p \in \Lambda,
\end{aligned}\right.
\end{equation}
which follows from  Lemmas \ref{lowlinshi1} and \ref{lowlinshi2}.
Indeed, if $p\in \Delta_{ct}\setminus \Lambda$, then \eqref{lengthsq} follows from
\eqref{cokernelofnu_p} and \eqref{S_pcokernelnu_p};
if $p\in \Lambda$, then $F_p$ is hyperelliptic, and so \eqref{lengthsq} follows from
\eqref{cokernelofnu_phyper} and \eqref{S_pcokernelnu_phyper}.
This completes the proof.
\end{proof}
\begin{lemma}\label{lowlinshi1}
For every $p\in \ol B\setminus \Delta_{nc}$, let
\begin{equation}\label{defofnu_p}
\nu_p:\, S^2H^0(F_p, \omega_{F_p}) \to H^0\big(F_p, \omega_{F_p}^{\otimes2}\big)
\end{equation}
the natural multiplication map on $F_p$, and denote by ${\rm coker}(\nu_p)$ the cokernel of $\nu_p$. Then
\begin{equation}\label{cokernelofnu_p}
\dim {\rm coker}(\nu_p) \geq 3\big(l_h(F_p)-1\big)+2l_1(F_p).
\end{equation}
If $F_p$ is hyperelliptic, then
\begin{equation}\label{cokernelofnu_phyper}
\dim {\rm coker}(\nu_p) \geq (g-2)+\big(l_h(F_p)+l_1(F_p)-1\big).
\end{equation}
\end{lemma}
\begin{proof}
Assume $F_p=\sum\limits_{i} C_i$.
Let ${\rm pr}_i:\, F_p \to C_i$ be the natural contraction map  to $C_i$.
By pulling-back, one may view $H^0(C_i,\omega_{C_i})$ as a subspace of $H^0(F_p,\omega_{F_p})$,
and
$$H^0(F_p,\omega_{F_p})\cong \bigoplus_{i} H^0(C_i,\omega_{C_i}).$$
Note that for   $\omega \in H^0(F_p,\omega_{F_p})$ which lies in $H^0(C_i,\omega_{C_i})$
in the above decomposition, we have $\omega|_{C_j}=0$ for all $j\neq i$.
Thus the map $\nu_p$ factors through
$$
\xymatrix{
S^2H^0(F_p, \omega_{F_p})  \ar[rr]^-{\nu_p}\ar[d]
&& H^0(F_p, \omega_{F_p}^{\otimes2})\\
\bigoplus\limits_{i} S^2H^0(C_i,\omega_{C_i}) \ar[rr]^-{\bigoplus\nu_i}
&&\bigoplus\limits_{i} H^0(C_i,\omega_{C_i}^{\otimes2})\ar[u]_-{\bigoplus {\rm pr}_i^*},
}$$
where $\nu_i:\,S^2H^0(C_i,\omega_{C_i}) \to H^0(C_i,\omega_{C_i}^{\otimes2})$ is
the natural multiplication map.
Note that for a smooth closed curve $C_i$ of genus $g(C_i)$,
\begin{equation*}
\dim \nu_i\big(S^2H^0(C_i,\omega_{C_i})\big) \leq\left\{
\begin{aligned}
&3g(C_i)-3, &\qquad& \text{if~} g(C_i) \geq 2;\\
&3g(C_i)-2=1, && \text{if~} g(C_i) =1;\\
&3g(C_i)=0, && \text{if~} g(C_i) =0.
\end{aligned}\right.
\end{equation*}
Since $p\in \ol B\setminus \Delta_{nc}$, $F_p$ has compact Jacobian,  then $\sum\limits_i g(C_i)=g$.
Hence
$$\dim \nu_p\big(S^2H^0(F_p,\omega_{F_p})\big)\leq \sum_i 3g(C_i)-3l_h(F_p)-2l_1(F_p)
=3g-3l_h(F_p)-2l_1(F_p).$$
Since $\dim H^0\big(F_p, \omega_{F_p}^{\otimes2}\big)=3g-3$,
\eqref{cokernelofnu_p} follows immediately.

If $F_p$ is hyperelliptic, then each component $C_i \subseteq F_p$
with $g(C_i)\geq 2$ must be hyperelliptic.
Thus it follows that
\begin{equation*}
\dim \nu_i\big(S^2H^0(C_i,\omega_{C_i})\big) =\left\{
\begin{aligned}
&2g(C_i)-1, &\qquad& \text{if~} g(C_i) \geq 1;\\
&2g(C_i)=0, && \text{if~} g(C_i) =0.
\end{aligned}\right.
\end{equation*}
So
$$\dim \nu_p\big(S^2H^0(F_p,\omega_{F_p})\big)= \sum_i 2g(C_i)-l_h(F_p)-l_1(F_p)
=2g-l_h(F_p)-l_1(F_p).$$
Hence \eqref{cokernelofnu_phyper} follows.
\end{proof}

\begin{lemma}\label{lemmapropertyoff}
For  $p\in \Lambda$, there exists a neighborhood
$p\in W\subseteq \ol B$ with local coordinate $t$ and local sections
$s_1,\cdots,s_g$ of $\bar f_*\omega_{\ol S/\ol B}$ such that every $s_i$ is a
function of $t^2$ and $H^0(F_p,\omega_{F_p})$ is generated by
$\{\varphi(s_1),\cdots,\varphi(s_g)\}$, where $\varphi$ is defined
as below:
\begin{equation}\label{linshi1}
\varphi:~H^0\big(W,\bar f_*\omega_{\ol S/\ol B}\big) \overset{\varphi_1}
\lra H^0\big(\bar f^{-1}(W),\omega_{\ol S/\ol B}\big)
\overset{\varphi_2}\lra H^0(F_q,\omega_{F_q}).
\end{equation}
\end{lemma}
\begin{proof}
Recall the logarithmic Higgs bundle on $\ol B$ (resp. $\ol C$)
$$\left(E^{1,0}_{\ol B}\oplus E^{0,1}_{\ol B},~\theta_{\ol B}\right) \qquad
\text{resp.~} \left(E^{1,0}_{\ol C}\oplus E^{0,1}_{\ol C},~\theta_{\ol C}\right).$$
 By \eqref{xinpullback},
$$E^{1,0}_{\ol B}=\bar f_*\omega_{\ol S/\ol B}=\bar j_B^* E^{1,0}_{\ol C}.$$
Moreover, as $\Lambda$ is the ramification locus of $j_B=(\bar j_B)|_{B}:\,B \to C$, it
follows that  around each $p\in \Lambda$, $\bar f_*\omega_{\ol S/\ol B}$ can be
locally generated by sections which are functions of $t^2$, where
$t$ is a suitable local coordinate of $\ol B$. The map $\varphi$ is
obviously surjective.
\end{proof}

\begin{lemma}\label{lowlinshi2}
For every $p\in \ol B\setminus \Delta_{nc}$, we have
\begin{equation}\label{S_pcokernelnu_p}
{\rm length\,} \cals_p \geq \dim {\rm coker}(\nu_p).
\end{equation}
If $p\in\Lambda$, then
\begin{equation}\label{S_pcokernelnu_phyper}
{\rm length\,} \cals_p \geq 2 \dim {\rm coker}(\nu_p).
\end{equation}
\end{lemma}
\begin{proof}
Let $\mathfrak m_p$ be the ideal of $p\in \ol B$.
Then we have the following natural isomorphism
\begin{equation*}
\varrho\left(S^2(\bar f_*\omega_{\ol S/\ol B})\right)\Big/
\left(\mathfrak m_p \cdot\varrho\big(S^2(\bar f_*\omega_{\ol S/\ol B})\big)\right)
\cong \nu_p\left(S^2H^0(F_p, \omega_{F_p})\right).
\end{equation*}
It follows that
$\cals_p\big/\mathfrak m_p\cdot \cals_p \cong {\rm coker}(\nu_p).$
Hence
$${\rm length\,} \cals_p \geq \dim
\left(\cals_p\big/\mathfrak m_p\cdot \cals_p\right) = \dim {\rm coker}(\nu_p).$$

If $p\in \Lambda$, then according to Lemma \ref{lemmapropertyoff} and its proof,
$\bar f_*\omega_{\ol S/\ol B}$ is generated by local sections which are functions in $t^2$ around $p$,
where $t$ is a local coordinate of $\ol B$ around $p$.
Hence the image of $\varrho$
is also generated by local sections which are functions in $t^2$ around $p$.
This implies in particular that
$$\dim \left(\cals_p\big/\mathfrak m_p^2\cdot \cals_p\right) =
   2\dim \left(\cals_p\big/\mathfrak m_p\cdot \cals_p\right).$$
Hence
$${\rm length\,} \cals_p \geq \dim \left(\cals_p\big/\mathfrak m_p^2\cdot \cals_p\right) =
   2\dim \left(\cals_p\big/\mathfrak m_p\cdot \cals_p\right) = 2\dim {\rm coker}(\nu_p).$$
\end{proof}

\section{Flat part of $R^1\bar f_*\mathbb Q$ for a family of hyperelliptic semi-stable curves}\label{sectionflathyper}
We are going to prove  Theorem \ref{thmFtrivial} based on  Lemma \ref{lemmapeters} regarding the global invariant
cycle with unitary locally constant coefficient
and Bogomolov's lemma (cf. \cite[Lemma 7.5]{sakai80}) concerning the Kodaira dimension of an invertible subsheaf
of the sheaf of logarithmic differential forms on a smooth projective surface.  Lemma \ref{lemmapeters}
comes from  a discussion with Chris Peters and is obtained by generalizing Deligne's original theorem
with the constant coefficient.
For technical reasons, we consider $\mathbb C$-local systems instead of $\mathbb Q$-local systems.

\begin{lemma}\label{lemmapeters}
Let $f: X^0\to \ol B\setminus \Delta$ be a smooth projective morphism.
Let $ X\supseteq X^0$ be a smooth compactification of $X^0$
and $\mathbb U$ be a locally constant sheaf on $ X$
coming from a representation of  $\pi_1(X)$ into a unitary group.
Then the following canonical homomorphism is surjective:
  $$ H^k( X,\mathbb U)\lra H^0(\ol B\setminus \Delta, R^kf_*\mathbb U).$$
\end{lemma}
\begin{proof}
We will follow
Deligne's  proof for the case
that $\mathbb U=\mathbb Q$ (cf. \cite[\S\,4.1]{deligne71})  verbatim.

The unitary locally constant sheaf $\mathbb U$ on $X$ naturally underlies a polarized variation
of Hodge structure, say, of pure type $(0,0)$. Hence
it follows from M. Saito's theory of polarizable Hodge modules
that there is an induced pure  Hodge structure of weight $k$ on
$H^k(X, \mathbb U)$ as well
as on $H^k(X_b,\mathbb U|_{X_b})$ where $X_b$ is any (smooth projective) fibre of $f:X^0\to \ol B\setminus \Delta$.

We first show  the ``edge-homomorphism"
\begin{equation}\label{eqnedgehom}
p_e:~H^k(X^0, \mathbb U)\lra H^0(\ol B\setminus \Delta,R^kf_*\mathbb U)
\end{equation}
is surjective by the following argument from the proof of \cite[Proposition 1.38]{peterssteenbrink}.

Indeed, if we take $h\in H^2(X^0, \mathbb{Q})$ to be the restriction of a hyperplane class of $X$,
then it suffices to show that the cup-products satisfy the hard Lefschetz property,
i.e., the following homomorphism is an isomorphism
for any $0\leq k\leq m$, where $m$ is the dimension of a general fibre of $f$:
$$
[\cup \,h]^k: ~R^{m-k}f_*\mathbb{U}\longrightarrow R^{m+k}f_*\mathbb{U}.
$$
Note that the hard Lefschetz property can be verified fiber-by-fiber.
On each fiber the natural locally constant metric on $\mathbb{U}$
induces a Hodge decomposition of the cohomology with coefficients in $\mathbb{U}$,
hence the hard Lefschetz property holds.
So $p_e$ in \eqref{eqnedgehom} is surjective.

Since the restriction homomorphism (as monodromy invariant)
$$ H^0(\ol B\setminus \Delta,R^kf_*\mathbb U)\to H^k(X_b,\mathbb U|_{X_b})$$
is injective and $H^k(X_b,\mathbb U|_{X_b})$ carries a pure Hodge structure of weight-$k$,
one gets that $H^0(\ol B\setminus\Delta, R^kf_*\mathbb U)$ carries a pure Hodge structure of weight-$k.$
The surjectivity of $p_e$ in \eqref{eqnedgehom} also induces surjective morphisms
between the weight-filtrations  of the both cohomologies. In particular, we have a surjective homomorphism
\begin{equation}\label{eqnpeters1}
W_k\big(H^k(X^0,\mathbb U)\big)\twoheadrightarrow
W_k\big(H^0(\ol B\setminus\Delta, R^kf_*\mathbb U)\big)=H^0(\ol B\setminus\Delta, R^kf_*\mathbb U).
\end{equation}
By \cite{peterssaito12},
$W_k \big(H^k(X^0,\mathbb U)\big)$ is nothing but the image of the restriction homomorphism
\begin{equation*}
H^k( X,\mathbb U)\to H^k(X^0,\mathbb U).
\end{equation*}
Combining this with \eqref{eqnpeters1}, one gets required surjective homomorphism.
\end{proof}

\begin{corollary}\label{corollaylifting}
Let $\bar f:\ol S\to \ol B$ be a semi-stable family of projective curves
(not necessarily hyperelliptic) over a smooth projective curve $\ol B,$
with semi-stable singular fibres $\Upsilon\to \Delta.$
Given any  vector subbundle
$\mathcal U\hookrightarrow \bar f_*\omega_{\ol S/\ol B}=\bar f_*\Omega^1_{\ol S/\ol B}(\log\Upsilon),$
which { is a flat subbundle and is induced by} locally constant subsheaf
$\mathbb U\hookrightarrow \mathbb V_{\ol B\setminus \Delta}:= R^1\bar f_*\mathbb C_{\bar f^{-1}(\ol B\setminus \Delta)},$
it lifts to a morphism
\begin{equation}\label{eqnliftofU}
\bar f^*\mathcal U\lra \Omega^1_{\ol S},
\end{equation}
such that the induced canonical morphism
\begin{equation}\label{eqncanmapforS}
\mathcal U\to \bar f_*\Omega^1_{\ol S}\to \bar f_*\Omega^1_{\ol S}(\log\Upsilon)
\to \bar f_*\Omega^1_{\ol S/\ol B}(\log\Upsilon)
\end{equation}
coincides with the inclusion $\mathcal U\hookrightarrow  \bar f_*\Omega^1_{\ol S/\ol B}(\log\Upsilon).$
\end{corollary}
\begin{proof}
Since the local monodromy of $\mathbb V_{\ol B\setminus \Delta}$ around $\Delta$ is unipotent and
the local monodromy of the subsheaf $\mathbb U$ around $\Delta$ is semisimple, $\mathbb U$
extends on $\ol B$ as a locally constant sheaf.  The inclusion
$ \mathbb U\hookrightarrow \mathbb V_{\ol B\setminus \Delta}$ corresponds to
a  section
$$\eta\in H^0(\ol B\setminus\Delta,~\mathbb V_{\ol B\setminus \Delta}\otimes\mathbb U^\vee)
=H^0(\ol B\setminus\Delta, ~R^1\bar f_*\big(\mathbb C_{\bar f^{-1}(\ol B\setminus \Delta)}\otimes \bar f^*\mathbb U^\vee)\big).$$
By Lemma \ref{lemmapeters}, $\eta$ lifts to a class $\tilde \eta\in H^1(\ol S, \bar f^*\mathbb U^\vee)$
under the canonical morphism
$$ H^1(\ol S, \bar f^*\mathbb U^\vee)\lra H^0(\ol B\setminus \Delta, ~
R^1\bar f_*\big(\mathbb C_{\bar f^{-1}(\ol B\setminus \Delta)}\otimes \bar f^*\mathbb U^\vee)\big).$$
Note that this canonical morphism is a morphism between  pure Hodge structures of weight one,
and by the construction
$\eta$ is of type (1,0), so $\tilde \eta$ is of type (1,0), i.e.,
$$\tilde \eta\in H^0(\ol S,\Omega^1_{\ol S}\otimes \bar f^*\mathcal U^\vee),$$
which corresponds to a morphism
$\bar f^*\mathcal U\to \Omega^1_{\ol S}$,
such that under the canonical morphism \eqref{eqncanmapforS} it goes back to the inclusion
$\mathcal U\hookrightarrow \bar f_*\Omega^1_{\ol S/\ol B}(\log\Upsilon).$
\end{proof}

Let $\tilde\pi:\,\wt S \to \wt Y$ be the smooth double cover described in Figure \ref{hyperellipticdiagram},
and $\vartheta:\,\wt S \to \ol S$ be the blow-ups.
Given any  vector subbundle
$\mathcal U\hookrightarrow \bar f_*\Omega^1_{\ol S/\ol B}(\log\Upsilon)$ as in Corollary \ref{corollaylifting},
by pulling back of \eqref{eqnliftofU}, we obtain a sheaf morphism
$$\tilde f^*\mathcal U=\vartheta^*\bar f^*\mathcal U \lra \Omega^1_{\wt S},
\qquad\qquad\text{where~}\tilde f=\bar f\circ \vartheta,$$
which corresponds to an element
$$\tilde\eta\in H^0(\wt S, \Omega^1_{\wt S}\otimes \tilde f^*\mathcal U^{\vee}).$$
By pushing-out, we also obtain an element (where $\tilde h:\,\wt Y \to \ol B$ is the induced morphism)
$$\tilde\pi_*(\tilde\eta)\in H^0\left(\wt Y, \tilde\pi_*\big(\Omega^1_{\wt S}\otimes
\tilde f^*\mathcal U^{\vee}\big)\right)
=H^0\left(\wt Y, \tilde\pi_*\Omega^1_{\wt S}\otimes \tilde h^*\mathcal U^{\vee}\right).$$
So one gets a  morphism of sheaves
$\tilde h^*\mathcal U \lra \tilde\pi_*\Omega^1_{\wt S}.$
The Galois group ${\rm Gal}(\wt S/\wt Y)\cong \mathbb Z_2$ acts on
$\tilde\pi_*\Omega^1_{\wt S}.$
Hence one obtains the eigenspace decomposition
$$ \tilde h^*(\mathcal U)\lra \left(\tilde\pi_*\Omega^1_{\wt S}\right)_1,
\qquad \tilde h^*(\mathcal U)\lra \left(\tilde\pi_*\Omega^1_{\wt S}\right)_{-1}.$$

\begin{lemma}\label{LR=0}
The image of the map
$$\varrho:~\tilde h^*(\mathcal U)\lra \left(\tilde\pi_*\Omega^1_{\wt S}\right)_{-1}$$
is an invertible subsheaf $M$
such that $M$ is numerically effective (nef),  $M^2=0$,
and $M\cdot D=0$ for any component $D$ of the branch divisor $\wt R\subseteq \wt Y$
of the double cover $\tilde\pi: \wt S\to \wt Y$.
\end{lemma}
\begin{proof}
First of all, we show that $\varrho\neq 0$. It is known that
\begin{equation}\label{pi_*fenjie}
\left(\tilde\pi_*\Omega^1_{\wt S}\right)_1=\Omega^1_{\wt Y},\qquad
\left(\tilde\pi_*\Omega^1_{\wt S}\right)_{-1}=\Omega^1_{\wt Y}\left(\log (\wt R)\right)(-\wt L),
\end{equation}
where $\wt R\equiv2\wt L$ ($\equiv$ stands for linearly equivalent)
is the defining data of the double cover $\tilde\pi:\, \wt S \to \wt Y$.
Note that by Corollary \ref{corollaylifting}, the induced map
\begin{eqnarray}
\mathcal U=\tilde h_*\tilde h^*\mathcal U
&\lra& \tilde h_*\bigg(\left(\tilde\pi_*\Omega^1_{\wt S}\right)_1\bigoplus
\left(\tilde\pi_*\Omega^1_{\wt S}\right)_{-1}\bigg)\nonumber\\
&&\hspace{-0.4cm}=\tilde h_*\bigg(\Omega^1_{\wt Y}\left(\log (\wt R)\right)(-\wt L)\bigg)\hookrightarrow
\bar f_*\Omega^1_{\ol S/\ol B}(\log\Upsilon)\label{Ninclusion}
\end{eqnarray}
is just the inclusion $\mathcal U\hookrightarrow \bar f_*\Omega^1_{\ol S/\ol B}(\log\Upsilon)$.
Hence in particular, $\varrho\neq0$.

Second, we prove that the image of $\varrho$ is a subsheaf of rank one. Otherwise,
it is of rank two,
and so the second wedge product
$$\wedge^2 \tilde h^* \mathcal U\overset{\wedge^2\varrho}\lra
\wedge^2\left( \tilde\pi_*\Omega^1_{\wt S}\right)_{-1}=\omega_{\wt Y}$$
is a non-zero map. Note that the image of the map
is a quotient sheaf of  $\wedge^2 \tilde h^*\mathcal U$ coming from a unitary local system, so the image sheaf
is semi-positive. But this is impossible, because $\omega_{\wt Y}$ can not contain
any non-zero semi-positive subsheaf.

Finally, we show that the image of $\varrho$ is an invertible sheaf.
Suppose on the contrary the image of $\varrho$ is of the form $M\otimes I_{Z}$,
where $M$ is an invertible subsheaf and $\dim Z=0$ with $Z\neq\emptyset$.
By a suitable blow-up $\rho:\,X\to \wt Y$, we may assume the image of
$\rho^*\tilde h^*\mathcal U$ is $\rho^*(M)\otimes \calo_X(-E)$,
where $E>0$ is a combination of the exceptional curves.
As $\mathcal U$ comes from a unitary local system, we get that $\rho^*(M)\otimes \calo_X(-E)$ is semi-positive and so
$$0\leq \big(\rho^*(M)-E\big)^2=M^2+E^2.$$
Hence $M$ is also semi-positive and $M^2 \geq -E^2>0$, which implies that the Kodaira dimension of $M$ is two.
However, by \eqref{pi_*fenjie}, we get the following inclusion of sheaves,
\begin{equation}\label{Linclusion}
\mathcal O_{\wt Y}(\wt L)\otimes M\hookrightarrow \Omega^1_{\wt Y}\left(\log (\wt R)\right).
\end{equation}
As $2\wt L\equiv \wt R$ is effective,
the Kodaira dimension of $\wt L\otimes M$ is also two, which
is impossible by Bogomolov's lemma (cf. \cite[Lemma 7.5]{sakai80}).

Hence the image of $\varrho$ is an invertible subsheaf $M$,
which is semi-positive since it is a quotient sheaf of a vector bundle coming from a unitary local system.
And we still have the inclusion \eqref{Linclusion}.
So again by Bogomolov's lemma, we get $M^2=0,\text{~and~} M\cdot D=0.$
\end{proof}

\begin{proposition}\label{theoremq_f=rankF10}
After a suitable base change unbranched over $\ol B\setminus \Delta$,
$F_{\ol B}^{1,0}$ is trivial, i.e.,
\begin{equation}\label{eqnpfofthmFtrivial1}
F_{\ol B}^{1,0}=\bigoplus\limits_{i=1}^{r} \mathcal O_{\ol B},\qquad\qquad {\rm where~} r=\rank F^{1,0}_{\ol B}.
\end{equation}
\end{proposition}
\begin{proof}
Similarly to the proof of Proposition \ref{boundofxi_0prop} on Page \pageref{proofofpropboundxi},
we may restrict ourselves to the situation that
the double cover $\pi:\,\ol S\to \ol Y$ induced by the hyperelliptic involution
is an admissible double cover.
Then the branched divisor $\wt R \subseteq \wt Y$ of the smooth double cover
$\tilde\pi:\, \wt S \to \wt Y$ in Figure \ref{hyperellipticdiagram}
is a union of $2g+2$ sections and some curves contained in fibres of $\tilde h:\, \wt Y \to \ol B$.
Let $D$ be such a section, and
$$F^{1,0}_{\ol B}=\bigoplus\limits_{i=1}^{t}\,\mathcal U_i,\qquad \text{with each~}\mathcal U_i\text{~irreducible}.$$

We first prove $\rank \mathcal U_i=1$ for any $i$.
Assume on the contrary that $\rank \mathcal U_i \geq 2$ for some $i$.
Let $f:\,S \to B=\ol B\setminus\Delta_{nc}$
be the largest part of $\bar f$ such that the relative Jacobian is smooth.
Then
\begin{equation}\label{eqnpfofthmFtrivial2}
\left(\mathbb V_{B}:=R^1jac(f)_*\mathbb C_{Jac(S/B)} \right)\big|_{\ol B\setminus\Delta} \cong
\mathbb V_{\ol B\setminus \Delta}:= R^1f_*\mathbb C_{f^{-1}(\ol B\setminus \Delta)}.
\end{equation}
Hence the inclusion
$$\mathcal U_i \hookrightarrow F^{1,0}_{\ol B} \hookrightarrow \bar f_*\omega_{\ol S/\ol B}
=\bar f_*\Omega^1_{\ol S/\ol B}(\log\Upsilon)$$
underlies a unitary locally constant subsheaf
$\mathbb U_i\hookrightarrow \mathbb V_{\ol B\setminus \Delta}:= R^1f_*\mathbb C_{f^{-1}(\ol B\setminus \Delta)}$.
By Lemma \ref{LR=0} with $\mathcal U=\mathcal U_i$,
we obtain $M\cdot D=0$,
i.e., $\deg\mathcal O_D(M)=0$.
As $D$ is a section, $D\cong \ol B$. Hence we may view $\calf:=\mathcal O_D(M)$ as an invertible
sheaf on $\ol B$, which is a quotient of $\mathcal U$ for $M$ is a quotient of $\tilde h^*\mathcal U$.
As $\mathcal U$ comes from a unitary local system, $\mathcal U$ is poly-stable.
Thus $\mathcal U=\calf \oplus \mathcal U'$ contradicting the irreducibility of $\mathcal U$.
Hence $\rank \mathcal U_i=1$ as required.

Now applying \cite[\S\,4.2]{deligne71} or \cite[Theorem 3.4]{bar98},
we get that $\mathcal U_i$ is torsion in ${\rm Pic}^0(\ol B)$.
So after a further suitable finite \'etale base change,
$\mathcal U_i \cong \mathcal O_{\ol B}$ for any $i$ as required.
\end{proof}

\begin{proof}[Proof of Theorem {\rm \ref{thmFtrivial}}]
Because of \eqref{eqnpfofthmFtrivial2}, the flat part $F_{\ol B}^{1,0}$ in the decomposition \eqref{decompB}
is a flat vector bundle underlying a unitary local subsystem of $\mathbb V_{\ol B\setminus\Delta}$
coming from a representation of the fundamental group
into a unitary group of rank $r=\rank F^{1,0}$:
$$\tilde\rho_{F}:~\pi_1\big(\ol B\setminus\Delta\big)\to U(r).$$
Note that the monodromy around $\Delta$ is unipotent, since $\bar f$ is semi-stable.
Hence $\tilde\rho_F$ actually factors through $\pi_1(\ol B)$:
$$
\xymatrix{
\pi_1\big(\ol B\setminus\Delta\big) \ar[rr]^-{\tilde\rho_F}\ar[dr]_-{i_*}&& U(r)\\
&\pi_1(\ol B)\ar[ur]_-{\rho_F}&
}$$

From Proposition \ref{theoremq_f=rankF10} it follows that the image of $\tilde\rho_F$ is finite.
Because $\tilde\rho_F$ factors through $\pi_1(\ol B)$
and $i_*$ is surjective, one gets that $\rho_F$ has also finite image.
It implies that $F_{\ol B}^{1,0}$ becomes trivial after a suitable finite \'etale base change.
From this together with Deligne's global invariant cycle theorem
(cf. \cite[\S\,4.1]{deligne71}) or Fujita's decomposition theorem (cf. \cite[Theorem\,3.1]{fujita78}),
it follows that $\rank F_{\ol B}^{1,0}=q_{\bar f}$ after passing a finite \'etale base change.
\end{proof}

\section{Examples and miscellaneous results}\label{sectionexample}
In this section, we construct some Shimura curves contained generically in the Torelli locus
in the low genus case,
and prove miscellaneous related results.

\begin{example}\label{exshimurag=3}
We construct a Shimura curve contained generically in the Torelli locus of hyperelliptic curves of genus $g=3$.
\vspace{0.2cm}

Let $C,\,H_{x_0}\subseteq X_0=\bbp^1\times \bbp^1$ be defined respectively by
$$1+(4t-2)x^2+x^4=0,\text{\quad and \quad} x=x_0,$$
where $t$ and $x$ are the coordinates of the first and second factor of $X_0$ respectively.
The projection of $C$ to the first factor $\bbp^1$ of $X_0$
branches exactly over three points, i.e., $\{0,1,\infty\}$.
Locally, it looks like the following.
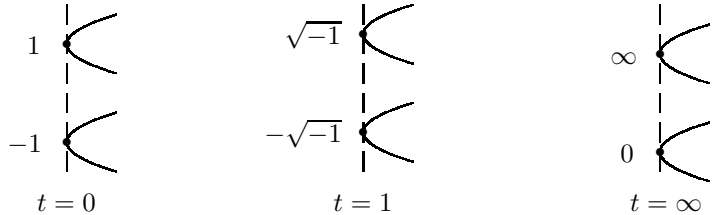
\begin{center}
\setlength{\unitlength}{1.3mm}
\begin{picture}(90,20)
\multiput(15,2)(0,3){6}{\line(0,1){2}}
\qbezier(20,2)(10,5)(20,8)
\qbezier(20,12)(10,15)(20,18)
\put(15,5){\circle*{0.8}}
\put(15,15){\circle*{0.8}}
\put(12,-2){$t=0$}
\put(9,4){$-1$}
\put(11,14){$1$}

\multiput(45,2)(0,3){6}{\line(0,1){2}}
\qbezier(50,3)(40,6)(50,9)
\qbezier(50,13)(40,16)(50,19)
\put(45,6){\circle*{0.8}}
\put(45,16){\circle*{0.8}}
\put(42,-2){$t=1$}
\put(35,5){$-\sqrt{-1}$}
\put(37,15){$\sqrt{-1}$}

\multiput(75,2)(0,3){6}{\line(0,1){2}}
\qbezier(80,1)(70,4)(80,7)
\qbezier(80,11)(70,14)(80,17)
\put(75,4){\circle*{0.8}}
\put(75,14){\circle*{0.8}}
\put(72,-2){$t=\infty$}
\put(71,3){$0$}
\put(70,13){$\infty$}
\end{picture}
\end{center}
\vspace{0.3cm}

Let $\varphi:\,\bbp^1 \to \bbp^1$ be the cyclic cover of degree $4$
defined by $t=(t')^4$,
totally ramified over $\{0,\infty\}$.
Let $X_1$ be the normalization of the fibre-product $X_0\times_{\bbp^1}\bbp^1$
and $R$ the inverse image of
$$C\cup H_{1}\cup H_{-1}\cup H_0\cup H_{\infty}\,.$$
Then $R$ is a double divisor,
i.e.,  we can construct a double cover $S_1 \to X_1$ branched exactly over $R$.
Let $S' \to X_r$ be the canonical resolution,
and $\bar f:\,\ol S \to \bbp^1$ the relatively minimal smooth model as follows.
\begin{center}\mbox{}
\xymatrix{
&S'\ar[d]\ar[ld]\ar[r] &X_r \ar[d] &&\\
\ol S  \ar[dr]_{\bar f} & S_1\ar[r]\ar[d]
& X_1 \ar[rr]^{\Phi}  \ar[dl]^{\tau_2} && X_0 \ar[d]^{\tau_1}  \\
& \bbp^1 \ar[rrr]^{\varphi} &&&\bbp^1
}
\end{center}

By the theory of double covers (cf. \cite[\S\,III.22]{bhpv}),
it is not difficult to show that $\bar f:\,\ol S \to \bbp^1$
is a semi-stable hyperelliptic family of genus $g=3$.
In fact, there are exactly $6$ singular fibres in the family $\bar f$,
i.e., those fibres $\Upsilon$ over $\Delta:=\varphi^{-1}(0\cup1\cup\infty)$.
More precisely, for any fibre $F$ over $\varphi^{-1}(1)$,
$F$ is an irreducible singular elliptic curve with exactly two nodes, hence
$\xi_0(F)=2$ and $\delta_1(F)=\xi_1(F)=0;$
for any fibre $F$ over $\varphi^{-1}(0\cup\infty)$,
$F$ is a chain of three smooth elliptic curves, hence
$\delta_1(F)=2$ and $\xi_0(F)=\xi_1(F)=0.$
So
$$\xi_0(\Upsilon)=8,\qquad \delta_1(\Upsilon)=4,\quad\text{and}\quad \xi_1(\Upsilon)=0.$$
Therefore by \eqref{formulaofdegomega},
\begin{equation*}
\deg \bar f_*\omega_{\ol S/\mathbb P^1}=2.
\end{equation*}

By definition, those fibres over $\Delta_{ct}:=\varphi^{-1}(0\cup\infty)$ have compact Jacobian,
while those over $\Delta_{nc}:=\varphi^{-1}(1)$ have non-compact Jacobian.
Hence the Jacobian of $\bar f$ admits exactly $\left|\Delta_{nc}\right|=4$ singular fibres over $\bbp^1$.
By \cite[\S\,7]{vz04}, the Higgs field of any semi-stable family of abelian varieties over $\bbp^1$
with exactly $4$ singular fibres must be maximal.
Hence the base $\mathbb P^1 \setminus \Delta_{nc}$
(more precisely, the image of $\mathbb P^1 \setminus \Delta_{nc}$ in $\cala_3$)
is a totally geodesic curve by Corollary \ref{prophyper} and Theorem \ref{theoremnumchar}.

It remains to show that $\mathbb P^1 \setminus \Delta_{nc}$ is in fact a Shimura curve.
We present here two ways.

(i).~
Let
$$\left(\bar f_*\omega_{\ol S/\mathbb P^1}\oplus R^1\bar f_*\mathcal O_{\ol S},~\theta_{\mathbb P^1}\right)
=\left(A_{\mathbb P^1}^{1,0}\oplus A_{\mathbb P^1}^{0,1},~\theta_{\mathbb P^1}\big|_{A_{\mathbb P^1}^{1,0}}\right)
\oplus \left(F_{\mathbb P^1}^{1,0}\oplus F_{\mathbb P^1}^{0,1},~ 0\right)$$
be the decomposition of the associated logarithmic Higgs bundle as in \eqref{decompB}.
Since the base is $\bbp^1$, one has
$q_{\bar f}=\rank F_{\mathbb P^1}^{1,0}$.
Hence by \eqref{arakelovB},
$$2=\deg \bar f_*\omega_{\ol S/\mathbb P^1}=  \deg A_{\mathbb P^1}^{1,0}
=\frac{g-q_{\bar f}}{2}\cdot\deg \left(\Omega^1_{\bbp^1}(\log\Delta_{nc})\right)=3-q_{\bar f},
~~\Longrightarrow~~ q_{\bar f}=1.$$
Since the base $\bbp^1$ is simply connected, by \cite[Theorem 0.2]{vz04},
the relative Jacobian of $\bar f$ is isogenous over $\bbp^1$ to a product
\begin{equation}\label{exampledec}
E\times_{\bbp^1}\mathcal E\times_{\bbp^1}\mathcal E,
\end{equation}
where $E$ is a constant elliptic curve, and $\mathcal E \to \bbp^1$ is
a family of semi-stable elliptic curves with maximal Higgs field.
To show that $\mathbb P^1 \setminus \Delta_{nc}$ is a Shimura curve, it suffices to prove that
the constant part $E$ has complex multiplication.

It is not difficult to see that our family is actually defined by
\begin{equation}\label{exampledefoff}
y^2=\big(1+(4(t')^4-2)x^2+x^4\big)\cdot(x^2-1)\cdot x.
\end{equation}
Let $E_0$ be a constant elliptic curve defined by
$u^4=v\cdot(v+1)^2.$
Then it is clear that $E_0$ has complex multiplication by $\mathbb Z\left[\sqrt{-1}\right]$.
Define a morphism from the family $\bar f$ to the constant family $E_0$ by
$$(u,\,v)=\psi(x,\,y)=\left(\frac{\sqrt{2}\cdot t'y}{(x^2-1)^2},\,\frac{4(t')^4x^2}{(x^2-1)^2}\right).$$
It can be checked easily that $\psi$ is well-defined.
Hence the Jacobian of $\bar f$ contains a constant part $E_0$.
Note that the constant part $E$ in the decomposition \eqref{exampledec} is unique up to isogenous,
and the property with complex multiplication is invariant under isogenous.
Therefore, the constant part $E\sim E_0$ has complex multiplication, and so
 $\mathbb P^1 \setminus \Delta_{nc}$ is a Shimura curve.

(ii).~
We prove that $\mathbb P^1 \setminus \Delta_{nc}$ is a Shimura curve
by showing that our family $\bar f$ is actually isomorphic to
a known special family constructed by Moonen and Oort \cite{mo11}.
Let
$$u=\frac{1+x^2}{1-x^2},\qquad v=\frac{2y}{(1-x^2)^2},\qquad  w=\left(\frac{1+x^2}{1-x^2}\right)^2.$$
Then by virtue of \eqref{exampledefoff}, we see that our family is isomorphic to
$$ {\mathcal U_{t'}}:\quad \left\{
\begin{aligned}
u^2&=w,\\
v^4&= \Big(2(t')^4w-2\big((t')^4-1\big)\Big)^2\cdot(w-1).
\end{aligned}\right.$$
Such a family can be viewed as a family of abelian covers of $\bbp^1$ branched exactly over
$4$ points with Galois group $\mathbb Z_2\times \mathbb Z_4$ and local monodromy of
the branched points being $\big((1,0), (1,1), (0,1), (0,2)\big)$.
And it is just the family (22) given in \cite[\S 6,\,{\sc Table\,2}]{mo11}, which is special.

We remark that by \cite{mo11}, we do not know whether the corresponding Shimura curve
is complete or not (i.e., whether $\Delta_{nc}=\emptyset$ or not).
Our concrete description shows that such a Shimura curve is a non-complete rational Shimura curve.
\end{example}

\begin{example}\label{exshimurag=4}
We construct a Shimura curve with strictly maximal Higgs field
contained generically in the Torelli locus of hyperelliptic curves of genus $g=4$.
\vspace{0.15cm}

The construction is similar to Example \ref{exshimurag=3}.

Let $C$, $H_{x_0}$ and $X_0$ be the same as those in Example \ref{exshimurag=3}.
Let $\varphi:\,\ol B \to \bbp^1$ be a cover of degree $8$,
ramified uniformly over $\{0,1,\infty\}$ with ramification indices equal to $4$.
It is easy to see that such a cover exists, and
$$g(\ol B)=2,\quad |\Delta|=6,\qquad \text{where $\Delta=\varphi^{-1}(0\cup1\cup\infty)$.}$$
Let $X_1$ be the normalization of $X_0\times_{\bbp^1}\ol B$
and $R$ the inverse image of
$$C\cup H_{1}\cup H_{-1}\cup H_{\sqrt{-1}}\cup H_{-\sqrt{-1}}\cup H_0\cup H_{\infty}\,.$$
Then $R$ is a double divisor,
i.e.,  we can construct a double cover $S_1 \to X_1$ branched exactly over $R$.
Let $\bar f:\,\ol S \to \ol B$ the relatively minimal smooth model as follows.
\begin{center}\mbox{}
\xymatrix{
\ol S  \ar[dr]_{\bar f}\ar@{<-->}[r] & S_1\ar[r]\ar[d]
& X_1 \ar[rr]^{\Phi}  \ar[dl]^{\tau_2} && X_0 \ar[d]^{\tau_1}  \\
& \ol B \ar[rrr]^{\varphi} &&&\bbp^1
}
\end{center}

By \cite[\S\,III.22]{bhpv}, one can show that $\bar f:\,\ol S \to \ol B$
is a semi-stable hyperelliptic family of genus $g=4$
with $6$ singular fibres,
i.e., those fibres $\Upsilon$ over $\Delta$.
More precisely, for any fibre $F\in \Upsilon$,
$F$ consists of two smooth elliptic curves $D_1$, $D_2$, and a smooth closed curve $\wt D$ of genus\,$2$,
such that $D_1$ does not intersect $D_2$, and $\wt D$ intersects each $D_i$
transversely in one point for $i=1,2$.
Hence
$$\delta_1(F)=2,\quad \text{~and~} \quad \delta_2(F)=\xi_0(F)=\xi_{1}(F)=0,
\qquad \forall~ F \in \Upsilon.$$
Since $|\Delta|=6$, by \eqref{formulanoether}, \eqref{formulaofdegomega} and \eqref{formulaofdelta_f} one gets
\begin{equation*}
\delta_{\bar f}=12,\qquad
\deg \bar f_*\omega_{\ol S/\ol B}=4,\qquad
\omega_{\ol S/\ol B}^2=36.
\end{equation*}

By definition, any singular fibre of $\bar f$ has a compact Jacobian,
so the Jacobian of $\bar f$ is a smooth family of abelian varieties of dimension $4$.
Let $A_{\ol B}^{1,0}\subseteq \bar f_*\omega_{\ol S/\ol B}$ be the ample part in the decomposition \eqref{decompB}.
Then according to the Arakelov inequality, we have
$$4=\deg \bar f_*\omega_{\ol S/\ol B}
\leq \frac{\rank A_{\ol B}^{1,0}}{2}\cdot \deg \Omega_{\ol B}=\rank A_{\ol B}^{1,0}
\leq \rank \bar f_*\omega_{\ol S/\ol B} =g=4.$$
Hence the Jacobian of $\bar f$ reaches the Arakelov bound with $\rank A^{1,0}=g$,
i.e., the Higgs field associated to $\bar f$ is strictly maximal.
Therefore $B=\ol B$ (more precisely, the image of $\ol B$ in $\cala_4$)
 is a Shimura curve of type II by Corollary \ref{prophyper} and
 Theorem \ref{theoremnumchar} since $\Delta_{nc}=\emptyset$.

We remark that in this example,
$$c_1^2\left(\Omega_{\ol S}^1(\log D)\right)=3c_2\left(\Omega_{\ol S}^1(\log D)\right)=72,$$
where $D$ is the union of those $12$ smooth disjoint elliptic curves contained in $\Upsilon$.
Hence $\ol S\setminus D$ is a ball quotient by \cite{kobayashi} or \cite{mok12}.
\end{example}

\begin{example}\label{exshimurag=3strict}
We construct a Shimura curve of type I contained generically in the Torelli locus $\calt_3$,
which can not be represented by a family $\bar f: \ol S\to \ol B$ of semi-stable curves of genus $g=3$
with strictly maximal Higgs field.\vspace{0.1cm}

In Section \ref{sectiontwotype}, we have constructed Shimura curves of type I in $\cala_g$ for each $g$.
Since $\calt_3=\cala_3$, there is Shimura curve of type I contained in $\calt_3$.
Note that any Hecke translate of a Shimura curve of type I is still a Shimura curve of type I,
and all the Hecke translates of such a curve are dense in $\cala_g$.
Therefore there must be a Shimura curve $C$ of type I contained generically in $\calt_3$,
and moreover we may find such a curve $C$ which is not contained in the Torelli locus of hyperelliptic
curves.
This implies that the family $\bar f$ of semi-stable curves representing $C$ by the Torelli morphism
is non-hyperelliptic.
Hence by Theorem \ref{theoremstrictarak3'} blow, $\bar f$ cannot have strictly maximal Higgs field.
\end{example}

The next proposition can be viewed as a byproduct of the proof of Theorem \ref{thmlower1}.
\begin{proposition}\label{theoremlower4}
Let $\bar f:\,\ol S \to \ol B$ be a non-hyperelliptic family of semi-stable curves.
Assume that $\bar f_*\omega_{\ol S/\ol B}$ is semi-stable. Then
\begin{equation}\label{eqntheoremlower4}
\omega_{\ol S/\ol B}^2 \geq \frac{5g-6}{g}\deg \bar f_*\omega_{\ol S/\ol B}
+\sum_{p\in \Delta_{ct}}\big(3l_h(F_p)+2l_1(F_p)-3\big).
\end{equation}
\end{proposition}
\begin{proof}
The proof is similar to that of Theorem \ref{thmlower1}.
We use the same notations.
The assumption that $\bar f_*\omega_{\ol S/\ol B}$ is semi-stable ensures that \eqref{deg1} is still valid.
Note that \eqref{cokernelofnu_p} and \eqref{S_pcokernelnu_p}
are true for any non-hyperelliptic semi-stable family.
Hence \eqref{eqntheoremlower4} follows.
\end{proof}

\begin{theorem}\label{theoremstrictarak3'}
Let $\bar f:\,\ol S \to \ol B$ be a non-isotrivial family of non-hyperelliptic semi-stable curves of genus $g\geq 3$.
Then $\bar f$ cannot have strictly maximal Higgs field, i.e., we have the following strict Arakelov inequality
\begin{equation}\label{eqnstrictarak3'}
\deg \bar f_*\omega_{\ol S/\ol B} < {g\over 2}\cdot\deg\Omega^1_{\ol B}(\log\Delta_{nc}).
\end{equation}
\end{theorem}
\begin{proof}
This is an improvement of Theorem \ref{theoremstrictarakfamily}.
It suffices to consider the cases $g=4$ or $3$ according to Theorem \ref{theoremstrictarakfamily}.

Consider first the case $g=4$. We argue by contradiction.
Suppose \eqref{eqnstrictarak3'} does not hold. Then by \eqref{eqnpfarak3}, one has
\begin{equation}\label{eqnpfarak3'1}
\deg \bar f_*\omega_{\ol S/\ol B} = \frac{g}{2}\cdot \deg\Omega^1_{\ol B}(\log\Delta_{nc})
=2\deg\Omega^1_{\ol B}(\log\Delta_{nc}).
\end{equation}
It follows that the associated Higgs field $\theta_{\ol B}$ is strictly maximal,
and both \eqref{eqnupper2} and \eqref{eqnlower2} are equalities.
By Remarks \ref{miyaoka12}\,(i), one has $\Delta_{nc}=\emptyset$ and $\Delta_{ct}=\Delta_{ct,ub}$,
where $\Delta_{ct,ub}$ is defined in Theorem \ref{theoremupperlinshi}.
Hence by \eqref{eqnupperlinshi} together with \eqref{eqnpfarak3'1}, one gets
\begin{equation}\label{eqnpfarak3'2}
\omega_{\ol S/\ol B}^2 \leq \frac{4(g-1)}{g}\cdot \deg \bar f_*\omega_{\ol S/\ol B}
+\sum_{p\in \Delta_{ct}}\big(3l_h(F_p)+2l_1(F_p)-3\big).
\end{equation}
As $\theta_{\ol B}$ is strictly maximal,
$\bar f_*\omega_{\ol S/\ol B}=E^{1,0}_{\ol B}$ is poly-stable by \cite[Proposition\,1.2]{vz04}.
In particular, $\bar f_*\omega_{\ol S/\ol B}$ is semi-stable.
Thus the condition of Proposition \ref{theoremlower4} is satisfied.
It follows from \eqref{eqntheoremlower4} and \eqref{eqnpfarak3'2} that
$\deg f_* \omega_{\ol S/\ol B}=0,$ which is impossible.

Now we consider the case $g=3$.
Let $\ol \calh_3 \subseteq \ol \calm_3$ be the locus of stable hyperelliptic curves, which is a divisor.
Let $[H]$ be its divisor class.
Let $$\varphi:~\ol B \lra \ol\calm_3,\qquad\quad p \mapsto [F_p],$$ be the moduli map of $\bar f$,
and $h=\deg \varphi^*([H])$.
Then by \cite[\S 3-H]{harrismorrison}, we have
\begin{eqnarray}
\deg \bar f_*\omega_{\ol S/\ol B} &=& \frac19h+\frac19\delta_0(\Upsilon)+\frac13\delta_1(\Upsilon),\label{eqnpfarak3'3}\\[0.1cm]
\omega_{\ol S/\ol B}^2 &=& \frac{4}{3}h+\frac13\delta_0(\Upsilon)+3\delta_1(\Upsilon).\label{eqnpfarak3'4}
\end{eqnarray}
Note that $\delta_h(\Upsilon)=0$ for $g=3$.
It follows from \eqref{eqnpfarak3'3}, \eqref{eqnpfarak3'4} and \eqref{eqnupper2} that
$$\begin{aligned}
\deg \bar f_*\omega_{\ol S/\ol B}&\,=\,
\frac{3}{8}\omega_{\ol S/\ol B}^2-\left(\frac{7}{18}h+\frac{1}{72}\delta_0(\Upsilon)+\frac{19}{24}\delta_1(\Upsilon)\right)\\
&\,\leq\,\frac{3}{2}\deg\Omega^1_{\ol B}(\log\Delta_{nc})
-\left(\frac{7}{18}h+\frac{1}{72}\delta_0(\Upsilon)+\frac{1}{24}\delta_1(\Upsilon)\right).
\end{aligned}$$
Since $\bar f$ is non-isotrivial, one of $\big\{h,\delta_0(\Upsilon),\delta_1(\Upsilon)\big\}$
must be positive due to \eqref{eqnpfarak3'3}.
Therefore, \eqref{eqnstrictarak3'} follows immediately for $g=3$.
The proof is complete.
\end{proof}

\begin{remarks}
(i). Together with \cite{grumoeller, vz04},
one can show that a family $\bar f$ of semi-stable genus-$g$ curves
can have strictly maximal Higgs field only when $\bar f$ is hyperelliptic and $g=2$ or $4$.
Indeed, by Theorems \ref{mainthm2} and \ref{theoremstrictarak3'}, It suffices to exclude the case
when $\bar f$ is hyperelliptic of genus $g=3$ with strictly maximal Higgs field.
It { this were the case},
it is proven in \cite{grumoeller} that $\Delta_{nc}\neq \emptyset$;
by \cite[Theorem\,0.5]{vz04}, it follows
that the relative Jacobian of $\bar f$ is isogenous a smooth family of abelian varieties
over a Shimura curve of Mumford type up to a finite \'etale base change,
which implies particularly that $g$ is even (cf Section \ref{sectiontwotype}), { which is absurd.}

(ii). We refer to \cite[\S\,3]{tu07} for such examples of $g=2$ and Example \ref{exshimurag=4} for an example of $g=4$.
\end{remarks}
\vspace{0.2cm}
\appendix
\stepcounter{section}
\phantomsection
\begin{center}
{\sc Appendix: involutions on the universal family of curves}
\end{center}
\addcontentsline{toc}{section}{Appendix: involutions on the universal family of curves}
\renewcommand{\theequation}{\Alph{section}-\arabic{equation}}

The purpose of the appendix is to prove the existence of involutions on the universal family of curves.
This is supposed to be a known result, { and we include it for readers' convenience.}

Recall that there is a universe family of stable curves over
the partial compactification $\calm_g^{ct}=\calm^{ct}_{g,[n]}$ of the moduli space of smooth projective genus-$g$ curves
with level-$n$ structure by adding those stable curves with  compact Jacobians as in \eqref{eqnuniverseonM}.
\setcounter{figure}{0}
\renewcommand{\thefigure}{\Alph{section}.\arabic{figure}\,}
\begin{lemma}\label{existprop}
There exists an involution $\sigma_g$ {\rm(}resp. $\tau_g${\rm)}
on $\cals_g$ {\rm(}resp. $\calm_g^{ct})$
such that the following diagram commutes.
$$\xymatrix{
\cals_g \ar[d]_{\mathfrak f}\ar[rr]^{\sigma_g} && \cals_g \ar[d]^{\mathfrak f}\\
\calm_g^{ct} \ar[rr]^{\tau_g}  && \calm_g^{ct}
}
$$
Moreover, the fixed locus of $\tau_g$ is exactly the hyperelliptic locus $\calh_g^{ct}\subseteq \calm_g^{ct}$,
and for $p\in \calh_g^{ct}$, $\sigma_g|_{C_p}:\,C_p \to C_p$ is the hyperelliptic involution of $C_p$,
where $C_p \subseteq \cals_g$ is the hyperelliptic curve over $p$.
\end{lemma}
\begin{proof}
According to \cite{os79}, there exists an involution $\tau_g$ on $\calm_{g}$.
But it is easy to see that it can be extended to $\calm_g^{ct}$,
which is defined by
\begin{equation}\label{defoftau}
\tau_g\big([C,\,\alpha]\big)=[C,\,-\alpha],
\end{equation}
where $C$ is a genus-$g$ stable curve of compact type and $\alpha$ is a level-$n$ structure of $C$.
Note that for a hyperelliptic curve $C$, the hyperelliptic involution gives an isomorphism
between $[C,\,\alpha]$ and $[C,\,-\alpha]$.
It follows that (cf. \cite{os79}) the fixed point of the involution $\tau_g$
is exactly the hyperelliptic locus $\calh_g^{ct} \subseteq \calm_g^{ct}$.

Let $\calm_{g,1}=\calm_{g,[n],1}$ (resp. $\calm^{ct}_{g,1}=\calm^{ct}_{g,[n],1}$) be
the moduli space of genus-$g$ smooth curves (resp. stable curves of compact type)
with level-$n$ structure and one marked point,
and ${\rm Fog}^{\rm o}_g:\,\calm_{g,1} \to \calm_{g}$
\big(resp. ${\rm Fog}_g:\,\calm^{ct}_{g,1} \to \calm_g^{ct}$\big)
the natural map by forgetting the marked point.
It is easy to see that ${\rm Fog}^{\rm o}_g$ (resp. ${\rm Fog}_g$) factors through
$\cals_g^{\rm o}=\mathfrak f^{-1}\left(\calm_{g}\right)$ \big(resp. $\cals_g$\big)
$$\xymatrix@M=0.15cm{
\calm_{g,1} \ar@/_9mm/"3,1"_{{\rm Fog}^{\rm o}_g} \ar[d]^{\rho^{\rm o}_g} \ar@{^(->}[rr]
       && \calm^{ct}_{g,1} \ar[d]_{\rho_g} \ar@/^9mm/"3,3"^{{\rm Fog}_g} \\
\cals_g^{\rm o} \ar[d]^{\mathfrak f^{\rm o}} \ar@{^(->}[rr] && \cals_g \ar[d]_{\mathfrak f}\\
\calm_{g} \ar@{^(->}[rr] && \calm_g^{ct}
}$$
Note that $\rho^{\rm o}_g$ is actually an isomorphism;
indeed, for any $p\in \calm_{g}$, the fibres over $p$ in both $\cals_g^{\rm o}$ and $\calm_{g,1}$
is isomorphic to $C_p$, where $C_p$ is the smooth curve corresponding to $p$.
However, it is not the case for $\rho_g$.
By the definition of $\rho_g$, it just maps a stable curve with one marked point to the curve.
But a stable curve with one marked point may become a non-stable curve when forgetting the marked point.
Hence if $p\in \calm_g^{ct} \setminus \calm_{g}$,
then $C_{p,1} \to C_{p}$ is just the contraction of non-stable components contained in $C_{p,1}$,
where $C_{p,1}$ (resp. $C_p$) be the fibre over $p$ in $\calm_{g,1}$ (resp. $\cals_g^{\rm o}$).

We define an involution $\sigma_{g,1}$ on $\calm^{ct}_{g,1}$ by
$$\sigma_{g,1}\big([C,\,\alpha,\,x]\big)=[C,\,-\alpha,\,x],$$
where $C$ is a genus-$g$ stable curve of compact type, $\alpha$ is a level-$n$ structure of $C$,
and $x\in C$ is a marked point.
Then it is clear that
\begin{equation}\label{existlinshi1}
{\rm Fog}_g\circ\sigma_{g,1}=\tau_g\circ{\rm Fog}_g.
\end{equation}
Define an involution $\sigma_g$ on $\cals_g$ by
$$\sigma_g(x)=\rho_g\circ\sigma_{g,1}\circ\rho_g^{-1}(x).$$
According to the description of $\rho_g$ above,
$\sigma_g$ is well-defined.
And the diagram in the lemma commutes by \eqref{existlinshi1}.

If $p\in \calh_g^{ct}$, then the fibre $C_{p,1}$ over $p$ in $\calm^{ct}_{g,1}$
is a stable hyperelliptic curve with one marked point.
The hyperelliptic involution $\iota$ induces an isomorphism
$$[C_{p,1},\,\alpha,\,x] ~\overset{\cong}\lra~ [C_{p,1},\,-\alpha,\,\iota(x)].$$
Hence
$$\sigma_{g,1}\big([C_{p,1},\,\alpha,\,x]\big)=[C_{p,1},\,-\alpha,\,x] = [C_{p,1},\,\alpha,\,\iota(x)].$$
It implies that $\sigma_{g,1}\big|_{C_{p.1}}:\,C_{p,1} \to C_{p,1}$ is just the hyperelliptic involution of $C_{p,1}$.
Hence $\sigma_g|_{C_p}:\,C_p \to C_p$ is the hyperelliptic involution of $C_p$.
\end{proof}

\vspace{0.3cm}
\noindent{\bf Acknowledgements.}
We would like to thank Chris Peters,  Guitang Lan and Jinxing Xu
for discussions on the topic related to global invariant
cycles with locally constant coefficients.
Especially, the proof of Lemma \ref{lemmapeters} comes from a discussion with Chris Peters.
We would also like to thank Shengli Tan and Hao Sun for discussing with us on
Miyaoka-Yau's inequality and the slope inequality,
and Alessandro Ghigi and Stefan M\"uller-Stach for their interests.
We thank Ke Chen for discussing with us on Shimura varieties and the formulation of the paper.
We are grateful to Yanhong Yang for her interests,
careful reading and valuable suggestions..

\end{document}